\documentclass[12pt]{article}

\pdfoutput=1

\usepackage{amsfonts}
\usepackage{mathtools}
\usepackage{fullpage}
\usepackage{amsmath}
\usepackage{amsthm}
\usepackage{amssymb}
\usepackage{xcolor}   
\usepackage{enumerate}
\usepackage{euscript}
\usepackage[hyphens]{url}
\usepackage{graphicx}
\usepackage{listings}
\usepackage{esint}
\usepackage{dsfont}
\usepackage{pifont}
\usepackage[english]{babel}
\usepackage[space]{grffile}
\usepackage{hyperref}
\usepackage{tikz}
\usepackage[style=alphabetic, backend=biber]{biblatex}

\renewbibmacro{in:}{}
\AtEveryBibitem{\clearfield{month}}
\usetikzlibrary{matrix,calc}

\theoremstyle{definition}

\newtheorem{definition}{Definition}[section]
\newtheorem{example}[definition]{Example}

\theoremstyle{plain}

\newtheorem{corollary}[definition]{Corollary}
\newtheorem{lemma}[definition]{Lemma}
\newtheorem{theorem}[definition]{Theorem}
\newtheorem{conjecture}[definition]{Conjecture}

\theoremstyle{remark}

\newtheorem{remark}[definition]{Remark}

\begin{document}

\title{Cohomogeneity one 4-dimensional gradient Ricci solitons}
\author{Patrick Donovan}
\date{\today}
\maketitle

\begin{center}
    \textbf{Abstract:}
\end{center}
Simply-connected four-dimensional gradient Ricci solitons that are invariant under a compact cohomogeneity one group action have been studied extensively. However, the special case where the group is $\mathrm{SU}(2)$ (the smallest possible example) has received comparatively little attention. The purpose of this article is to give a comprehensive study of simply-connected $\mathrm{SU}(2)$-invariant expanding and shrinking cohomogeneity one gradient Ricci solitons. The first result is the construction of new $3$-parameter families of complete $\mathrm{SU}(2)$-invariant asymptotically conical expanding gradient Ricci solitons. New shrinking K\"ahler $\mathrm{U}(2)$-invariant gradient Ricci solitons in dimension 4 with orbifold singularities are also constructed, leading to a classification of such metrics when the base space of the orbifold is a simply-connected smooth manifold. Finally, we highlight numerical evidence that all the compact cohomogeneity one shrinking gradient Ricci solitons are known.

\section{Introduction}
A Riemannian metric $g$ on a smooth manifold $M$ is called a Ricci soliton if there exists a smooth vector field $X \in \mathfrak{X}(M)$ and a constant $\lambda \in \mathbb{R}$ such that
\begin{equation}
    \label{eq_soliton}
    \text{Ric}(g) + \frac{1}{2}L_{X}g = \lambda g,
\end{equation}
where $L$ is the Lie derivative. A Ricci soliton  with $\lambda = 0$ in (\refeq{eq_soliton}) is a called a steady soliton and is a solution to the Ricci flow that evolves via pullbacks of diffeomorphisms in the one-parameter family generated by $X$, i.e., $g_t = \phi_t^*(g)$ where $\phi_0 = \text{id}_M$ and $X = \frac{d\phi_t}{dt}$. If instead $\lambda < 0$ ($\lambda > 0$), then the soliton is said to be expanding (shrinking) as the corresponding solution to the Ricci flow with initial metric $g$, expands (shrinks) uniformly while also evolving via a diffeomorphism in the direction of $X$. Ricci solitons may be considered the \textit{geometric} fixed points of the Ricci flow as they form constant solutions to Ricci flow up to rescaling and diffeomorphism. 

Ricci solitons are also generalizations of Einstein manifolds; Einstein metrics trivially solve (\refeq{eq_soliton}) for any Killing field $X$. Einstein manifolds are named as such since a Lorenzian $4$-manifold with constant Ricci curvature is a solution to Einstein's field equations in the special case of no mass. Einstein metrics are also interesting as geometries with constant Ricci curvature, and may be considered candidates for a canonical metric on a smooth manifold. We produce a family of negatively curved Einstein manifolds in this work --- see Theorem \refeq{thm_expanders}.

The Ricci flow is a weakly parabolic heat type flow and the reaction terms in the Ricci flow equation allow curvature to concentrate until singularities arise (see \cite{Hamilton1986} for example). Shrinking gradient Ricci solitons are important in the study of Ricci flow as they commonly appear as singularity models: limits of parabolic dilations around a singularity of a solution to the Ricci flow. In Conjecture \refeq{con_compact_solitons}, we postulate that the shrinking compact gradient Ricci solitons --- solitons whose vector field is the gradient of a smooth function $u \in C^\infty (M)$ called the soliton potential --- are all known if they are invariant under a large symmetry group.

Ricci flow can be resolved through singularities in dimension 3, but it is a central goal in geometric analysis to be able to continue Ricci flow past any singularity in higher dimensions. In dimension 4, a new generic singularity model arises --- conical singularities. Expanding asymptotically conical gradient Ricci solitons can been used to flow past finite time conical singularities appearing in Ricci flow in dimension at least $4$ (see \cite{GianniotisSchulze2018,AngenentKnopf2022,FIK2003}). We show that in the presence of strong symmetries, expanding asymptotically conical gradient Ricci solitons in dimension 4 are abundant --- again see Theorem \refeq{thm_expanders}.

As with Einstein metrics, one of the most prolific methods used to find Ricci solitons is to apply the ansatz of invariance under some large symmetry group. As Ricci solitons are \textit{geometric} fixed points of the Ricci flow, which has strong smoothing properties, one should expect Ricci solitons to have an abundance of symmetry, particularly in low dimensions. For example, expanding asymptotically conical gradient Ricci solitons with positive sectional curvature were shown to have rotational symmetry in \cite{Chodosh2014}. When solving (\refeq{eq_soliton}) in the class of metrics which are invariant under a given compact group action, we may assume that the vector field $X$ is also invariant by the averaging procedure in \cite{DancerWang2011}. Therefore, with this additional assumption, all the terms in the Ricci soliton equation are invariant under the isometry group of $g$. In particular, finding Ricci solitons invariant under the action of a compact group reduces to solving (\refeq{eq_soliton}) over the orbit space of the action, with boundary conditions to ensure smoothness. Constructions without symmetry assumptions are rare in low dimensions. The most notable exceptions in dimension $4$ are found in \cite{Deruelle2016} and \cite{BamlerChen2025}, which prove existence of expanding gradient Ricci solitons asymptotic to a prescribed cone. The former uses a continuity method relying on positive curvature, while the latter formulates a new degree theory for $4$-dimensional asymptotically conical expanding gradient Ricci solitons.

The strongest symmetries one can enforce are homogenous: where $M$ is $G$-equivariantly diffeomorphic to $G/H$ for a Lie group $G$ acting on $M$ with a closed stabilizer subgroup $H = G_p$ for some $p \in M$. As the orbit space is a point, finding complete homogeneous gradient Ricci solitons is an algebraic problem. The complete list of homogeneous gradient Ricci solitons on $4$-dimensional simply-connected manifolds is as follows:
\begin{itemize}
    \item the shrinkers are the obvious metrics on $\mathbb{R}^4$, $\mathbb{S}^2 \times \mathbb{R}^2$, $\mathbb{S}^3 \times \mathbb{R}^1$, $\mathbb{S}^4$, $\mathbb{S}^2 \times \mathbb{S}^2$ and $\mathbb{CP}^2$; 
    \item the translator is the flat Euclidean metric; and
    \item the expanders are the obvious metrics on $\mathbb{R}^4$, $H^2 \times \mathbb{R}^2$, $H^3 \times \mathbb{R}^1$ $H^4$, $H^2 \times H^2$ and $\mathbb{C}H^2$,
\end{itemize}
where $H^n$ is hyperbolic space of dimension $n$ and $\mathbb{C}H^n$ is complex hyperbolic space of real dimension $2n$. This classification is completed via the work of \cite{PetersenWylie2007} and the classification of homogenous $4-$dimensional Einstein metrics in \cite{Jensen1969}. In fact, non-trivial homogenous Ricci solitons are necessarily expanding and not gradient. A classification of these metrics in dimension 4 was completed in \cite{ArroyoLafuente2014}. 

We are therefore interested in Ricci solitons that need not have such strong symmetries. A Riemannian manifold $(M,g)$ is cohomogeneity one if it admits an invariant proper action by a Lie group $G$ such that the orbit space $M/G$ is one dimensional. The maximal symmetry possible for a non-trivial gradient Ricci solitons is cohomogeneity-one. In dimension 2 compact gradient Ricci solitons have constant curvature. More generally, $2$-dimensional gradient Ricci solitons are necessarily rotationally symmetric, allowing the remaining gradient Ricci solitons to be classified in \cite{BernsteinMettler2015}. In dimension 3, the only shrinking gradient Ricci solitons are the Gaussian shrinker on $\mathbb{R}^3$, the round sphere $\mathbb{S}^3$, the round cylinder $\mathbb{S}^2 \times \mathbb{R}$ or a quotient of one of these. Bryant constructed a non-flat steady gradient Ricci soliton in dimension 3, which was later shown to be unique in the class of non-collapsed steady gradient Ricci solitons by Brendle \cite{Brendle2013}. Byrant also constructed a one-parameter family of expanding gradient Ricci solitons on $\mathbb{R}^3$. 
Hence, the $4$-dimensional setting is what we focus on. A cohomogeneity one ansatz reduces (\refeq{eq_soliton}) to a system of ODEs which enables the use of standard ODE techniques. We will only consider $4$-dimensional simply-connected smooth manifolds that admit a cohomogeneity one action by a smooth compact Lie group $G$.

Cohomogeneity one actions, alongside assumptions of further structure, such as K\"ahlerity, have been used to produce multiple families of expanding Ricci solitons. In particular, we revisit the construction of both some $\mathrm{U}(2)$-invariant Einstein manifolds with negative Einstein constant (see \cite{PetersenZhu1995}), and the one-parameter family of $\mathrm{SO}(4)$-invariant expanders found by Bamler (see \cite{CCG2007}). There is an extensive history of gradient K\"ahler Ricci solitons: one-parameter families of $\mathrm{U}(2)$-invariant expanding gradient Ricci solitons were found by Cao, Chave and Valent on $\mathbb{C}^2$ \cite{Cao1997,ChaveValent1996} and by Chave and Valent and by Feldman, Ilmanen and Knopf on $\mathcal{O}(-n)$ for $n \geq 3$ in \cite{FIK2003,ChaveValent1996}. The space $\mathcal{O}(-n)$ is a complex line bundle of $\mathbb{CP}^1$ with twisting number $-n$. Specifically, $\mathcal{O}(-1) = \text{Bl}_1(\mathbb{C}^2)$ is the blow up of $\mathbb{C}^2$ at a point and $\mathcal{O}(-2)$ is the tangent bundle of $\mathbb{CP}^1$. See \cite{Appleton2022} for more details on $\mathcal{O}(-n)$. These families are also included in the solitons listed in Theorem \refeq{thm_expanders} (see Table \refeq{tab_theorem}). 

\renewcommand{\arraystretch}{1.35}
\begin{table}[h!]
\begin{center}
    \begin{tabular}{||c||c|c|c|c||} 
        \hline
        Manifold & One parameter family of Solitons & Einstein & K\"ahler & Source \\ [0.5ex] 
        \hline\hline
        $\mathbb{R}^4$, $\mathcal{O}(-n)$, $n \geq 1$ & $\mathrm{U}(2)$-invariant & \ding{51} & \ding{55} & \cite{PagePope1987,PetersenZhu1995}\\
        \hline
        $\mathbb{R}^4$ & Bamler's $\mathrm{SO}(4)$-invariant & \ding{55} & \ding{55} & \cite{CCG2007} \\
        \hline
        $\mathbb{R}^4$ & $\mathrm{U}(2)$-invariant & \ding{55} & \ding{51} & \cite{Cao1997} \\
        \hline
        $\mathcal{O}(-n)$, $n \geq 3$ & $\mathrm{U}(2)$-invariant & \ding{55} & \ding{51} & \cite{FIK2003}\\
        \hline
    \end{tabular}
    \caption{Notable Einstein and non-Einstein expanding solitons included in the solitons of Theorem \refeq{thm_expanders}}\label{tab_theorem}
\end{center}
\end{table}
\renewcommand{\arraystretch}{1}

\pagebreak

In this work, we are primarily interested in $\mathrm{SU}(2)$-invariant metrics. $\mathrm{SU}(2)$, and its quotient $\mathrm{SO}(3)$, are the smallest compact Lie groups that can be found in cohomogeneity one group actions on $4$-manifolds. Consequently, $\mathrm{SU}(2)$ actions are less studied when searching for Ricci solitons as it is often more challenging than other cases (notably when compared to $\mathrm{U}(2)$, $\mathrm{SO}(4)$ and $\mathrm{SO}(3) \times \mathrm{SO}(2)$ actions). 

The study of $\mathrm{SU}(2)$ actions leads us to find the $3$-parameter families of cohomogeneity-one expanding asymptotically conical (see Definition 1.2 of \cite{Deruelle2016}) gradient Ricci solitons in dimension 4. The previous largest families of cohomogeneity-one gradient Ricci solitons were both $2$-parameter families \cite{GastelKronz2004,NienhausWink2024}. The families we construct lie on $\mathbb{R}^4$ and $\mathcal{O}(-n)$ for any $n \in \mathbb{Z}_{>0}$.

\begin{theorem}\label{thm_expanders}
    There exists a $2$-parameter family of complete $\mathrm{U}(2)$-invariant expanding gradient Ricci solitons containing a $1$-parameter family of Einstein metrics on $\mathcal{O}(-n)$ for all $n\in \mathbb{Z}_{>0}$. Moreover, there exists a $3$-parameter family of complete $\mathrm{SU}(2)$-invariant expanding gradient Ricci solitons containing a $2$-parameter family of Einstein metrics on $\mathbb{R}^4$ and $\mathcal{O}(-n)$ if $n=1,2$ or $4$. The Einstein metrics are categorized by asymptotically hyperbolic volume growth and all non-Einstein solitons are asymptotically conical.
\end{theorem}

\begin{remark}
    The non-Einstein and Einstein families of Theorem \refeq{thm_expanders} are each open sets in their parameter space of (not necessarily complete) $\mathrm{SU}(2)$-invariant non-Einstein Ricci solitons or $\mathrm{SU}(2)$-invariant Einstein metrics on $\mathbb{R}^4$ and $\mathcal{O}(-n)$. 
\end{remark}

\begin{remark}
    The parameter space of $\mathrm{SU}(2)$-invariant (shrinking and expanding) solitons is three dimensional on $\mathbb{R}^4$ and $\mathcal{O}(-n)$ if $n \in {1,2,4}$. $\mathrm{SU}(2)$-invariant metrics on $\mathcal{O}(-n)$ for $n \not \in \{1,2,4\}$ are necessarily $\mathrm{U}(2)$-invariant and hence the parameter space of invariant solitons is $2$-dimensional. Furthermore, the parameter space of Einstein metrics within gradient Ricci solitons is always codimension 1.
\end{remark}

\begin{remark}
    As well as the metrics from Table \refeq{tab_theorem}, there are also notable individual metrics such as the hyperbolic metric and the Gaussian expander that are included in Theorem \refeq{thm_expanders}. On the other hand, the complex hyperbolic metric on $\mathbb{C}H^2$ and the $\mathrm{U}(2)$-invariant K\"ahler-Einstein metrics on $\mathcal{O}(-n)$ for $n\geq 3$ \cite{GibbonsPope1979,Pedersen1985} are found on the boundary of this family. Further perturbation results may be used to construct one-parameter families of complete Einstein metrics that also lie within the boundary of the families constructed here. Finally, we also note that the one-parameter family of $\mathrm{SU}(2)$-invariant K\"ahler Einstein metrics on $\mathcal{O}(-2)$ found in \cite{DancerStrachan1994} are numerically shown to be on the boundary our space of metrics too.
\end{remark}

There are many other examples of cohomogeneity one expanding solitons as in \cite{Thompson2024,Wink2019,Wink2024} and many others we do not list here. Expanding gradient Ricci solitons may also be found by considering multiply warped metrics --- reducing (\refeq{eq_soliton}) to an ODE without an overt use of symmetry groups. Multiply warped metrics over the product of a ray, a round sphere and a (product of) Einstein manifold(s) have yielded many examples as in \cite{GastelKronz2004}, which was generalized in \cite{BDGW2015} and further again in \cite{NienhausWink2024}. These multiply warped expanders are typically cohomogeneity one, particularly in low dimensions, and they include a comprehensive analysis of the $\mathrm{SO}(3)\times \mathrm{SO}(2)$-invariant expanders on $\mathbb{S}^2 \times \mathbb{R}^2$ and $\mathbb{S}^1 \times \mathbb{R}^3$. Theorem \refeq{thm_expanders} treats expanders for the remaining irreducible effective compact actions on simply-connected smooth manifolds with a non-principal orbit.

We then move into the exploration of non-expanding solitons. We note that the steady soliton case is well-studied, particularly as the system of ODEs in the cohomogeneity one case reduces a dimension as in \cite{Buttsworth2025}. A list of known cohomogeneity-one steady gradient Ricci solitons includes:
\begin{itemize}
    \item The $\mathrm{SO}(3) \times \mathrm{SO}(2)$-invariant metrics on $\mathbb{S}^2 \times \mathbb{R}^2$ \cite{Ivey1994};
    \item the Euclidean metric, the $\mathrm{SO}(4)$-invariant Bryant soliton, the $\mathrm{U}(2)$-invariant Taub-NUT metric and the $\mathrm{U}(2)$-invariant K\"ahler soliton on $\mathbb{R}^4$ \cite{Cao1996};
    \item The $\mathrm{U}(2)$-invariant Taub-bolt metric on $\mathcal{O}(-1)$;
    \item The $\mathrm{U}(2)$-invariant Hyper-K\"ahler Eguchi-Hanson metric and the $\mathrm{U}(2)$-invariant K\"ahler metric on $\mathcal{O}(-2)$ \cite{Cao1996,ChaveValent1996};
    \item Appleton's non-collapsed $\mathrm{U}(2)$-invariant metric on $\mathcal{O}(-n)$ for each $n\geq 3$ \cite{Appleton2022}; and 
    \item A one-parameter family of $\mathrm{SU}(2)$-invariant metrics found by perturbing the Appleton soliton on $\mathcal{O}(-4)$ \cite{Buttsworth2025}.
\end{itemize}
The perturbative result generating a one-parameter family of $\mathrm{SU}(2)$-invariant metrics on $\mathcal{O}(-4)$ inspires our study of expanders.

This leaves the shrinking case --- we exclusively study the compact case. It should be noted that, in the 4 dimensional K\"ahler setting, all shrinking gradient Ricci solitons are known, only recently classified in \cite{BCCD2024} where a Ricci soliton was constructed on $\text{Bl}_1(\mathbb{C} \times \mathbb{CP}^1)$ with the use of a cohomogeneity two toric symmetry. 

Compact 4-dimensional solitons are challenging to produce; cohomogeneity one assumptions have only yielded two new simply-connected constructions --- both of which had additional structure. First, Page was able to find a new $\mathrm{U}(2)$-invariant Einstein metric with positive scalar curvature on $\text{Bl}_1(\mathbb{CP}^2) = \mathbb{CP}^2 \# \overline{\mathbb{CP}^2}$ \cite{Page1978}. Koiso and later Cao were also able to find a gradient shrinking Ricci soliton on $\text{Bl}_1(\mathbb{CP}^2)$ in the K\"ahler setting, also with $\mathrm{U}(2)$-invariance \cite{Koiso1990,Cao1996}.

All cohomogeneity one proper actions on a compact simply-connected $4$-manifolds are known, described in \cite{Hoelscher2008} for example, and often reduce to an $\mathrm{SU}(2)$-action (see Table \refeq{tab_compact}). One may hope that, as in the K\"ahler setting, a classification of compact cohomogeneity one shrinkers would be attainable, however the Ricci solitons invariant under such actions are rare and few results have alluded to their uniqueness. The cohomogeneity one actions on simply-connected compact $4$-manifolds that do not reduce to an $\mathrm{SU}(2)$-action are an extension of a cohomogeneity two toric action, on which the most recent relevant uniqueness result applies: simply-connected cohomogeneity two toric Einstein manifolds with non-negative sectional curvature are exactly the Fubini--Study metric, the round metric or the direct product of round metrics (of the same radii) proven in \cite{Liu2023}. We conjecture that the simply-connected $4$-dimensional compact solitons are all known.

\begin{conjecture}\label{con_compact_solitons}
    If a compact simply-connected $4$-dimensional Ricci soliton admits an invariant cohomogeneity one action, then it is isometric to a rescaling of one of the following:
    \begin{itemize}
        \item the Koiso--Cao metric;
        \item the Page metric;
        \item the Fubini--Study metric on $\mathbb{CP}^2$;
        \item the round metric on $\mathbb{S}^4$; or 
        \item the direct product of round metrics on $\mathbb{S}^2 \times \mathbb{S}^2$ with the same radii.
    \end{itemize}
\end{conjecture}

In the process of justifying Conjecture \refeq{con_compact_solitons}, we construct and classify the $\mathrm{U}(2)$-invariant gradient K\"ahler Shrinking Ricci soliton orbifolds whose base space is a smooth manifold. 

\begin{theorem}
    \label{thm_Kahlerclassification}
    Let $M$ be a $4$-dimensional simply-connected smooth manifold. If $g$ is a complete shrinking gradient K\"ahler Ricci soliton on $M$, invariant under cohomogeneity one action by a subgroup of $I(M)$ isomorphic to a finite quotient of $\mathrm{U}(2)$ such that the subaction by a finite quotient of $\mathrm{SU}(2)$ is also cohomogeneity one, potentially with orbifold singularities along non-principal orbits, then it is isometric to a rescaling of one of the following:
    \begin{enumerate}
        \item The shrinking Gaussian soliton on $\mathbb{C}^2$;
        \item The $\mathrm{U}(2)$-invariant shrinking gradient K\"ahler Ricci soliton on $\textup{Bl}_1(\mathbb{CP}^2)$ described by a positive odd integer $n$ and distinct ordered cone angles $2\pi/q_1<4\pi/n$ and $2\pi/q_2$;
        \item The $\mathrm{U}(2)$-invariant shrinking gradient K\"ahler Ricci soliton on $\mathbb{S}^2 \times \mathbb{S}^2$ described by a positive even integer $n$ and distinct ordered cone angles $2\pi/q_1< 4 \pi/n$ and $2\pi/q_2$;
        \item The $\mathrm{U}(2)$-invariant shrinking gradient K\"ahler Ricci soliton with cone angle $2\pi/q$ on $\mathbb{CP}^2$; or
        \item The $\mathrm{U}(2)$-invariant shrinking gradient K\"ahler Ricci soliton with cone angle ${2\pi}/{q}< {4 \pi}/{n}$ on $\mathcal{O}(-n)$ for each $n \in \mathbb{Z}_{>0}$.
    \end{enumerate}
\end{theorem}

\begin{remark}
    We construct the metrics in Theorem \refeq{thm_Kahlerclassification} for all $q$ satisfying the inequalities in Theorem \refeq{thm_Kahlerclassification}, even when $q$ is not a positive integer (and hence the metric is not an orbifold). 
\end{remark}

We review the cohomogeneity one construction of metrics, described in \cite{DancerWang2011} and \cite{Buzano2011} among others, in \S 2 and \S 3 with a focus on $\mathrm{SU}(2)$-invariant metrics. We also produce a classification of the cohomogeneity one simply-connected non-compact spaces in \S 2. We then introduce the boundary value problem in \S 3 and follow \cite{Buttsworth2025} to establish short-time existence of solutions for each boundary condition. We then prove Theorem \refeq{thm_expanders} using centre manifold analysis in \S 4. Numerical and analytical evidence for Conjecture \refeq{con_compact_solitons} is highlighted in \S 5 with additional figures in the appendix. Theorem \refeq{thm_Kahlerclassification} is proven in \S5.3, and is used to provide further evidence for Conjecture \refeq{con_compact_solitons}. 

\subsection*{Acknowledgements}
The author would like to thank his Masters supervisors at the University of New South Wales: Timothy Buttsworth and Michael Cowling. This work would not have been possible without the the many useful insights and fruitful ideas of Timothy Buttsworth. The author thanks the participants of the ``Geometric Analysis and Symmetries'' workshop, held at MATRIX in February 2025, for their interesting questions that shaped this paper. The University of New South Wales provided finical contributions to the author.

\section{Cohomogeneity one group actions}

The purpose of this section is to review the classification of irreducible almost effective cohomogeneity one actions on simply-connected compact $4$-manifolds and extend it to the non-compact case. The classification of cohomogeneity one actions is achievable by two key results we list below. Firstly, Theorem \refeq{thm_1}, which is a consequence of Kleiner's lemma. 

\begin{theorem}
\label{thm_1}
    If a Lie group acts properly on a connected smooth manifold $M$, then the orbit space contains a connected open dense subset of regular orbits.
\end{theorem}

We recall that regular orbits are orbits whose conjugacy class of stabilizer groups is minimal (under the standard ordering). The minimal stabilizer groups are called principal, stabilizer groups that are of higher dimension than the principal stabilizers are called singular, and the remainder are called exceptional. Identical naming conventions are given to the corresponding orbits.

\begin{remark}
    A geometric consequence of Theorem \refeq{thm_1} that we shall see in \S 3 is that the process of constructing metrics under symmetries is completed via a boundary value problem on the regular orbits with non-regular orbits as the boundary.
\end{remark}

The second result is the classical theorem from \cite{Mostert1957} below. 

\begin{theorem}
\label{thm_mos}
    If a compact Lie group acts smoothly via a cohomogeneity one action, then the orbit space is a $1$-dimensional manifold with boundary.
\end{theorem} 

Theorem \refeq{thm_1} also implies that a proper connected cohomogeneity-one action of a simply-connected manifold produces a simply-connected orbit space. Indeed, the space of regular orbits can be shown to be the interior of the orbit space as a manifold with boundary. Any loop in the orbit space is therefore homotopic to one containing only regular orbits by Theorem \refeq{thm_1} and hence lifts to a path in the manifold which can be closed by a path along an orbit. 

As we only consider simply-connected spaces, Theorem \refeq{thm_mos} implies we only need to consider group actions whose orbit spaces are homeomorphic to one of $\mathbb{R}$, a closed ray or a closed interval. We will henceforth always assume the action is via a compact group (and thus is proper). One should note that Hoelscher's classification of compact cohomogeneity one actions on simply-connected manifolds works over all proper actions. Indeed, a proper action induces a proper irreducible effective action which exists if and only if there exists a complete invariant metric. Furthermore, if $G$ acts properly and effectively on a connected compact manifold, then $G$ must be compact since the isometry group of a connected compact Riemannian manifold is compact.

We can now rule out the possibility of exceptional orbits using the following standard result.

\begin{lemma}
\label{lem_isotropy_sphere}
    Let a Lie group $G$ act with cohomogeneity one on a smooth manifold $M$. If $K = G_p$ is not principal and $H = G_x$ is principal for $x \in \exp(T_p(G\cdot p)^\perp)$, then $K/H$ is a sphere.
\end{lemma}

\begin{proof}
    $K/H$ is homeomorphic to $K\cdot v$, the image of the induced action of $K$ on sufficiently small but non-zero $v \in T_p(G \cdot p)^\perp$. In particular, $K \cdot v = G \cdot v$ is the sphere $\partial B_{||v||}(T_p(G\cdot p)^\perp)$ since $G$ acts homogeneously on it (trivially by Theorem \refeq{thm_mos} for example).
\end{proof}

\begin{corollary}
    Cohomogeneity one manifolds that admit exceptional orbits are not simply-connected.
\end{corollary}

\begin{proof}
    If a cohomogeneity one manifold, $M$, has an exceptional orbit, $\dim(K) = \dim(H)$, then $K/H = \mathbb{Z}_2$. Therefore $M$ admits a $2$-fold cover so it cannot be simply-connected. This argument can be repeated to show cohomogeneity manifolds with finite fundamental group cannot admit exceptional orbits following the proof of Lemma 1.7 of \cite{GWZ2008}.
\end{proof}

Since the action is proper, there exists a complete Riemannian metric such that the group acts by isometries on it. In the case of a compact group action, this metric can be constructed via a standard averaging procedure and the slice theorem. The metric on $M$ induces a metric on $M/G$ which is isometric to a connected closed interval $I \subseteq \mathbb{R}$ such that $0 = \inf (I)$ when $I \neq \mathbb{R}$. 

Thus, we may construct a unit-speed geodesic $c:I \to M$ between non-principal orbits (if they exist) that intersects each orbit perpendicularly and only once. By Kleiner's lemma and the slice theorem, if a Lie group acts with cohomogeneity one, then the Lie group $G$, the stabilizer group at a principal orbit $H = G_{c(t)}$ ($t \in \text{Int}(I)$) and at any non-principal orbits (denoted $K$ if there is one and $K_{\pm}$ if there are two) $K_\cdot = G_{c(t)}$ ($t \in \partial I$) are sufficient to reconstruct the group action up to equivariant diffeomorphism.

As in the compact case, we call $H \subseteq G$ or $H \subseteq K \subseteq G$ the group diagram for a non-compact manifold when there exists a geodesic $c:\mathbb{R} \to M$ intersecting the orbits perpendicularly such that $G_{c(t)} = H$ is a principal isotropy for $t\neq 0$ and $G_{c(0)}=K$ is a non-principal isotropy if one exists.

To summarize, choosing $I$ homeomorphic to $M/G$ such that $0 = \inf I$ if there is a singular orbit, a simply-connected Riemannian manifold has $I = \mathbb{R}$ or $I = [0,\infty)$ if it is not compact and $I = [0,T]$ if it is compact and:
\begin{itemize}
    \item if $M$ is compact, then a group action by Lie group $G$ is determined up to $G$-equivariant diffeomorphism by its group diagram: $H \subseteq K_-,K_+ \subseteq G$, where $K_-$ and $K_+$ are the singular isotropies at $c(0)$ and $c(T)$ respectively;
    \item if $M$ is non-compact and admits a singular orbit, then a group action is determined up to $G$-equivariant diffeomorphism by its group diagram: $H \subseteq K \subseteq G$, where $K = G_{c(0)}$; and 
    \item if $M$ is non-compact and does not admit a singular orbit, then $M$ is $G$-equivariantly diffeomorphic to $G/H \times \mathbb{R}$ under the standard action by $G$ on the first component. 
\end{itemize}

In the compact case, the group diagram can be changed via the three following operations without changing the group action (up to equivariant diffeomorphism) as shown at the end of \S 1 in \cite{GWZ2008}.

\begin{theorem}
    \label{thm_equiv_diagrams}
    If a group action by $G$ on $M$ has group diagram $H \subseteq K_{\pm} \subseteq G$, then any of the following will result in a $G$-equivariantly diffeomorphic manifold:
    \begin{enumerate}
        \item Swapping $K_-$ and $K_+$;
        \item Conjugating each stabilizer by the same group element; and
        \item Replacing $K_-$ with $hK_-h^{-1}$ for $h \in N(H)_0$.
    \end{enumerate}
    Conversely, the group diagrams of two $G$-equivariantly diffeomorphic manifolds can be transformed into each other by some combination of the above operations.
\end{theorem}

\begin{remark}
    The first two operations are geometric in nature, achievable by reversing the geodesic, or translating the geodesic via a group element, respectively. However, a geodesic intersecting each orbit perpendicularly that describes the geometric group diagram of an action may not do so after applying the last operation.
\end{remark}

It is easily checked that the relevant conditions apply to the non-compact case.

\begin{lemma}
\label{lem_noncompactequivariant}
    If $G$ acts on two non-compact simply-connected cohomogeneity one manifolds, they are $G$-equivariantly diffeomorphic if and only if their group diagrams are related to each other by conjugation of some element in $G$.
\end{lemma}

Any collection of Lie groups $H \subseteq K \subseteq G$ (or $H \subseteq K_{\pm} \subseteq G$) such that $K/H$ (or $K_{\pm}/H$) is a sphere can be realized as the group diagram of the obvious cohomogeneity one group action of $G$ on the disc bundle $G \times_K \mathbb{R}^{l+1}$ (or the union of two disc bundles in the compact case) as in the slice theorem, where $K/H \simeq \mathbb{S}^{l}$. Trivially, any group diagram of the form $H \subseteq G$ is realized by the obvious action of $G$ on $G/H \times \mathbb{R}$.

Therefore, it is sufficient to consider group diagrams when classifying and analyzing cohomogeneity one group actions. By studying the possible group diagrams, Hoelscher was able to give a simple classification of the almost effective cohomogeneity one actions on simply-connected compact $4$-manifolds in \cite{Hoelscher2008}. Hoelscher's classification implies that any cohomogeneity one metric on a compact $4$-dimensional manifold is invariant under an action $G$-equivariantly diffeomorphic to one in Table \refeq{tab_compact}.

\renewcommand{\arraystretch}{1.35}
\begin{table}[h!]
\begin{center}
    \begin{tabular}{||c | c c c c||} 
        \hline
        $M$ & $G$ & $K_-$ & $K_+$ & $H$\\ [0.5ex] 
        \hline\hline
        $\mathbb{S}^4$ & $\mathrm{SO}(3) \times \mathrm{SO}(2)$ & $\mathrm{SO}(2)\times \mathrm{SO}(2)$ & $\mathrm{SO}(2)\times \mathrm{SO}(2)$ & $\mathrm{SO}(2)\times 1$\\
        \hline
        $\mathbb{S}^2\times \mathbb{S}^2$ & $\mathrm{SO}(3) \times \mathrm{SO}(2)$ & $\mathrm{SO}(3)\times 1$ & $\mathrm{SO}(2)\times \mathrm{SO}(2)$ & $\mathrm{SO}(2)\times 1$\\
        \hline
        $\mathbb{S}^4$ & $\mathrm{SO}(3)$ & $\mathrm{S}(\mathrm{O}(2)\times \mathrm{O}(1))$ & $\mathrm{S}(\mathrm{O}(1)\times \mathrm{O}(2))$ & $\mathrm{S}(\mathrm{O}(1)\times\mathrm{O}(1)\times \mathrm{O}(1))$\\
        \hline
        $\mathbb{CP}^2$ & $\mathrm{SO}(3)$ & $\mathrm{S}(\mathrm{O}(1)\times \mathrm{O}(2))$ & $\mathrm{SO}(2)\times \mathrm{SO}(1)$ & $\mathbb{Z}_2$\\
        \hline
        $\mathbb{S}^4$ & $\mathrm{SU}(2)$ & $\mathrm{SU}(2)$ & $\mathrm{SU}(2)$ & $1$\\
        \hline
        $\mathbb{CP}^2$ & $\mathrm{SU}(2)$ & $\mathrm{SU}(2)$ & $\mathrm{U}(1)$ & $1$\\
        \hline
        $\mathbb{CP}^2\# \overline{\mathbb{CP}^2}$ & $\mathrm{SU}(2)$ & $\mathrm{U}(1)$ & $\mathrm{U}(1)$ & $\mathbb{Z}_n$,  $n$ \text{odd}\\
        \hline
        $\mathbb{S}^2\times \mathbb{S}^2$ & $\mathrm{SO}(3)$ & $\mathrm{SO}(2)$ & $\mathrm{SO}(2)$ & $\mathbb{Z}_n$\\
        \hline
    \end{tabular}
    \caption{Simply-connected compact $4$-manifolds that admit a cohomogeneity one group action and their effective irreducible group diagrams.}\label{tab_compact}
\end{center}
\end{table}
\renewcommand{\arraystretch}{1}
\begin{remark}
    By Theorem \refeq{thm_equiv_diagrams}, each row of Table \refeq{tab_compact} corresponds to a distinct group action. The group actions by $\mathrm{SO}(3)$ lift to other distinct $\mathrm{SU}(2)$ actions which are less convenient to present but more convenient to use. In particular, the final action lifts to $\mathbb{Z}_n \subseteq \mathrm{U}(1),\mathrm{U}(1) \subseteq \mathrm{SU}(2)$ for even $n$.
\end{remark}

In the remainder of this section, we similarly list the minimal cohomogeneity one actions by compact groups on non-compact $4$-manifolds --- a simply-connected Riemannian $4$-manifold invariant under some cohomogeneity one action must be invariant under one we list in Table \refeq{tab_non_compact}. This follows immediately from Theorem \refeq{thm_noncompact} by reducing the group action, taking the quotient by the ineffective kernel and taking the identity component. Before we prove Theorem \refeq{thm_noncompact}, we prove a useful Lemma. 

\begin{lemma}\label{lem_maxdimG}
    Let $G$ be a compact simply-connected Lie group of dimension $6$ that acts by a cohomogeneity one action on a non-compact manifold $M$ of dimension $4$. If there are no exceptional orbits, then $G = Spin(4)$ and the action is equivalent to 
    \begin{enumerate}
        \item the defining action of $\mathrm{SO}(4)$ on $\mathbb{R}^4$ if $M$ admits a singular orbit; or
        \item $\mathrm{SO}(4)$ on $\mathbb{S}^3 \times \mathbb{R}$ acting homogeneously on the first component if $M$ does not admit a singular orbit.
    \end{enumerate} 
\end{lemma}

\begin{remark}
    Lemma \refeq{lem_maxdimG} holds in higher dimensions when replacing $4$ with $n$ and $6$ with ${n(n-1)}/{2}$.
\end{remark}

\begin{proof}
    The proof follows as in the proof of Proposition 1.22 of \cite{Hoelscher2008}, noting these actions are given by group diagrams $Spin(n-1) \subseteq Spin(n) \subseteq Spin(n)$ and $Spin(n-1) \subseteq Spin(n)$ respectively.
\end{proof} 

\begin{theorem}
    \label{thm_noncompact}
    Each irreducible effective cohomogeneity one group action by a connected compact group on a non-compact simply-connected $4$-dimensional manifold is equivariantly diffeomorphic to one with group diagram in Table \refeq{tab_non_compact}.
\end{theorem}

\renewcommand{\arraystretch}{1.35}
\begin{table}[h!]
\begin{center}
    \begin{tabular}{||c | c c c||} 
        \hline
        $M$ & $G$ & $K$ & $H$\\ [0.5ex] 
        \hline\hline
        $\mathbb{S}^3 \times \mathbb{R}$ & $\mathrm{SU}(2)$ & - & $1$\\
        \hline
        $\mathbb{S}^2\times \mathbb{R}^2$ & $\mathrm{SO}(3) \times \mathrm{SO}(2)$ & $\mathrm{SO}(2)\times \mathrm{SO}(2)$ & $\mathrm{SO}(2)\times 1$\\
        \hline
        $\mathbb{R}^4$ & $\mathrm{SU}(2)$ & $\mathrm{SU}(2)$ & $1$\\
        \hline
        $\mathcal{O}(-n)$ & $\mathrm{SU}(2)$ & $\mathrm{U}(1)$ & $\mathbb{Z}_n$,  $n$ \text{odd}\\
        \hline
        $\mathcal{O}(-2n)$ & $\mathrm{SO}(3)$ & $\mathrm{SO}(2)$ & $\mathbb{Z}_n$\\
        \hline
    \end{tabular}
    \caption{Simply-connected non-compact $4$-manifolds that admit a cohomogeneity one compact group action and their effective irreducible group diagrams.}\label{tab_non_compact}
\end{center}
\end{table}
\renewcommand{\arraystretch}{1}

\begin{proof}
    As before, it will be convenient to construct ineffective actions in place of some effective actions. Since $G$ is connected, compact and acts by an effective action, it is covered by the product of simply-connected Lie groups and a torus, and the induced action by the cover is almost effective. We may therefore assume that $G$ acts almost effectively and is the product of simply-connected Lie groups and a torus. We also have the well known dimension restrictions of almost effective actions, $3 \leq \dim(G) \leq 6$ (see Proposition 1.21 of \cite{Hoelscher2008} for example). We can immediately rule out the case $\dim(G) = 6$ by Lemma \refeq{lem_maxdimG} as $Spin(4)$ is isomorphic to $\mathrm{SU}(2)\times \mathrm{SU}(2)$ which necessarily produces reducible actions. Hence, it suffices to consider $G = \mathrm{SU}(2)\times \mathbb{T}^k$ or $G = \mathbb{T}^{k+3}$ where $\mathbb{T}^k$ is a torus of dimension $0 \leq k \leq 3$.
    
    If there are no singular orbits, $M$ is diffeomorphic to $G/H \times \mathbb{R}$ while $G/H$ must be compact and simply-connected and hence is diffeomorphic to $\mathbb{S}^3$. The only irreducible action from the admissible groups is $G = \mathrm{SU}(2)$ and it must have trivial principal isotropy as $G$ covers $G/H$.

    If instead there is a singular orbit, then there is exactly one. Moreover, the singular orbit $G/K$ is a deformation retract of $M$, so $G/K$ is simply-connected and hence $K$ is connected. In particular, since $\dim(M)>2$, $G$ is not a torus. As the $\mathrm{U}(1)$ subgroups are the only connected non-trivial subgroups of $\mathrm{SU}(2)$, if $G = \mathrm{SU}(2)$, then either $K=\mathrm{SU}(2)$ and $H=1$ since $K/H$ must be a $3$-sphere, or $K=\mathrm{U}(1)$ and $H = \mathbb{Z}_n$ for any $n\in\mathbb{Z}_{>0}$. If $G$ has a torus component and $H$ projects surjectively onto a circle group in the torus component, then the action is obviously reducible. If the action is by $G=\mathrm{SU}(2)\times \mathbb{T}^k$, for $k \geq 2$, then it is reducible as there are no $2$ dimensional subgroups of $\mathrm{SU}(2)$ and $M$ is $4$-dimensional. This leaves one case; if $G = \mathrm{SU}(2) \times \mathrm{U}(1)$, then $H = \mathrm{U}(1) \times 1$ and the conditions on $G/K$ and $K/H$ imply that $K = \mathrm{U}(1) \times \mathrm{U}(1)$, which can be written more effectively as in Table \refeq{tab_non_compact}. Lemma \refeq{lem_noncompactequivariant} implies we do not need to specify the $\mathrm{U}(1)$ (or $\mathrm{SO}(2)$) subgroups in Table \refeq{tab_non_compact} and each corresponds to only one group action.
\end{proof}

\begin{remark}
    One more action worth mentioning is the action of $\mathrm{SO}(3) \times \mathrm{SO}(2)$ on the space $\mathbb{R}^3 \times \mathbb{S}^1$ with group diagram $\mathrm{SO}(2)\times 1 \subseteq \mathrm{SO}(3) \times 1 \subseteq \mathrm{SO}(3) \times \mathrm{SO}(2)$. The invariant metrics are doubly warped metrics --- the same type as on the $\mathrm{SO}(3)\times \mathrm{SO}(2)$ action in Table \refeq{tab_non_compact}. One may also note that this action lifts to a proper action by the non-compact group $\mathrm{SO}(3)\times \mathbb{R}$ on $\mathbb{R}^4$ and these two actions have equivalent spaces of invariant metrics (and hence equivalent spaces of Ricci solitons whose potential is also invariant). This can be seen by lifting the compact action to the equivalent but less effective action by $\mathrm{SO}(3)\times \mathbb{R}$ action on $\mathbb{R}^3 \times \mathbb{S}^1$. Finally we note that for the $\mathrm{SO}(3)\times \mathrm{SO}(2)$ actions here, all the invariant asymptotically conical expanding gradient Ricci solitons are already constructed \cite{GastelKronz2004,NienhausWink2024}.
\end{remark}

For the remainder of the work we focus solely on the case with singular orbits. There has been some study into the case without singular orbits, in particular, an $\mathrm{SU}(2)$-invariant expanding soliton on $\mathbb{R} \times \mathbb{S}^3$ must satisfy the scalar curvature conditions $-4 \leq S < -3$ at some point (normalizing $\lambda = - 1$) by Theorem 1.4 of \cite{MunteanuWang2011}.

Before moving to our boundary value problem in the following section, we note that the classification in Theorem \refeq{thm_noncompact} is sufficient to study simply-connected expanders with compact cohomogeneity one actions. If a compact gradient Ricci soliton is steady or expanding, the work of Perelman \cite{Perelman2002} shows that it must be Einstein. Since these metrics are invariant under group actions containing $\mathrm{SU}(2)$ or $\mathrm{SO}(3)$ components, the isometry group of the Einstein metrics cannot be abelian. However, Theorem 1.84 of \cite{Besse1987} implies that these metrics must have abelian isometry group. Therefore, simply connected cohomogeneity one gradient expanders (and steadies) can only be found on non-compact spaces.

\section{The initial value problem for $\mathrm{SU}(2)$-invariant gradient Ricci solitons in dimension 4}
In this section, we determine the initial value problem for gradient Ricci solitons invariant under the $\mathrm{SU}(2)$ actions in the previous section. We then prove short-time existence of solutions to the initial value problem while parameterizing and characterizing the short-time solutions.
\subsection{Constructing the initial value problem}
We first determine the Ricci soliton ODEs on the regular orbits before describing the smoothness conditions on the singular orbits. Following \S 2, if $(M,g)$ is a complete simply-connected Riemannian manifold invariant under a cohomogeneity one group action $G$, then there exists an open interval $I \subseteq \mathbb{R}$ and a unit-speed geodesic $c: I \to M_{reg}$ that interests each regular orbit perpendicularly and only once. The geodesic induces a $G$-equivariant diffeomorphism $I \times G/H \mapsto M_{reg}$ by the slice theorem. Therefore, each metric on $M$ corresponds via pullback to (the completion of) a $G$-invariant metric
\begin{equation}
    \label{eq_cohomogeneityonemetric}
    g = dt^2 + g_t
\end{equation}
on $I \times G/H$, where $(G/H,g_t)$ is homogenous for each $t \in I$. To further elucidate the metric, one must use action fields: the natural way to generate vector fields from elements of the Lie algebra $\mathfrak{g}$ of $G$. If $G$ is a Lie group acting on $M$, the action field of $X \in \mathfrak{g}$ is the vector field given by $X_p^* = \frac{d}{dt}|_{t=0} \exp(tX)p \in T_p(G \cdot p)$. In particular, if $\mathfrak{n}$ is an $\text{Ad}_H$-invariant complement of the Lie subalgebra $\mathfrak{h}$ of $\mathfrak{g}$, then action fields identify $\mathfrak{n}$ with $T_{c(t)}(G \cdot c(t))$. Specializing to $\mathrm{SU}(2)$ actions, $\mathfrak{h}$ is trivial and $\mathfrak{n} = \mathfrak{g} = \mathfrak{su}(2)$.

Action fields can also be used as in Proposition 2 of \cite{Buttsworth2025} to diagonalize $\mathrm{SU}(2)$-invariant $4$-dimensional gradient Ricci solitons. We can then use such a frame to produce the ODEs for the Ricci soliton equation on the regular orbits and the boundary smoothness conditions on the singular orbits.

\begin{theorem}
\label{thm_diagonalization}
    If a $4$-dimensional gradient Ricci soliton and its soliton potential are invariant under a cohomogeneity one $\mathrm{SU}(2)$-action, then there exists a basis $X_1,X_2,X_3 \in \mathfrak{su}(2)$ with relations $[X_i,X_{i+1}] = 2X_{i+2}$, where indices are taken modulo 3, such that the corresponding dual frame of smooth one-forms $dt,\omega_1,\omega_2,\omega_3$ to $\partial_t,X_1^*,X_2^*,X_3^*$ diagonalize $g = dt^2 + g_t$ on $M_\text{reg}$:
\begin{equation}
    \label{eq_metricfunctions}
    g = dt^2 + {f_1(t)}^2\omega_1^2+{f_2(t)}^2\omega_2^2+{f_3(t)}^2\omega_3^2
\end{equation}
for some smooth positive functions $f_i: I \to \mathbb{R}$. Moreover, $f_i$ are unique up to permutation.
\end{theorem}

\begin{proof}
    The proof follows as in \cite{Buttsworth2025} which generalizes the Bianchi IX (Einstein) case explicitly shown in \cite{Dammerman2009} (which was already known --- see \cite{DancerStrachan1994} for example). As the automorphism group of bases satisfying the Lie bracket relations in the theorem is $\mathrm{SO}(3)$, we may take $X_1,X_2,X_3$ to diagonalize the matrix $(g_{t})_{ij}(t_0)$ for some $t_0 \in I$. As the soliton potential is invariant, $\text{Hess}(\partial_t,T_{c(t)}(G \cdot c(t))) = 0$ and (\refeq{eq_soliton}) implies $\text{Ric}(\partial_t,T_{c(t)}(G \cdot c(t))) = 0$ too. In coordinates for a diagonal basis at $t_0$, $\text{Ric}(\partial_t,T_{c(t)}(G \cdot c(t))) = 0$ becomes
    \begin{equation}
        \label{eq_to_diag}
        (g_t^{ii}(t_0)-g_t^{jj}(t_0))g_{t \; ij}'(t_0) = 0.
    \end{equation} 
    Equation (\refeq{eq_to_diag}) implies the basis can be rotated to also ensure that $(g_{t})'_{ij}(t_0)$ is diagonal too. Standard ODE existence and uniqueness results imply that $g$ is diagonal using the other Ricci solitons equations (see equations (2.14)--(2.16) of \cite{DancerWang2011}). 
    
    The one-parameter family of metrics ${(g_t)}_{t \in I}$ determines $f_1,f_2,f_3$ up to permutation as the frame is determined up to permutation at any point unless two functions are equal on an open interval by (\refeq{eq_to_diag}). If two functions are equal on any open interval, they are identical (see (\refeq{eq_ODE1})--(\refeq{eq_ODEi})) and the choice of corresponding one-forms does not affect these functions.
\end{proof}

\begin{remark}
    Metrics satisfying (\refeq{eq_metricfunctions}) are $\mathrm{U}(2)$-invariant if and only if two of the functions $f_1,f_2$ or $f_3$ are identical.
\end{remark}

\begin{remark}
    The $\mathrm{SO}(3)\times \mathrm{SO}(2)$ invariant metrics are necessarily diagonal (i.e., doubly warped) by the isotropy representation on $\mathfrak{n}$ which decomposes into its one dimensional abelian part $\mathfrak{a}$ and the two dimensional $Ad_H$-invariant subspace of $\mathfrak{su}(2)$. Choose a basis $a \in \mathfrak{a}$ and $X_1,X_2,X_3 \in \mathfrak{su}(2)$ satisfying the same Lie bracket relations as in Theorem \refeq{thm_diagonalization}, such that $X_1$ is tangent to the chosen $\mathrm{SO}(2) \times 1$ isotropy subgroup. One computes that, for any $h \in H$, $(dL_h)(X^*_{c(t)}) = (Ad(h)(X))^*_{c(t)}$ and so for any invariant metric:
    \begin{alignat*}{2}
    g(a^*,X_{2,3}^*) &= g((Ad_{\exp(\frac{\pi}{2} X_1)}a)^*,(Ad_{\exp(\frac{\pi}{2} X_1)}X_{2,3})^*) &{}={}& -g(a^*,X_{2,3}^*),\\
    g(X_2^*,X_3^*) &= g((Ad_{\exp(\frac{\pi}{4} X_1)}X_2)^*,(Ad_{\exp(\frac{\pi}{4} X_1)}X_3)^*) &{}={}& -g(X_3^*,X_2^*),\\
    g(X_2^*,X_2^*) &= g((Ad_{\exp(\frac{\pi}{4} X_1)}X_2)^*(Ad_{\exp(\frac{\pi}{4} X_1)}X_2)^*) &{}={}& g(X_3^*,X_3^*).
    \end{alignat*}
    Therefore the metric is diagonal: $g = dt^2 + f_1(t)^2(da^*)^2 + f_2(t)^2((dX_2^*)^2 +(dX_3^*)^2)$.
\end{remark} 

The frame in Theorem \refeq{thm_diagonalization} can be used to calculate the Ricci curvature of the metric as in Proposition 1.14 of \cite{GroveZiller2002} for example. This reveals that a metric satisfying (\refeq{eq_cohomogeneityonemetric}) that is invariant under an almost effective cohomogeneity one $\mathrm{SU}(2)$-action is a Ricci soliton only if it satisfies (\refeq{eq_metricfunctions}) for an invariant potential function $u$ and an Einstein constant $\lambda$ such that $u,f_i:I\to\mathbb{R}$ satisfy the following system:
\begin{alignat}{1}
    \lambda &= - \frac{f_1''}{f_1}-\frac{f_2''}{f_2}-\frac{f_3''}{f_3} + u''; \text{  and} \label{eq_ODE1} \\
    \lambda &= -\frac{f_i''}{f_i}+\frac{f_i'}{f_i}\left(u'- \frac{f_j'}{f_j}-\frac{f_k'}{f_k}\right)+ 2 \frac{f_i^4-{(f_j^2-f_k^2)}^2}{f_i^2f_j^2f_k^2}, \label{eq_ODEi}
\end{alignat}
for all $i,j,k \in \{1,2,3\}$, $i\neq j \neq k \neq i$. Moreover functions satisfying (\refeq{eq_ODE1}) and (\refeq{eq_ODEi}) correspond to a Ricci soliton with these metric functions if the metric can be extended to the singular orbits smoothly. 

Smoothness conditions for the metric in (\refeq{eq_metricfunctions}) and the soliton potential can be written in terms of $f_i$ and $u$, which gives a boundary value problem we analyze in \S 4 and beyond. The exact boundary conditions depend on the choice of principal and non-principal stabilizer groups. First, if the soliton potential is invariant and smooth around a singular orbit, then it is smooth on the singular orbit whenever it is smooth in the direction $\partial_t$, i.e., along an extension of the geodesic. In particular, it suffices that the odd derivatives of the soliton potential $u:I \to \mathbb{R}$ vanish at the boundary:
\begin{equation}
    \label{eq_u_smoothness}
    u^{(odd)}(0)=0.
\end{equation}
Smoothness conditions for $\mathrm{SU}(2)$-invariant actions can be written in terms of $f_1,f_2,f_3$ using the results in \cite{VerdianiZiller2022}, again via the frame in Theorem \refeq{thm_diagonalization}. The conditions are given in (\refeq{eq_boundaryzero}) and (\refeq{eq_boundarynonzeroSU2}) below, where we again use the convention $0 = \inf I$ when $\partial (M_{reg}/G)$ is non-empty. The conditions are similar at $T = \sup I$ if $T<\infty$ as we remark below. All smoothness conditions required in this work are equivalent to either: 
\begin{itemize}
\item the smoothness conditions of a metric invariant under the standard fixed point action of $\mathrm{SU}(2)$ on $\mathbb{C}^2$:
\begin{equation}
    \label{eq_boundaryzero}
    f_i'(0) = 1 \; \; \; \text{ and } \; \; \; f_i^{(even)}(0)=0
\end{equation}
for all $i=1,2,3$; or
\item the smoothness conditions at the singular orbit of a metric invariant under the standard $\mathrm{SU}(2)$-action on $\mathcal{O}(-n)$ for any $n \in \mathbb{Z}_{>0}$:
\begin{equation}
    \label{eq_boundarynonzeroSU2}
    {f_1(t)}^2 = n^2t^2 + t^4\phi_1(t^2) \; \; \; \text{ and } \; \; \;  {f_{2,3}(t)}^2 = \phi_2(t^2) \pm t^{4/n}\phi_3(t^2)
\end{equation}
on $(-\varepsilon,\varepsilon)$ for some $\varepsilon>0$, where $f_1$ is chosen without loss of generality to vanish on the singular orbit, $\phi_1,\phi_2,\phi_3: \mathbb{R} \to \mathbb{R}$ are any smooth functions such that $\phi_2(0) >0$ and $\phi_3$ vanishes if $n \not\in\{1,2,4\}$. 
\end{itemize}

If $M$ is simply-connected and compact, then $0<T<\infty$ and one computes the smoothness conditions at $T$ by applying the transformation $t \mapsto T-t$ and potentially applying some permutation of indices $\{1,2,3\}$ to once again arrive at (\refeq{eq_boundaryzero}) or (\refeq{eq_boundarynonzeroSU2}). In particular, the same function or number of functions need not vanish at $0$ and $T$. The possible combinations of smoothness conditions are given in Table \refeq{tab_compact_smoothness} and are easily computed in almost all scenarios again using the results in \cite{VerdianiZiller2022} when projecting $H$ onto $K_{\pm}$. The remaining special cases are when the group diagram is $\mathbb{Z}_n \subseteq \mathrm{U}(1),\mathrm{U}(1) \subseteq \mathrm{SU}(2)$ for $n=1$ and $n=2$. Since $N(\mathbb{Z}_2)_0 = \mathrm{SU}(2) = N(1)_0$, for either action, we may replace the second circle group with any other while achieving an equivalent action as per Theorem \refeq{thm_equiv_diagrams}. While these group diagrams are from equivalent group actions, we must consider each individually as the geodesic and hence the class of invariant metrics of the form (\refeq{eq_cohomogeneityonemetric}) may not be the same. Nevertheless, the diagonalization of Theorem \refeq{thm_diagonalization} and the dependence of equations (\refeq{eq_ODE1})--(\refeq{eq_ODEi}) on only the Lie bracket relations imply that the only other boundary conditions we must investigate are when different functions vanish at the boundary (see Example \refeq{ex_sum_of_round}), i.e., a permutation of indices.

By Theorem \refeq{thm_diagonalization}, if not all functions vanish at both ends, we may choose $f_1$ to vanish as $t \to 0$ and $f_3$ to vanish as $t \to T$ as in Table \refeq{tab_compact_smoothness}. Altogether, these smoothness conditions are necessary and sufficient when the metric functions are smooth on $I$ or equivalently when the metric is smooth on $M_{reg}$. 

We have therefore proven the following result where $M_n$ is the manifold from the group diagram $\mathbb{Z}_n \subseteq \mathrm{U}(1),\mathrm{U}(1) \subseteq \mathrm{SU}(2)$, i.e., it is diffeomorphic to $\text{Bl}_1(\mathbb{CP}^2) = \mathbb{CP}^2\# \overline{\mathbb{CP}^2}$ if $n$ is odd and $\mathbb{S}^2\times \mathbb{S}^2$ otherwise.

\begin{lemma}
\label{lem_BVP}
    Suppose $\mathrm{SU}(2)$ acts with a cohomogeneity one action on a smooth $4$-dimensional manifold $M$. A smooth $\mathrm{SU}(2)$-invariant metric $g$ on $M$ is a Ricci soliton if there exist functions $f_1,f_2,f_3,u:I \to \mathbb{R}$ satisfying: 
\begin{enumerate}
    \item (\refeq{eq_ODE1})--(\refeq{eq_ODEi}) on $I$;
    \item $u$ has vanishing odd derivatives on $\partial I$; and
    \item $f_1,f_2,f_3$ satisfy one of the corresponding smoothness condition(s):
    \begin{itemize}
        \item (\refeq{eq_boundaryzero}) if $M$ is diffeomorphic to $\mathbb{R}^4$;
        \item (\refeq{eq_boundarynonzeroSU2}) if $M$ is diffeomorphic to $\mathcal{O}(-n)$ with consistent $n$; or
        \item as in Table \refeq{tab_compact_smoothness} if $M$ is compact.
    \end{itemize}        
\end{enumerate}
    Conversely, a collection of four functions satisfying all the above conditions establishes a smooth $\mathrm{SU}(2)$-invariant Ricci soliton and its corresponding smooth manifold.
\end{lemma}

\renewcommand{\arraystretch}{1.35}
\begin{table}[h!]
\begin{center}
    \begin{tabular}{||c|| c | c ||} 
        \hline
        $M$ & Smoothness as $t\to 0$ & Smoothness as $T-t\to 0$ \\ [0.5ex] 
        \hline\hline
        $\mathbb{S}^4$ & (\refeq{eq_boundarynonzeroSU2}) where $n=4$ & (\refeq{eq_boundarynonzeroSU2}) where $n=4$ and $(f_1,f_3) \mapsto (f_3,f_1)$ \\
        \hline
        $\mathbb{CP}^2$ & (\refeq{eq_boundarynonzeroSU2}) where $n=4$ & (\refeq{eq_boundarynonzeroSU2}) where $n=2$ and $(f_1,f_3) \mapsto (f_3,f_1)$ \\
        \hline
        $\mathbb{S}^4$ & (\refeq{eq_boundaryzero}) & (\refeq{eq_boundaryzero}) \\
        \hline
        $\mathbb{CP}^2$ & (\refeq{eq_boundaryzero}) & (\refeq{eq_boundarynonzeroSU2}) where $n=1$ \\
        \hline
        $M_n$ & (\refeq{eq_boundarynonzeroSU2}) & (\refeq{eq_boundarynonzeroSU2}) \\
        \hline
        $M_n$ \; $n \in \{1,2\}$ & (\refeq{eq_boundarynonzeroSU2}) & (\refeq{eq_boundarynonzeroSU2}) where $(f_1,f_3) \mapsto (f_3,f_1)$  \\
        \hline
    \end{tabular}
    \caption{Smoothness conditions of Ricci solitons on $\mathrm{SU}(2)$-invariant manifolds in Table \refeq{tab_compact}, written in the same order.}\label{tab_compact_smoothness}
\end{center}
\end{table}
\renewcommand{\arraystretch}{1}

\subsection{Short-time existence}
For the remainder of this section, we concern ourselves with proving the short-time existence of solutions to the boundary value problem from \S3.1 and developing convenient parameter spaces for them. 

It should be noted that only the sign of $\lambda$ in (\refeq{eq_soliton}) is important as one produces Ricci solitons for any $\lambda$ of the same sign by rescaling the metric. As a result, the parameter space of steady Ricci solitons is one dimension smaller than the parameter spaces for shrinkers and for expanders. In particular, if the system $(f_1,f_2,f_3,u)$ of functions satisfies the conditions of Lemma \refeq{lem_BVP} for Einstein constant $\lambda$, then the functions
    \[\overline{f_i}(t) = \frac{f_i(ct)}{c} \; \; \; \text{ and } \; \; \; \overline{u}(t) = u(ct)\]
satisfy the same set of conditions in Lemma \refeq{lem_BVP} (after rescaling the interval $I$ appropriately) with $\overline{\lambda} = c^2\lambda$ for any $c>0$. Consequently, we set $\lambda = -1$ when investigating expanding solitons and $\lambda = 1$ for shrinking solitons.

To develop short-time existence and construct a solution which we later perturb, we convert the 4 equations (\refeq{eq_ODE1})--(\refeq{eq_ODEi}) into a first order system of 7 ODEs using the transformation $L_i = f_i'/f_i$, $R_i = {f_i}/(f_j f_k)$ and $\xi = L_1+L_2+L_3-u'$:
\begin{alignat}{1}
    \xi' &= -L_1^2-L_2^2-L_3^2-\lambda; \label{eq_xiSU2} \\
    R_i' &= R_i(L_i-L_j-L_k); \text{ and} \label{eq_RiSU2} \\
    L_i' &= -\xi L_i+ 2R_i^2 - 2{(R_j-R_k)}^2 - \lambda, \label{eq_LiSU2}
\end{alignat}
where we always require $i,j,k \in \{1,2,3\}$, $i\neq j \neq k \neq i$ as before.

We define the parameters for the metrics on $\mathbb{R}^4$ referenced in Theorem \refeq{thm_expanders} as follows:
\begin{equation}
    \label{eq_a_i}
    a_i := \frac{f_i'''(0)}{3} = \lim_{t \to 0} \frac{L_i(t)t-1}{t}
\end{equation}
for each $i \in \{1,2,3\}$. For convenience we also set 
\begin{equation}
    \label{eq_alpha}
    \alpha := -u''(0) = \lim_{t \to 0} \frac{\xi(t)-L_1(t)-L_2(t)-L_3(t)}{t},
\end{equation}
which is not independent of $(a_1,a_2,a_3)$ as it can be written as $\alpha =-\lambda-3a_1-3a_2-3a_3$ by computing $a_1+a_2+a_3+\alpha = (\xi - {3}/{t})'|_{t=0} = -\lambda-2(a_1+a_2+a_3)$ with (\refeq{eq_xiSU2}). We now give the short-time existence of $\mathrm{SU}(2)$-invariant Ricci solitons about fixed points.

\begin{lemma}\label{lem_fixedSTE}
For all $\lambda\in\mathbb{R}$ there exists a continuous function $\Theta_1:\mathbb{R}^3 \to \mathbb{R}^7$ with the following properties:
\begin{itemize}
    \item for all $a = (a_1,a_2,a_3) \in \mathbb{R}^3$ there exists some $\tau>0$ depending continuously on $a$ such that there is a unique solution of (\refeq{eq_xiSU2})--(\refeq{eq_LiSU2}) satisfying (\refeq{eq_a_i}) whose evaluation at $t=\tau$ coincides with $\Theta_1(a)$;
    \item the corresponding functions $(f_1,f_2,f_3,u')$ satisfy the smoothness conditions in (\refeq{eq_u_smoothness})--(\refeq{eq_boundaryzero});
    \item if $\alpha=-\lambda-3a_1-3a_2-3a_3 = 0$, then $\xi = L_1 + L_2 + L_3$ and the corresponding metric is Einstein; and
    \item if $a_i=a_j$, then $L_i=L_j$, $R_i=R_j$ and the corresponding metric is $\mathrm{U}(2)$-invariant, moreover if $a_k=4a_i$ too, then $R_k=L_i$, $\xi = L_k+2R_k+{\alpha}/{R_i}$ and the corresponding metric is also K\"ahler.
\end{itemize}
\end{lemma}

\begin{proof}
In the transformed system, the boundary conditions in (\refeq{eq_u_smoothness}) and (\refeq{eq_boundaryzero}) imply that solutions which give smooth metrics (and potential functions) are of the form:
\[\xi = \frac{3}{t} + \eta_0(t), \; \;  L_i = \frac{1}{t} + \eta_{i}(t) \; \text{ and } \; R_i = \frac{1}{t} + \eta_{i+3}(t),\]
for each $i \in \{1,2,3\}$, where each $\eta_k$ is smooth and $\eta_k(0) = 0$. The ODEs (\refeq{eq_xiSU2})--(\refeq{eq_LiSU2}) can be written in terms of $\eta = (\eta_0,\ldots,\eta_6)$ by 
\[\eta'(t) = \frac{1}{t}A\eta'(t)+K+B(\eta(t))\]
where 
\[A = \left(\begin{matrix} 
    0 & -2 & -2 & -2 & 0 & 0 & 0\\
    -1 & -3 & 0 & 0 & 4 & 0 & 0\\
    -1 & 0 & -3 & 0 & 0 & 4 & 0\\
    -1 & 0 & 0 & -3 & 0 & 0 & 4\\
    0 & 1 & -1 & -1 & -1 & 0 & 0\\
    0 & -1 & 1 & -1 & 0 & -1 & 0\\
    0 & -1 & -1 & 1 & 0 & 0 & -1\\
\end{matrix}\right), \; \; \; K = \left(\begin{matrix} 
    -\lambda\\
    -\lambda\\
    -\lambda\\
    -\lambda\\
    0\\
    0\\
    0\\
\end{matrix}\right),  \]
and $B:\mathbb{R}^7 \to \mathbb{R}^7$ is smooth with $B(0) = 0$. We note that $A = PDP^{-1}$ where 
\[D = \left(\begin{matrix} 
    -5 & 0 & 0 & 0 & 0 & 0 & 0\\
    0 & -5 & 0 & 0 & 0 & 0 & 0\\
    0 & 0 & -3 & 0 & 0 & 0 & 0\\
    0 & 0 & 0 & -2 & 0 & 0 & 0\\
    0 & 0 & 0 & 0 & 1 & 0 & 0\\
    0 & 0 & 0 & 0 & 0 & 1 & 0\\
    0 & 0 & 0 & 0 & 0 & 0 & 1\\
\end{matrix}\right),  \; \; \; P = \left(\begin{matrix} 
    0 & 0 & 4 & 3 & 4 & 4 & 4\\
    2 & 2 & 2 & 1 & -1 & -1 & 0 \\
    0 & -2 & 2 & 1 & -1 & 0 & -1\\
    -2 & 0 & 2 & 1 & 0 & -1 & -1\\
    -1 & -1 & 1 & 1 & 0 & 0 & 1\\
    0 & 1 & 1 & 1 & 0 & 1 & 0\\
    1 & 0 & 1 & 1 & 1 & 0 & 0\\
\end{matrix}\right) \] \[ \text{ and } P^{-1} = \frac{1}{36}\left(\begin{matrix} 
    0 & 4 & 4 & -8 & -4 & -4 & 8\\
    0 & 4 & -8 & 4 & -4 & 8 & -4\\
    9 & 9 & 9 & 9 & -18 & -18 & -18\\
    -12 & -8 & -8 & -8 & 32 & 32 & 32\\
    3 & -5 & -5 & 7 & -10 & -10 & 14\\
    3 & -5 & 7 & -5 & -10 & 14 & -10\\
    3 & 7 & -5 & -5 & 14 & -10 & -10\\
\end{matrix}\right).\]
Setting $\nu = P^{-1}\eta$, gives another system $\nu'(t) = {1}/{t} \cdot D \nu + P^{-1}K+P^{-1}B(P\nu)$. All entries of $P^{-1}K$ except the third and fourth necessarily vanish. By a standard singular ODE existence and uniqueness argument, there exists a unique smooth solution for $\nu$ (and for $\eta$) on a neighborhood of $0$ for each $(\nu_4'(0),\nu_5'(0),\nu_6'(0)) \in \mathbb{R}^3$ corresponding to the entries of $1$ in $D$. 
For fixed $\lambda \in \mathbb{R}$, the short-time solutions can be parameterized by $a = (a_1,a_2,a_3) \in \mathbb{R}^3$ as we can write $\eta_{i+3}$ in terms of $a_i$ (see (\refeq{eq_fixed_approx})),
\[\nu_{i+3}'(0)= - \frac{\lambda}{12}+a_i-a_j-a_k.\]
The first condition of Lemma \refeq{lem_fixedSTE} will now hold by choosing sufficiently small $\tau >0$. However, to ensure $\Theta_1$ is continuous, for each $m \in \mathbb{Z}_{>0}$, set $\tau_m>0$ a constant such that all solutions with parameters $a \in \overline{B_m(0)}$ exist on $(0,2\tau_m)$. Then let $\tau_r:\mathbb{R}_{\geq 0} \to \mathbb{R}^+$ be a continuous function such that $\tau_r(t) < \tau_m$ whenever $t \leq m-1$ and finally choose $\tau = \tau_r(||a||)$.

The boundary conditions $f_i(0) = 0$ and $f_i'(0) = 1$ follow immediately as $f_i^2 = (R_j R_k)^{-1} = t^2 + O(t^3)$ and by induction, it follows that $\eta^{(even)}(0) = 0$ which reproduces the boundary conditions in (\refeq{eq_boundaryzero}). 

Now note that 
\begin{align*}
\frac{d}{dt}\left(\begin{matrix}
        L_i-L_j \\
        R_i-R_j \\
    \end{matrix} \right) & = \left(\begin{matrix}
        -\xi & 4(R_i+R_j-R_k) \\
        R_i+R_j & -L_k \\
    \end{matrix} \right) \left(\begin{matrix}
        L_i-L_j \\
        R_i-R_j \\
    \end{matrix} \right) \\
& = \left(\frac{1}{t}+ O(t)\right)\left(\begin{matrix}
        -3 & 4 \\
        2 & -1 \\
    \end{matrix} \right)\left(\begin{matrix}
        L_i-L_j \\
        R_i-R_j \\
    \end{matrix} \right) 
\end{align*}
\begin{equation}
    \label{eq_U2condition}
    \frac{d}{dt}S^{-1}\left(\begin{matrix}
        L_i-L_j \\
        R_i-R_j \\
    \end{matrix} \right) = \left(\frac{1}{t}+ O(t)\right)\left(\begin{matrix}
        -5 & O(t^2) \\
        O(t^2) & 1 \\
    \end{matrix} \right)S^{-1}\left(\begin{matrix}
        L_i-L_j \\
        R_i-R_j \\
    \end{matrix} \right), 
\end{equation}
where $S = \left(\begin{matrix}
    -2 & 1 \\
    1 & 1 \\
\end{matrix} \right)$ and $\displaystyle S^{-1} = \frac{1}{3} \left(\begin{matrix}
    -1 & 1 \\
    1 & 2 \\
\end{matrix} \right)$. 
By another standard singular ODE existence and unique argument using equation (\refeq{eq_U2condition}), $-2(R_i-R_j) = L_i-L_j = R_i-R_j$ if and only if $0 = (L_i-L_j)'(0)+2(R_i-R_j)'(0) = 3a_i-3a_j$. This proves the $\mathrm{U}(2)$-invariant part of the final condition of Lemma \refeq{lem_fixedSTE}. For the K\"ahler part, consider the system
\begin{equation}
    \label{eq_U2_Kahler}
    \begin{aligned}
    \xi' &= -L_k^2-2R_k^2-\lambda; \\
    L_k' &= -\xi L_k+ 2R_k^2-\lambda; \\
    R_i' &= -R_i L_k; \; \text{ and} \\
    R_k' &= R_k L_k-2R_k^2.
    \end{aligned}
\end{equation}
We note that the system (\refeq{eq_U2_Kahler}) is the $\mathrm{U}(2)$-invariant subsystem of (\refeq{eq_xiSU2})--(\refeq{eq_LiSU2}) (see (\refeq{eq_U2ODEs})) where $L_i = R_k$. 

If $Z  = \lambda + R_k(\xi+L_k-4R_i)$, then $Z' = -2R_k Z$ in (\refeq{eq_U2_Kahler}). Since $R_k(t)>0$, if $Z(0) = 0$, then $Z$ vanishes uniformly in (\refeq{eq_U2_Kahler}).
Moreover, if $Z  = \lambda + R_k(\xi+L_k-4R_i)$, then $Z = R_k'-L_i'$ in (\refeq{eq_U2ODEs}). In particular, if $a_k=4a_i=4a_j$, then $Z(0) = \lim_{t \to 0}R_k'(t)-L_i'(t) = (a_k-4a_i)/2 = 0$ (see (\refeq{eq_fixed_approx})). Therefore, the solution in (\refeq{eq_U2_Kahler}) where $Z$ vanishes is also a solution of (\refeq{eq_U2ODEs}) and hence is realized in (\refeq{eq_U2ODEs}) by uniqueness. This proves that $a_k=4a_i=4a_j$ implies $L_i = R_k$. The other K\"ahler condition is found by integrating
\[\frac{d}{dt}\ln(|L_k+2R_k-\xi|)=-\frac{d}{dt}\ln(|R_i|).\] 
The Einstein condition is proven similarly, as per the proof of Theorem 4 of \cite{Buttsworth2025}.
\end{proof}

We define the parameters for the metrics on $\mathcal{O}(-n)$ as follows:
\begin{alignat}{2}
    \beta &:= \frac{n}{{f_2(0)}^2} &&= \lim_{t \to 0} \frac{R_1(t)}{t}, \label{eq_beta} \\
    \gamma &:= \frac{\phi_3(0)}{\phi_2(0)} &&= \lim_{t \to 0} \frac{n(L_2(t)-L_3(t))}{4t^{\frac{4}{n}-1}}, \label{eq_gamma}
\end{alignat}
and $\alpha$ as in (\refeq{eq_alpha}). We now develop the short-time existence of gradient Ricci solitons on $\mathcal{O}(-n)$, extending the work in \cite{Buttsworth2025}.

\begin{lemma}\label{lem_SU2STE}
For all $\lambda\in\mathbb{R}$ and $n \in \mathbb{Z}_{>0}$, there exists a continuous function $\Theta_n:\mathbb{R}^3 \to \mathbb{R}^7$ with the following properties:
\begin{itemize}
    \item for all $(\alpha,\beta,\gamma) \in \mathbb{R}^3$, where $\gamma = 0$ if $n\not \in\{1,2,4\}$, there exists some $\tau>0$ depending on $n$ and continuously on $(\alpha, \beta,\gamma)$ such that there is a unique solution of (\refeq{eq_xiSU2})--(\refeq{eq_LiSU2}) satisfying (\refeq{eq_alpha}), (\refeq{eq_beta}) and (\refeq{eq_gamma}) whose evaluation at $t=\tau$ coincides with $\Theta_n(\alpha,\beta,\gamma)$;
    \item the corresponding functions $(f_1,f_2,f_3,u')$ satisfy the smoothness conditions in (\refeq{eq_u_smoothness}) and (\refeq{eq_boundarynonzeroSU2});
    \item if $\alpha=0$, then $\xi = L_1 + L_2 + L_3$ and the corresponding metric is Einstein 
    \begin{itemize} 
        \item Moreover if $\alpha=0$, $n=2$ and $\gamma = \epsilon{\lambda}/{4}$ for some $\epsilon \in \{\pm 1\}$, then $L_1+L_j = 2R_k$, $L_1-L_j=2(R_j-R_1)$ and the metric is K\"ahler, where $j=2$ and $k=3$ when $\epsilon=1$ and vise versa if $\epsilon = -1$;
    \end{itemize}
    \item if $\beta=0$, then $R_1=0$; and
    \item if $\gamma=0$, then $L_2=L_3$ and $R_2=R_3$ and the corresponding metric is $\mathrm{U}(2)$-invariant. 
    \begin{itemize} 
        \item Moreover if $\gamma=0$ and $(4-2\epsilon n)\beta=n\lambda$ for some $\epsilon \in \{\pm 1\}$, then $L_2 = \epsilon R_1$, $\xi = L_1 + 2\epsilon R_1+{\alpha}/(n R_2)$ and the corresponding metric is K\"ahler.
    \end{itemize}
\end{itemize}
\end{lemma}

\begin{proof}
    Proceed as in the proof of Lemma \refeq{lem_fixedSTE}. Details can be found in Theorem 4 and its proof in \cite{Buttsworth2025}. We note that the eigenvalues of the $t^{-1}$ terms in the desingularised system for (\refeq{eq_xiSU2})--(\refeq{eq_LiSU2}) are, $-1-{4}/{n},-1+{4}/{n},-2,-1,-1,1$ and $1$. Thus, the lowest order difference between $L_2$ and $L_3$ is of order ${4}/{n}-1$ and is parameterized by $\gamma$ in (\refeq{eq_gamma}). 
    
    It remains to prove the K\"ahler quantities are indeed conserved. The $\mathrm{U}(2)$-invariant K\"ahler condition $L_2 = R_1$ is proven exactly as in Lemma \refeq{lem_fixedSTE} when $(4-2n)\beta = n \lambda$. Similarly, the condition $L_2 = - R_1$ is shown to be conserved if $(4+2n)\beta = n \lambda$ by instead replacing $L_1$ with $-R_1$. This gives $Z'=2R_1 Z$ in the corresponding system when $Z = - \lambda + R_1(\xi + 4 R_1+L_1)$, which yields the desired result by uniqueness as before. The K\"ahler Einstein family for $n=2$ also proceeds similarly --- consider the system:
    \begin{equation}
        \label{eq_SU2_kahler1}
        \begin{aligned}
        L_i' &= -(L_3+2R_3) L_i + 2R_i^2-2(R_j-R_k)^2-\lambda; \\
        R_1' &= R_1 (L_1-L_2-L_3);  \\
        R_2' &= R_2(L_2-L_1-L_3); \text{ and} \\
        R_3' &= R_3 (L_3-2R_3).
        \end{aligned}
    \end{equation}
    We note that the system (\refeq{eq_SU2_kahler1}) is the Einstein subsystem of (\refeq{eq_xiSU2})--(\refeq{eq_LiSU2}) where $L_1 + L_2 = 2 R_3$. Set $\overline{Z} = 2\lambda + 4R_3(L_3-R_1-R_2+R_3)$ and $\overline{Y} = 2R_1-2R_2+L_1-L_2$. In the system (\refeq{eq_SU2_kahler1}), we compute 
    \[\overline{Z}' = -2R_3\overline{Z}+2(R_1-R_2)\overline{Y} \; \text{ and } \; \overline{Y}' = \overline{Y}(2R_1+2R_2-L_3-2R_3).\]
    If a solution that satisfies (\refeq{eq_xiSU2})--(\refeq{eq_LiSU2}) and boundary conditions conditions corresponding to (\refeq{eq_u_smoothness}) and (\refeq{eq_boundarynonzeroSU2}) also satisfies (\refeq{eq_SU2_kahler1}), then $\overline{Y}'(t) =  O(t)\cdot \overline{Y}(t)$. Thus,  when $n=2$, a solution of this form has $\overline{Y}(0) = 0$ (see Table \refeq{tab_compact_smoothness_approx}) and consequently $\overline{Y}$ vanishes. In particular, when $n=2$, $\alpha = 0$ and $\gamma = {\lambda}/{4}$ we compute $\overline{Z}(0) = 0$ (again see Table \refeq{tab_compact_smoothness_approx}) which in turn implies $\overline{Z}$ vanishes too.
    
    If the K\"ahler Einstein conditions $L_1 + L_2 = 2 R_3$ and $L_1-L_2=2R_2-2R_1$ hold, then $\overline{Z} = 2R_3'-L_1'-L_2'$. Therefore, as in the proof of Lemma \refeq{lem_fixedSTE}, the uniqueness of solutions to ODEs proves that $\alpha = 0$ and $\gamma = {\lambda}/{4}$ imply that $2R_3-L_1-L_2 = 0$ if $n=2$. The other equality $2R_2-2R_1 = L_1-L_2$ then holds immediately as $\overline{Y}$ vanishes in (\refeq{eq_SU2_kahler1}). 
    
    Finally, we note that this also holds for $\gamma = -{1}/{4}$ after interchanging the subscripts 2 and 3.      
\end{proof}

\begin{remark}
    Following the work in \cite{DancerStrachan1994} ,one finds that the K\"ahler Ricci solitons are exactly those that satisfy 
    \begin{equation}
        \label{eq_Kahler_necessary}
        L_i+L_j = \epsilon R_k \; \text{ and } \; L_i-L_j = \epsilon 2(R_j-R_i),
    \end{equation}
    for some $i,j,k \in \{1,2,3\}$ $i\neq j\neq k \neq i$ and $\epsilon \in \{\pm1\}$. A metric of this form is K\"ahler with respect to the almost complex structure:
    \[J\left(\frac{\partial}{\partial t}\right) = -\epsilon \frac{1}{f_k}X_k, \; \; J\left(-\epsilon\frac{1}{f_k}X_k\right) = -\frac{\partial}{\partial t}, \; \; J\left(\frac{1}{f_i}X_i\right) = \frac{1}{f_j}X_j \; \text{ and } \; J\left(\frac{1}{f_j} X_j\right) = -\frac{1}{f_i}X_i, \]
    where $\epsilon$ changes sign if $j \neq i+1 \mod 3$. The condition $L_i-L_j = \epsilon 2(R_j-R_i)$ is equivalent to the integrability of the above almost complex structure and the other condition is found via K\"ahlerity. This is what we implicitly use to state that the metrics are K\"ahler when (\refeq{eq_Kahler_necessary}) is true in Lemmas \refeq{lem_fixedSTE} and \refeq{lem_SU2STE}. Moreover, writing low order approximations of (\refeq{eq_Kahler_necessary}) in terms of $a_1,a_2,a_3$ or $\alpha,\beta,\gamma$ (see (\refeq{eq_fixed_approx}) and Table \refeq{tab_compact_smoothness_approx}), reveals that the $\mathrm{SU}(2)$-invariant K\"ahler Ricci solitons are exactly the ones listed in Lemma \refeq{lem_fixedSTE} and \refeq{lem_SU2STE}. The initial conditions for K\"ahler metrics in Lemma \refeq{lem_fixedSTE} and \refeq{lem_SU2STE} immediately imply that cohomogeneity one $4$-dimensional $\mathrm{SU}(2)$-invariant gradient K\"ahler Ricci solitons:
    \begin{enumerate}
        \item are $\mathrm{U}(2)$-invariant unless they are on $\mathcal{O}(-2)$ or $M_2$; 
        \item that are steady can only be found on $\mathbb{C}^2$ and $\mathcal{O}(-2)$; and
        \item that are on a compact space can only be found on $\mathbb{CP}^2$, invariant under the fixed point action, or on $M_1$, invariant under the action with $\mathbb{S}^3$ regular orbits.
    \end{enumerate}
    It should be noted that, statement 1. was known in the Einstein case (see \cite{DancerStrachan1994}), statement 2. was known for $\mathrm{U}(2)$ actions (see Lemma 1.2 of \cite{FIK2003}) and statement 3. is trivial as all shrinking gradient K\"ahler Ricci solitons are known.
\end{remark}

\begin{corollary}\label{cor_EAU}
    For every $\lambda \in \mathbb{R}$, $n\in \mathbb{Z}_{>0}$ and $(a_1,a_2,a_3),(\alpha,\beta,\gamma) \in \mathbb{R}^3$ such that $\beta>0$ and $\gamma=0$ if $n \not \in \{1,2,4\}$, there exists a maximal $\overline{\tau} \in (0,\infty]$ and a unique smooth function $\varphi: (0,\overline{\tau}) \to \mathbb{R}^7$, $\varphi = (\xi,L_1,L_2,L_3,R_1,R_2,R_3)$ satisfying (\refeq{eq_xiSU2})--(\refeq{eq_LiSU2}) and either (\refeq{eq_a_i}) or all of (\refeq{eq_alpha}),(\refeq{eq_beta})--(\refeq{eq_gamma}). These functions construct a smooth metric on $\bigcup \{G\cdot c(t): t \in [0,\overline{\tau})\} \subseteq M$. Moreover, if $\overline{\tau}=\infty$ the corresponding metric is on all of $M$ and is complete.
\end{corollary}

\begin{proof}
    The existence and uniqueness of a maximal solution follows from Lemmas \refeq{lem_fixedSTE} and \refeq{lem_SU2STE} and existence and uniqueness to solutions of ODEs. If $\overline{\tau} = \infty$, then any Cauchy sequence in $(M,dt^2+g_t^2)$ is contained in the preimage of $[0,t_0]$ by projection onto the orbit space for sufficiently large $t_0$ and the result follows from boundedness of $f_i|_{[0,t_0]}^2 = (R_j R_k)^{-1}|_{[0,t_0]}$ for all $i \in \{1,2,3\}$.
\end{proof}

\begin{corollary}\label{cor_param}
    Every complete cohomogeneity one $\mathrm{SU}(2)$-invariant expanding gradient Ricci soliton $g$ on a $4$-manifold $M$ can be written in the form (\refeq{eq_metricfunctions}) that satisfies (\refeq{eq_xiSU2})--(\refeq{eq_LiSU2}) and is uniquely determined up to rescaling by the condition $\lambda = -1$ and the parameters 
    \begin{itemize}
        \item $a_1 \geq a_2 \geq a_3 \geq -{\lambda}/{3}-a_1-a_2$ if $M$ is diffeomorphic to $\mathbb{R}^4$;
        \item $\beta >0$,  $\alpha \geq 0$ and $\gamma \geq 0$ if $M$ is diffeomorphic to $\mathcal{O}(-1),\mathcal{O}(-2)$ or $\mathcal{O}(-4)$; or 
        \item $\beta >0$ and $\alpha \geq 0$ if $M$ is diffeomorphic to $\mathcal{O}(-n), n \not \in \{1,2,4\}$.
    \end{itemize}
\end{corollary}

\begin{proof}
    Recalling that $\alpha = -\lambda - 3a_1-3a_2-3a_3$, the only non-trivial statement that is yet to be addressed that $\alpha \geq 0$, which is implied by Proposition 1.11 of \cite{BDGW2015}.
\end{proof}

\begin{remark}
    One easily achieves a similar corollary for compact solitons, whose metrics are uniquely determined by the parameters at both ends. We present this in Lemma \refeq{lem_compactconditions}, but delay its statement until it is relevant. Also, while we are not interested in steady solitons here, one can also obtain a similar corollary for $\lambda=0$ after normalizing parameters (since solutions are not yet determined up to rescaling).
\end{remark}

\begin{remark}
    Not all metrics whose parameters satisfy the conditions of Corollary \refeq{cor_param} are complete.
\end{remark}

\section{$\mathrm{SU}(2)$-invariant expanding solitons}
In this section, we choose convenient parameters from (\refeq{eq_a_i}), (\refeq{eq_alpha}) and (\refeq{eq_beta})--(\refeq{eq_gamma}) to find both Einstein and asymptotically conical solutions to the system (\refeq{eq_xiSU2})--(\refeq{eq_LiSU2}) that exist for all $t>0$. In particular, we select parameters in the closure of the parameter spaces given in Corollary \refeq{cor_param} (with respect to the vector space they live within) such that after perturbation we produce the complete gradient Ricci solitons in Theorem \refeq{thm_expanders}. Theorem \refeq{thm_expanders} is then proven for the Einstein and non-Einstein cases individually. Our choice of parameters are $\gamma=\beta=0$ on $\mathcal{O}(-n)$ and $a_1=a_2=a_3$ on $\mathbb{R}^4$ both with any $\alpha \geq 0$. For these choices of parameters, any $\alpha>0$ produces an asymptotically conical solution and $\alpha = 0$ produces an Einstein solution.

\subsection{Degenerate $\mathrm{U}(2)$-invariant solutions on $\mathcal{O}(-n)$}
Before we find solutions of (\refeq{eq_xiSU2})--(\refeq{eq_LiSU2}), we again remark that that the parameters with $\beta = 0$ are not contained in the parameter spaces given in Corollary \refeq{cor_param}, but are instead on the boundary. Indeed, when $\beta \leq 0$, the solution to (\refeq{eq_xiSU2})--(\refeq{eq_LiSU2}) will not correspond to a Riemannian metric (see (\refeq{eq_beta})). However, a family of metrics can still be found about $\gamma = \beta = 0$ after perturbation.

By Lemma \refeq{lem_SU2STE}, $\gamma = 0$ implies the solution is $\mathrm{U}(2)$-invariant, and if in addition, $\beta = 0$, then the solution is uniquely determined by $\alpha$ as in the proof of Lemma \refeq{lem_SU2STE}. We begin with the characterization of $\beta = 0$ as per Lemma \refeq{lem_SU2STE}. As the ODE for $R_1$ in (\refeq{eq_RiSU2}) is of the form
$R_1' = \left({1}/{t}+O(t)\right)R_1$, $\beta = 0$ implies $R_1(0) = 0$ and $R_1$ vanishes uniformly. The $\mathrm{U}(2)$-invariant system 
\begin{equation}
    \label{eq_U2ODEs}
    \begin{aligned}
    \xi' &= -L_1^2-2L_2^2 - \lambda;\\
    L_1' &= -\xi L_1+2R_1^2 - \lambda;\\
    L_2' &= -\xi L_2+4R_1R_2 - 2R_1^2 - \lambda;\\
    R_1' &= R_1(L_1-2L_2); \text{ and}\\
    R_2' &= -R_2L_1,
    \end{aligned}
\end{equation}
therefore splits leaving us to analyze the initial value problem:
\begin{alignat}{1} 
    \xi' &= 1-L_1^2-2L_2^2; \text{ and} \label{eq_reducedxi} \\ 
    L_i' &= 1-\xi L_i \; \; \text{for} \; i \in \{1,2\} \label{eq_reducedLi}
\end{alignat}
subject to the initial conditions
\begin{equation} 
    \label{eq_reducedbdy1}
    \begin{aligned}
        \lim_{t \to 0} \left(\xi(t) - \frac{1}{t}\right) = 0; \hspace{1.4cm}  &\lim_{t \to 0} \left(L_1(t) - \frac{1}{t}\right) = 0; \hspace{1.7cm} L_2(0) = 0;\\
        \lim_{t \to 0} \left(\xi(t) - \frac{1}{t}\right)' = \frac{2\alpha}{3} + 1; \; \; &\lim_{t \to 0} \left(L_1(t) - \frac{1}{t}\right)' = -\frac{\alpha}{3}; \; \text{ and } \; L_2'(0) = \frac{1}{2}.
    \end{aligned}
\end{equation}

\begin{lemma}\label{lem_alpha_monotone_n}
    Let $L_1^{\alpha},L_2^{\alpha}$ and $\xi^{\alpha}$ denote the solution to (\refeq{eq_reducedxi})--(\refeq{eq_reducedbdy1}) for fixed $\alpha \in \mathbb{R}$. If $\alpha_1 > \alpha_0$, then $0<L_i^{\alpha_1}(t)<L_i^{\alpha_0}(t)$ for $i=1,2$ and $\xi^{\alpha_1}(t)>\xi^{\alpha_0}(t)$ for all $t \in (0,\infty)$.
\end{lemma}

\begin{proof}
Note that $L_i' > L_i\xi$ and the initial conditions in (\refeq{eq_reducedbdy1}) imply $L_i$ is positive on $(0,\infty)$. 
Approximating $\xi$ and $L_i$ to higher orders about $0$ one computes:
\begin{align*}
(\xi^{\alpha_1}-\xi^{\alpha_0})(t) &= \frac{2}{3}(\alpha_1-\alpha_0)t+O(t^2); \\
(L_1^{\alpha_0}-L_1^{\alpha_1})(t) &= \frac{1}{3  }(\alpha_1-\alpha_0)t +O(t^2); \text{ and}\\
(L_2^{\alpha_0}-L_2^{\alpha_1})(t) &= \frac{1}{12}(\alpha_1-\alpha_0)t^3 + O(t^4).
\end{align*}
In particular, for any sufficiently small $t>0$, $(L_i^{\alpha_0}-L_i^{\alpha_1})(t), (\xi^{\alpha_1}-\xi^{\alpha_0})(t) >0$.
Comparing solutions of (\refeq{eq_reducedxi})--(\refeq{eq_reducedLi}) for $\alpha_1$ and $\alpha_0$ yields
\begin{align*}
    (L_i^{\alpha_0}-L_i^{\alpha_1})' &= L_i^{\alpha_1}(\xi^{\alpha_1}-\xi^{\alpha_0}) - \xi^{\alpha_0}(L_i^{\alpha_0}-L_i^{\alpha_1}) \text{ and } \\
    (\xi^{\alpha_1}-\xi^{\alpha_0})' &= (L_1^{\alpha_0}+L_1^{\alpha_1})(L_1^{\alpha_0}-L_1^{\alpha_1})+2(L_2^{\alpha_0}+L_2^{\alpha_1})(L_2^{\alpha_0}-L_2^{\alpha_1}),
\end{align*}
from which the result immediately follows.
\end{proof}

\begin{lemma}\label{lem_n_dynamics}
    Let $L_1^{\alpha},L_2^{\alpha}$ and $\xi^{\alpha}$ denote the solution to (\refeq{eq_reducedxi})--(\refeq{eq_reducedbdy1}) for fixed $\alpha \in \mathbb{R}$.
    If $\alpha = 0$, then 
    \[\xi^0(t) = \sqrt{3}\left(\frac{e^{2\sqrt{3}t}+1}{e^{2\sqrt{3}t}-1}\right), \; \; L_1^0(t) = \frac{1}{\sqrt{3}}\left(\frac{e^{2\sqrt{3}t}+4e^{\sqrt{3}t}+1}{e^{2\sqrt{3}t}-1}\right) \; \; \text{ and } \; \; L_2^0(t) = \frac{1}{\sqrt{3}}\left(\frac{e^{\sqrt{3}t}-1}{e^{\sqrt{3}t}+1}\right).\]
    If $\alpha>0$, then
    \begin{alignat}{1}
        \lim_{t \to \infty} \frac{\xi^{\alpha}(t)}{t} &= 1;  \label{eq_asxi} \\
        \lim_{t \to \infty} L_i^{\alpha}(t) t &= 1; \text{ and} \label{eq_asLi} \\
        \lim_{t \to \infty} R_i^{\alpha}(t) t &< \infty \label{eq_asRi}
    \end{alignat}
    for all $i \in {1,2}$.
\end{lemma}

\begin{proof}
    By Lemma \refeq{lem_SU2STE}, $\alpha = 0$ implies $\xi = L_1 + 2L_2$ and the system is reduced to solving the ODE: $2{L_2^0}'=1-3{(L_2^0)}^2$, $L_2^0(0)=0$ which has the solution given in the lemma. We note that the Einstein solution approaches the only critical point with $L_i \geq 0$, at $L_i = {1}/{\sqrt{3}}$ and $\xi = \sqrt{3}$. From the proof of Lemma \refeq{lem_alpha_monotone_n}, we note that if $\alpha_1 > \alpha_0$ then $\xi^{\alpha_1}-\xi^{\alpha_0}$ is strictly increasing. Comparing $\xi^\alpha$ and $L_1^\alpha$ to the Einstein solution, (\refeq{eq_reducedLi}) implies that for all $\alpha>0$, and sufficiently large $t_0$, $L_1^\alpha(t_0)< {1}/{\sqrt{3}}$. As $L_i^\alpha$ is strictly bounded below ${1}/{\sqrt{3}}$ after finite time, $\xi$ is monotone increasing and the asymptotics (\refeq{eq_asxi}) and (\refeq{eq_asLi}) are now easily verified. Finally, (\refeq{eq_asLi}) and $R_2'=-R_2L_1$ imply (\refeq{eq_asRi}).
\end{proof}

\subsection{$\mathrm{SO}(4)$-invariant solitons on $\mathbb{R}^4$}
For completeness, we now recover the one-parameter family of non-trivial expanding $\mathrm{SO}(4)$- invariant solitons on $\mathbb{R}^4$ found by Bryant and described in \cite{CCG2007}. We give an alternate description of their construction in this section such that Theorem \refeq{thm_expanders} is self contained within this work. Moreover, the work done above allows us to recreate these metrics with two simple lemmas.

By Lemma \refeq{lem_fixedSTE}, $a_1=a_2=a_3=(1-\alpha)/{9}$ implies $f_1=f_2=f_3$ and hence the $\mathrm{SU}(2)$-invariant metrics we consider on $\mathbb{R}^4$ are the $\mathrm{SO}(4)$-invariant metrics $dt^2 + {f(t)}^2ds^2$ where $ds^2$ denotes the round metric on $\mathbb{S}^3$ of radius $1$. Therefore, $\mathrm{SO}(4)$-invariance implies equations (\refeq{eq_xiSU2})--(\refeq{eq_LiSU2}) become 
\begin{alignat}{1}
    \xi' &= 1-3L^2; \label{eq_xifixed} \\
    R' &= -RL; \text{ and } \label{eq_Rifixed} \\
    L' &= 1-\xi L+2R^2, \label{eq_Lifixed}
\end{alignat}
with initial conditions
\begin{equation} 
    \label{eq_fixed0}
    \begin{aligned}
    \lim_{t \to 0} \left(\xi(t) - \frac{3}{t}\right) = 0, \; \; & \lim_{t \to 0} \frac{d}{dt}\left(\xi(t) - \frac{3}{t}\right) = \frac{1+2\alpha}{3}; \\
    \lim_{t \to 0} \left(L(t) - \frac{1}{t}\right) = 0, \; \; & \lim_{t \to 0} \frac{d}{dt}\left(L(t) - \frac{1}{t}\right) = \frac{1-\alpha}{9}; \text{ and } \\
    \lim_{t \to 0} \left(R(t) - \frac{1}{t}\right) = 0, \; \; & \lim_{t \to 0} \frac{d}{dt}\left(R(t) - \frac{1}{t}\right) = \frac{\alpha-1}{18}. 
    \end{aligned}
\end{equation} 
Using the same notation as in the previous subsection, we now set $\xi^{\alpha},L^{\alpha}$ and $R^{\alpha}$ as the solution to (\refeq{eq_xifixed})--(\refeq{eq_fixed0}) for fixed $\alpha \in \mathbb{R}$. If $\alpha = 0$, we recover the hyperbolic metric:
\[\xi^0(t) = \sqrt{3}\left(\frac{e^{2t/\sqrt{3}}+1}{e^{2t/\sqrt{3}}-1}\right), \; \; L^0(t) = \frac{1}{\sqrt{3}}\left(\frac{e^{2t/\sqrt{3}}+1}{e^{2t/\sqrt{3}}-1}\right) \; \; \text{ and } \; \; R^0(t) = \frac{2}{\sqrt{3}}\left(\frac{e^{t/\sqrt{3}}}{e^{2t/\sqrt{3}}-1}\right).\]
If $\alpha = 1$, we recover the Gaussian expander on $\mathbb{R}^4$:
\[\xi^1(t) = t+\frac{3}{t}, \; \; \text{ and } \; \; L^1(t) = \frac{1}{t} = R^1(t).\]

Now that the boundary value problem is set up, we quickly reconstruct the complete asymptotically conical expanding $\mathrm{SO}(4)$-invariant solitons (and they are complete via Corollary \refeq{cor_EAU}).

\begin{lemma}\label{lem_alpha_monotone_fixed}
    Let $\xi^{\alpha},L^{\alpha}$ and $R^{\alpha}$ denote the solution to (\refeq{eq_xifixed})--(\refeq{eq_fixed0}) for fixed $\alpha \in \mathbb{R}$. If $\alpha_1 > \alpha_0$ and $\alpha_0 \leq 1$, then $\xi^{\alpha_1}(t)>\xi^{\alpha_0}(t)$, $0<L^{\alpha_1}(t)<L^{\alpha_0}(t)$ and $R^{\alpha_1}(t)>R^{\alpha_0}(t)>0$ for all $t \in (0,\infty)$.
\end{lemma}

\begin{proof}
As in the proof of Lemma \refeq{lem_alpha_monotone_n}, $L>0$, $R>0$ and the lemma holds for sufficiently small $t>0$ by (\refeq{eq_xifixed})--(\refeq{eq_fixed0}). Moreover, the equations
\begin{align}
    (\xi^{\alpha_1}-\xi^{\alpha_0})' &= 3(L^{\alpha_0}+L^{\alpha_1})(L^{\alpha_0}-L^{\alpha_1}), \nonumber \\
    (R^{\alpha_1}-R^{\alpha_0})' &= R^{\alpha_1}(L^{\alpha_0}-L^{\alpha_1}) - L^{\alpha_0}(R^{\alpha_1}-R^{\alpha_0}) \text{ and} \nonumber\\
    (L^{\alpha_0}-L^{\alpha_1})' &= L^{\alpha_0}(\xi^{\alpha_1}-\xi^{\alpha_0})-2(R^{\alpha_1}+R^{\alpha_0})(R^{\alpha_1}-R^{\alpha_0})-\xi^{\alpha_1}(L^{\alpha_0}-L^{\alpha_1}) \label{eq_lemma_pf_0}
\end{align}
imply that if the lemma fails, then the first obstruction (in $t$) must be $L^{\alpha_1} \not < L^{\alpha_0}$. We may therefore assume the lemma holds on $(0,t_0)$, $\xi^{\alpha_1}(t_0)>\xi^{\alpha_0}(t_0)$, $R^{\alpha_1}(t_0)>R^{\alpha_0}(t_0)$ and $L^{\alpha_1}(t_0) \leq L^{\alpha_0}(t_0)$, where it suffices to prove $L^{\alpha_1}(t_0) < L^{\alpha_0}(t_0)$.

\textbf{Step 1: $\alpha_1 = 1$.} If $\alpha_1 = 1$, then $L^{\alpha_0}(t) \geq L^1(t) = R^1(t) > R^{\alpha_0}(t)$ for all $t \in (0,t_0]$. Consider the function
\[\hat{h^{\alpha}}(t) := 1-\xi^{\alpha}(t)L^{\alpha}(t)+2{R^{\alpha}(t)}^2+{L^{\alpha}(t)}^2 = {L^{\alpha}}'(t)+{L^{\alpha}(t)}^2.\] 
We note that $\hat{h^{\alpha}}(0) = (1-\alpha)/3$ and
\begin{equation}
    \label{eq_lemma_pf_2}
    \hat{h^{\alpha}}' = (3L^{\alpha}-\xi^{\alpha})(\hat{h^{\alpha}}) + 2L^{\alpha}({L^{\alpha}}+{R^{\alpha}})({L^{\alpha}}-{R^{\alpha}}).
\end{equation} 
In particular, $\hat{h^{\alpha_0}}(0) > 0$ and (\refeq{eq_lemma_pf_2}) implies $\hat{h^{\alpha_0}}(t) = {L^{\alpha_0}}'+{L^{\alpha_0}}^2 > 0$ for all $t \in [0,t_0]$ so $L^{\alpha_0}(t) < {1}/{t} = L^{\alpha_1}(t)$ for all $t \in (0,t_0]$ and the lemma is proven when $\alpha_1 = 1$. This also proves that $L^{\alpha}(t) \leq {1}/{t}$ for all $\alpha \leq 1$.

\textbf{Step 2: $\alpha_1 \leq 1$.} We now assume $\alpha_1 \leq 1$. Equation (\refeq{eq_lemma_pf_0}) implies it is sufficient to prove $\hat{f}(t)\geq 0$ for all $t \in (0,t_0)$ where 
\[\hat{f}(t) := L^{\alpha_0}(t)(\xi^{\alpha_1}(t)-\xi^{\alpha_0}(t))-2{{R^{\alpha_1}}(t)}^2+2{{R^{\alpha_0}}(t)}^2.\]
For sufficiently small $t>0$, about $0$, $\hat{f}(t) = \frac{4}{9}(\alpha_1-\alpha_0)/t+O(t) > 0$. We then compute
\begin{multline*} \hat{f}'(t) = -2L^{\alpha_0}\hat{f}(t)+(\hat{h^{\alpha_0}}(t)+{L^{\alpha_0}}^2)(\xi^{\alpha_1}-\xi^{\alpha_0}) \\
+(L^{\alpha_0}-L^{\alpha_1})(3L^{\alpha_0}(L^{\alpha_0}+L^{\alpha_1})-4{R^{\alpha_1}}^2)+2(\xi^{\alpha_1}-\xi^{\alpha_0}){R^{\alpha_0}}^2.
\end{multline*}
Using the $\alpha_1 = 1$ case proved above, the inequalities $\alpha_1,\alpha_0 \leq 1$ imply that $L^{\alpha_0}(t) \geq L^{\alpha_1}(t) \geq {1}/{t} \geq R^{\alpha_1}(t)$ for all $t \in (0,t_0]$ so 
\[3L^{\alpha_0}(t)(L^{\alpha_0}(t)+L^{\alpha_1}(t)) - 4{R^{\alpha_1}(t)}^2 \geq \frac{2}{t^2} >0.\]
Therefore, on $(0,t_0)$, $\hat{f}'(t) \geq -2L^{\alpha_0}(t)\hat{f}(t)$ and hence $\hat{f}(t) \geq 0$.

\textbf{Step 3: $\alpha_1 > 1$.} Finally, we move to the case $\alpha_1 > 1$. We can assume without loss of generality that $\alpha_0 = 1$ using step 2. If $\alpha_1>1$, then $L^{\alpha_1}(t) \leq {1}/{t} < R^{\alpha_1}(t)$ on $(0,t_0]$, $h^{\alpha_1}(0)<0$ and (\refeq{eq_lemma_pf_2}) implies $h^{\alpha_1}(t)<0$ for all $t \in (0,t_0)$. This gives ${L^{\alpha_1}}'(t) < -L^{\alpha_1}(t)^2$ for all $t \in (0,t_0)$ hence $L^{\alpha_1}(t_0) < {1}/{t_0} = L^{\alpha_0}(t_0)$.
\end{proof}

\begin{lemma}\label{lem_fixed_dynamics}
    Let $L_1^{\alpha},L_2^{\alpha}$ and $\xi^{\alpha}$ denote the solution to (\refeq{eq_xifixed})--(\refeq{eq_fixed0}) for fixed $\alpha \in \mathbb{R}$.
    If $\alpha > 0$, then (\refeq{eq_asxi})--(\refeq{eq_asRi}) are satisfied. 
\end{lemma}

\begin{proof}
    We note that the hyperbolic metric functions approach the only critical point with $L \geq 0$ at $L = {1}/{\sqrt{3}}$, $R = 0$ and $\xi = \sqrt{3}$. We note that if $0 < \alpha \leq 1$, Lemma \refeq{lem_alpha_monotone_fixed} and its proof imply that $R^\alpha(t) \leq {1}/{t}$ and $\xi^\alpha-\xi^{0}$ is strictly increasing. Therefore, comparing $\xi^\alpha$, $L^\alpha$ and $R^\alpha$ with the hyperbolic metric functions, (\refeq{eq_Lifixed}) implies that for all $\alpha \in (0,1)$, and sufficiently large $t_0$, $L^\alpha(t_0)< {1}/{\sqrt{3}}$. As $L^\alpha$ is strictly bounded below ${1}/{\sqrt{3}}$ for sufficient large $t$, $\xi^\alpha$ is monotone increasing here and both (\refeq{eq_asxi}) and (\refeq{eq_asLi}) are thereby easily validated with (\refeq{eq_xifixed}) and (\refeq{eq_Lifixed}). If $\alpha > 1$, by comparison with the Gaussian shrinker, $L^\alpha(t) < {1}/{t}$ and $\xi^\alpha(t) > t$ so (\refeq{eq_asxi}) and (\refeq{eq_asLi}) are again easily validated. Finally, we note that (\refeq{eq_asLi}) and (\refeq{eq_Rifixed}) imply (\refeq{eq_asRi}).
\end{proof}

\subsection{Einstein metrics}
In this subsection we produce the Einstein metrics in Theorem \refeq{thm_expanders}. However, we first recall some important relevant Einstein metrics, some of their properties and write them in terms of the parameters in \S 3.

Lemmas \refeq{lem_fixedSTE} and \refeq{lem_SU2STE} can be used to show that $\mathrm{U}(2)$-invariant Einstein manifolds satisfy
\begin{equation}
    \label{eq_PZ}
    f_1^2 = \frac{f_2^2(f_2')^2}{1+cf_2^2}
\end{equation} for some $c \in \mathbb{R}$, where K\"ahlerity is equivalent to $c=0$; an explicit calculation can be found in \cite{PetersenZhu1995}. The Einstein manifolds with Einstein constant $\lambda = -1$ are complete if and only if $c \geq 0$ \cite{PetersenZhu1995}. The complete non-K\"ahler metrics can be shown to satisfy the asymptotics $L_i \to {1}/{\sqrt{3}}$ and $R_i \to 0$ for $i\in\{1,2,3\}$ as $t \to \infty$, while the K\"ahler metrics achieve the only other $\mathrm{U}(2)$-invariant Einstein critical point which is unstable --- containing a 1 dimensional unstable manifold in both the $\mathrm{U}(2)$ and $\mathrm{SU}(2)$-invariant systems.

Writing $c$ in terms of $a_1,a_2,a_3$, one notes that the complete non-K\"ahler $\mathrm{U}(2)$-invariant Einstein metrics on $\mathbb{R}^4$ correspond to the one-parameter family ${1}/{18} < a_2 = a_3 = {1}/{3}-a_1$ while the K\"ahler Einstein metric is given by $a_1={2}/{9}$ and $a_2 = a_3 = {1}/{18}$, with metric functions
\[f_1(t) = \sqrt{6}\sinh\left(\frac{t}{\sqrt{6}}\right)\cosh\left(\frac{t}{\sqrt{6}}\right) \; \; \text{ and } \; \; f_2(t) = f_3(t) = \sqrt{6}\sinh\left(\frac{t}{\sqrt{6}}\right).\]

On $\mathcal{O}(-n)$, equation (\refeq{eq_PZ}) now implies 
\[c = \frac{1}{4n^3\beta}\left(2(2+n)\beta + n\right)\left(2(2-n)\beta +n\right).\]
In particular, if $n \in \{1,2\}$ then $c > 0$ and all $\mathrm{U}(2)$-invariant Einstein metrics $\alpha = 0 = \gamma, \; \beta >0$ are complete. If $n \geq 3$, the complete non-K\"ahler $\mathrm{U}(2)$-invariant Einstein metrics are given by $\alpha = 0 = \gamma$ and $\beta < {n}/{(2n-4)}$, while $\beta = {n}/{(2n-4)}$ gives the K\"ahler Einstein metrics on $\mathcal{O}(-n)$ found by \cite{GibbonsPope1979} and \cite{Pedersen1985}.

If $n \not \in \{1,2,4\}$, $\mathrm{SU}(2)$-invariant metrics on $\mathcal{O}(-n)$ are necessarily $\mathrm{U}(2)$-invariant and all Einstein manifolds are hence already known --- given above. However, on $\mathbb{R}^4$ and $\mathcal{O}(-n)$ for $n \in \{1,2,4\}$, the $\mathrm{U}(2)$-invariant one-parameter families can be perturbed to construct $\mathrm{SU}(2)$-invariant but not $\mathrm{U}(2)$-invariant Einstein metrics. also a family of complete $\mathrm{SU}(2)$-invariant K\"ahler Einstein metrics corresponding to the parameters $\alpha = 0$, $\beta >0$ and $\gamma = \pm {1}/{4}$ (see Lemma \refeq{lem_SU2STE}), however these metrics have the aforementioned K\"ahler asymptotics allowing us to exclude them from the following lemma. All smooth $\mathrm{SU}(2)$-invariant K\"ahler Einstein metrics are known to be complete \cite{DancerStrachan1994}, so we focus on the non-K\"ahler setting.

\begin{lemma}\label{lem_einstein}
    There exists open sets $\Omega^E \subseteq \mathbb{R}_{\geq 0}^2$ and $\Omega_n^E \subseteq \mathbb{R}^+ \times \mathbb{R}_{\geq 0}$ for $n \in \{1,2,4\}$ such that:
    \begin{itemize}
        \item each $(a_1-a_2,a_2-a_3) \in \Omega^E$ and $(\beta,\gamma) \in \Omega_n^E$ produces a distinct $\mathrm{SU}(2)$-invariant complete Einstein metric on $\mathbb{R}^4$ and $\mathcal{O}(-n)$ respectively;
        \item \big($[0,{1}/{6}) \times \{0\}\big)\cup \big(\{0\} \times \mathbb{R}_{\geq 0} \big) \subseteq \Omega^E$ and $({1}/{6},0) \not \in \Omega^E$;
        \item $\mathbb{R}^+ \times \{0\} \subseteq \Omega_n^E$  for $n \in \{1,2\}$ and $(\mathbb{R}^+ \times \{{1}/{4}\}) \cap \Omega_2^E = \emptyset$; and 
        \item $(0,1) \times \{0\} \subseteq \Omega_4^E$ and $(1,0) \not \in \Omega_4^E$.
    \end{itemize}
\end{lemma}

\begin{proof}
    By Lemmas \refeq{lem_fixedSTE} and \refeq{lem_SU2STE}, if $\alpha = 0$, then (\refeq{eq_xiSU2}) is a first integral of the system (\refeq{eq_xiSU2})--(\refeq{eq_LiSU2}) and the solution is Einstein. Corollary \refeq{cor_param} implies the parameters satisfying $(a_1-a_2,a_2-a_3) \in \mathbb{R}_{\geq 0}^2$ and $(\beta,\gamma) \in \mathbb{R}^+\times \mathbb{R}_{\geq 0}$ are parameters to distinct (not necessarily complete) solutions of the Einstein equations (\refeq{eq_RiSU2})--(\refeq{eq_LiSU2}). We define $\Omega^E \subseteq \mathbb{R}_{\geq 0}^2$ and $\Omega_n^E \subseteq \mathbb{R}^+ \times \mathbb{R}_{\geq 0}$ as the set of parameters such that the corresponding solution to (\refeq{eq_RiSU2})--(\refeq{eq_LiSU2}) satisfies $L_i \to {1}/{\sqrt{3}}$ and $R_i \to 0$ for $i\in\{1,2,3\}$ as $t \to \infty$. Corollary \refeq{cor_EAU} implies the corresponding metrics are complete and by Lemma \refeq{lem_SU2STE} and the discussion above, it only remains to prove these sets are open which is trivial as the critical point $(\frac{1}{\sqrt{3}},\frac{1}{\sqrt{3}},\frac{1}{\sqrt{3}},0,0,0)$ is stable in (\refeq{eq_RiSU2})--(\refeq{eq_LiSU2}). Indeed, if $Z = (L_1 - \frac{1}{\sqrt{3}},L_2 - \frac{1}{\sqrt{3}},L_3 - \frac{1}{\sqrt{3}},R_1,R_2,R_3)$, then $Z' = AZ+B(Z)$ where $B(0) = 0$, $DB(0)=0$ and 
    \[ A = \frac{1}{\sqrt{3}}\left(\begin{matrix} 
        -4 & -1 & -1 & 0 & 0 & 0 \\
        -1 & -4 & -1 & 0 & 0 & 0 \\
        -1 & -1 & -4 & 1 & 0 & 0 \\
        0 & 0 & 0 & -1 & 0 & 0 \\
        0 & 0 & 0 & 0 & -1 & 0 \\
        0 & 0 & 0 & 0 & 0 & -1 \\
    \end{matrix}\right)\]
    which has eigenvalues $-2\sqrt{3},-\sqrt{3},-\sqrt{3},-{1}/{\sqrt{3}},-{1}/{\sqrt{3}},-{1}/{\sqrt{3}}<0$. The result then follows by the continuity of $\Theta_1$ and $\Theta_2$ in Lemmas \refeq{lem_fixedSTE} and \refeq{lem_SU2STE} and the flow for equations (\refeq{eq_RiSU2})--(\refeq{eq_LiSU2}).
\end{proof}

\begin{figure}[h!]
    \caption{Numerical simulation of $\Omega_\cdot^E$}
    \; \includegraphics[scale = 0.45]{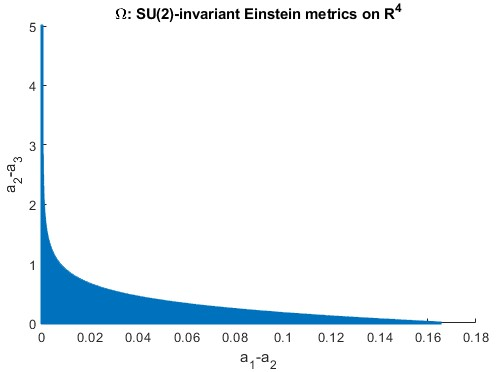}\includegraphics[scale = 0.45]{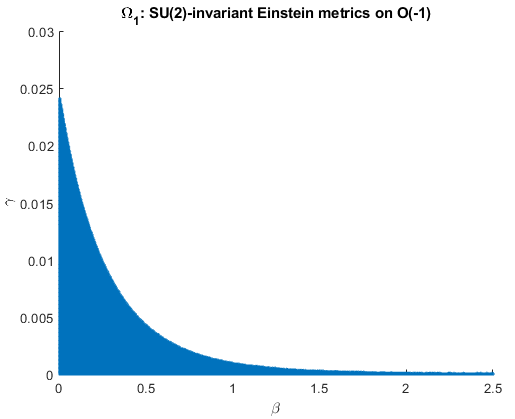}
    \includegraphics[scale = 0.45]{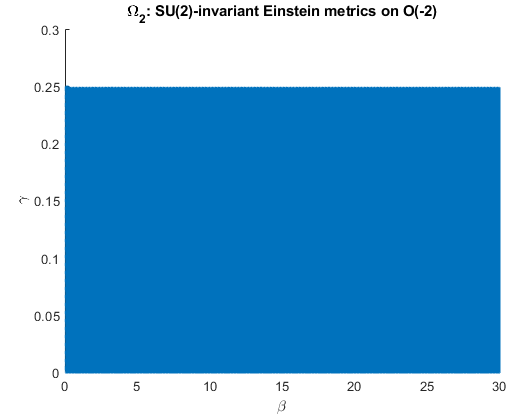}\includegraphics[scale = 0.45]{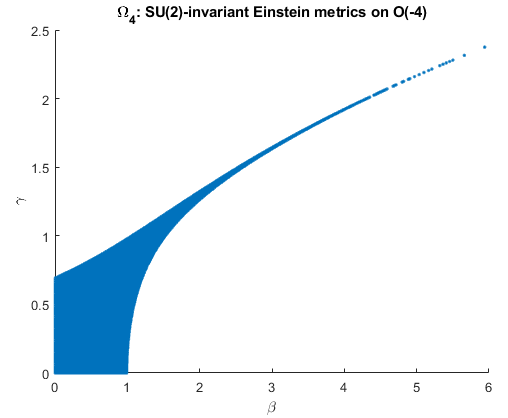}\label{fig_einstein_fixed_1_2_4}
\end{figure}

Using MATLAB, we numerically simulate the sets $\Omega^E$,$\Omega_1^E$,$\Omega_2^E$ and $\Omega_4^E$, as the shaded areas in Figure \refeq{fig_einstein_fixed_1_2_4}. The plots in Figure \refeq{fig_einstein_fixed_1_2_4} give a clearer picture of the space of complete Einstein metrics improving on the last 3 conditions in Lemma \refeq{lem_einstein}. In particular, we predict that $\Omega_2 = \mathbb{R}^+ \times (0,{1}/{4})$.

\subsection{Expanding non-Einstein solitons}
In this subsection we produce the non-Einstein metrics in Theorem \refeq{thm_expanders}. We again first recall some important relevant expanding gradient Ricci solitons and write them in terms of the parameters in \S 3.

Following the work in \cite{Cao1996,Cao1997,FIK2003}, one finds that K\"ahler $\mathrm{U}(2)$-invariant Ricci solitons on $I \times \mathrm{SU}(2)/\mathbb{Z}_n \cong I \times \mathbb{S}^3/\mathbb{Z}_n \cong (\mathbb{C}^2 \setminus \{0\})/\mathbb{Z}_n$ are of the form 
\begin{alignat*}{2}
    g &= \phi_r(\frac{1}{4}dr^2+\eta^2)+\phi\cdot \pi^*g_{FS} &&= \phi_r(\frac{1}{4}dr^2+\omega_i^2)+\phi(\omega_j^2+\omega_k^2)\\ 
    &= \; dt^2 + \frac{\phi_{t}^2}{4}\omega_i^2+\phi(\omega_j^2+\omega_k^2) &&= dt^2 + {(ff_t)}^2\omega_i^2+f^2(\omega_j^2+\omega_k^2)
\end{alignat*}
where $g_{FS}$ is the Fubini--Study metric on $\mathbb{CP}^1$, $\eta$ is a one form dual to the vertical fiber direction of the Hopf fibration, $\pi:\mathbb{S}^3 \to \mathbb{CP}^1$, $r = \log(|z|^2)$ and $\phi = P_r$ is the derivative of the $\mathrm{U}(2)$-invariant K\"ahler potential. In particular, the last equality implies K\"ahlerity for $\mathrm{U}(2)$-invariant gradient Ricci solitons still yields (\refeq{eq_PZ}) with $c = 0$ (see \S 5.3 for more details). The one-parameter families of $\mathrm{U}(2)$-invariant expanding Ricci solitons found in \cite{Cao1997} and \cite{FIK2003} can be given in terms of parameters (\refeq{eq_a_i}),(\refeq{eq_alpha}) and (\refeq{eq_beta})--(\refeq{eq_gamma}) using Lemma \refeq{lem_fixedSTE} and Lemma \refeq{lem_SU2STE} respectively as in Remark \refeq{rem_GRS_parameters} below. We also note that these solitons satisfy (\refeq{eq_asxi}) since they are asymptotically conical and solve (\refeq{eq_ODE1}).

\begin{lemma}\label{lem_noneinstein}
    There exists open sets $\Omega \subseteq \mathbb{R}^+ \times \mathbb{R}_{\geq 0}^2$, $\Omega_n \subseteq {\mathbb{R}^+}^2 \times \mathbb{R}_{\geq 0}$ for $n \in \{1,2,4\}$ and $\Omega_n \subseteq {\mathbb{R}^+}^2$ for $n \not\in \{1,2,4\}$ such that:
    \begin{itemize}
        \item each $(\alpha,a_1-a_2,a_2-a_3) \in \Omega$, $(\alpha,\beta,\gamma) \in \Omega_n$ for $n \in \{1,2,4\}$ and $(\alpha,\beta) \in \Omega_n$ for $n \not\in \{1,2,4\}$ produces a distinct $\mathrm{SU}(2)$-invariant complete non-Einstein expanding gradient Ricci soliton on $\mathbb{R}^4$ and $\mathcal{O}(-n)$ respectively;
        \item $\mathbb{R}^+ \times\{0\}\times \{0\}$, $\{(\alpha,({1-\alpha})/{6},0): \alpha \in (0,1)\}$, $\{(\alpha,0,({\alpha-1})/{6}): \alpha >1\} \subseteq \Omega$;
        \item $\Omega_n' \cap ({\mathbb{R}^+}^2 \times \mathbb{R}_{\geq 0}) \subseteq \Omega_n$ for $n \in \{1,2,4\}$ where $\Omega_n'$ is an open neighborhood of $\mathbb{R}^+ \times \{0\}\times\{0\}$ in $\mathbb{R}^3$;
        \item $\Omega_n' \cap {\mathbb{R}^+}^2 \subseteq \Omega_n$ for $n \not\in \{1,2,4\}$ where $\Omega_n'$ is an open neighborhood of $\mathbb{R}^+ \times \{0\}$ in $\mathbb{R}^2$; and 
        \item $\mathbb{R}^+ \times \{1\} \times \{0\} \subseteq \Omega_4$ and $\mathbb{R}^+ \times \{{n}/{(2n-4)}\} \subseteq \Omega_n$ for all other $n \geq 3$.
    \end{itemize}
\end{lemma}

\begin{remark}
    \label{rem_GRS_parameters}
    The family $\mathbb{R}^+ \times\{0\}\times \{0\} \subseteq \Omega$ is the $\mathrm{SO}(4)$-invariant family constructed in \S 4.2 above. The family $\{(\alpha,({1-\alpha})/{6},0): \alpha \in (0,1)\} \cup \{(\alpha,0,({\alpha-1})/{6}): \alpha >1\} \subseteq \Omega^E$ is the $\mathrm{U}(2)$-invariant K\"ahler expanders of \cite{Cao1997}. The families $\mathbb{R}^+ \times \{0\}\times\{0\}$ and $\mathbb{R}^+ \times\{0\}$ on the boundary of $\Omega_n$ are from the degenerate metrics constructed in \S 4.1. Finally the families $\mathbb{R}^+ \times \{1\} \times \{0\} \subseteq \Omega_4$ and $\mathbb{R}^+ \times \{{n}/{(2n-4)}\} \subseteq \Omega_n$ for all other $n \geq 3$ are the $\mathrm{U}(2)$-invariant K\"ahler expanders of \cite{FIK2003}. 
\end{remark}

\begin{proof}[Proof of Lemma 4.6]
Consider the further transformation of (\refeq{eq_xiSU2})--(\refeq{eq_LiSU2}) given by $\mathcal{L} = {1}/{\xi}$. We define $\Omega$ and $\Omega_n$ as the set of parameters in their corresponding spaces such that $Z =(\mathcal{L},L_1,L_2,L_3,R_1,R_2,R_3) \to 0$ as $t \to \infty$. Since all solitons listed in the previous remark satisfy these asymptotics, it suffices to prove $\Omega$ are $\Omega_n$ are open. Fix parameters in $\Omega$ or $\Omega_n$ and choose $\tau$ as in Lemma \refeq{lem_fixedSTE} or \refeq{lem_SU2STE} respectively. The initial behavior of $\xi$ implies we can reparameterize via an injective function $t:\mathbb{R} \to \mathbb{R}^+$ satisfying $t(0) = \tau>0$ and $\frac{dt}{ds} = \mathcal{L}(t(s))$. In particular, $\displaystyle \lim_{s \to -\infty} t(s) = 0$ and if (\refeq{eq_asxi}) is satisfied, then $\displaystyle \lim_{s \to \infty} t(s) = \infty$ and the metric will be complete by Corollary \refeq{cor_EAU}. Under such a transformation, equations (\refeq{eq_xiSU2})--(\refeq{eq_LiSU2}) become:
\begin{equation}
    \label{eq_ReSU2}
\begin{aligned}
    \frac{d\mathcal{L}}{ds} &= \mathcal{L}^3(L_1^2+L_2^2+L_3^2-1); \\
    \frac{dR_i}{ds} &= \mathcal{L}R_i(L_i-L_j-L_k); \; \text{ and} \\
    \frac{dL_i}{ds} &= -L_i+\mathcal{L}(2R_i^2-2{(R_j-R_k)}^2+1).
\end{aligned}
\end{equation}
Therefore, as in the proof of Lemma {\refeq{lem_einstein}}, it suffices to prove $Z=0$ is asymptotically stable in the system (\refeq{eq_ReSU2}) and solutions satisfy (\refeq{eq_asxi}). Linearising $Z =(\mathcal{L},L_1,L_2,L_3,R_1,R_2,R_3)$ about $0$, one computes $Z' = AZ+B(Z)$ where $B(0) = 0$, $DB(0)=0$ and 
\[ A = \left(\begin{matrix} 
    0 & 0 & 0 & 0 & 0 & 0 & 0 \\
    -1 & 1 & 0 & 0 & 0 & 0 & 0 \\
    -1 & 0 & 1 & 0 & 0 & 0 & 0 \\
    -1 & 0 & 0 & 1 & 0 & 0 & 0 \\
    0 & 0 & 0 & 0 & 0 & 0 & 0 \\
    0 & 0 & 0 & 0 & 0 & 0 & 0 \\
    0 & 0 & 0 & 0 & 0 & 0 & 0 \\
\end{matrix}\right).\]
We diagonalize $A$, $AP=PD$ where $D = \text{diag}(0,0,0,0,-1,-1,-1)$ and consider the transformation $\mu = P^{-1}Z$. The ODE system becomes
\[\mu' = D\mu +P^{-1}B(P\mu)\]
where
\[ P = \left(\begin{matrix} 
    0 & 0 & 0 & 1 & 0 & 0 & 0 \\
    0 & 0 & 0 & 1 & 1 & 0 & 0 \\
    0 & 0 & 0 & 1 & 0 & 1 & 0 \\
    0 & 0 & 0 & 1 & 0 & 0 & 1 \\
    1 & 0 & 0 & 0 & 0 & 0 & 0 \\
    0 & 1 & 0 & 0 & 0 & 0 & 0 \\
    0 & 0 & 1 & 0 & 0 & 0 & 0 \\
\end{matrix}\right) \; \text{ and } \; 
P^{-1} = \left(\begin{matrix} 
    0 & 0 & 0 & 0 & 1 & 0 & 0 \\
    0 & 0 & 0 & 0 & 0 & 1 & 0 \\
    0 & 0 & 0 & 0 & 0 & 0 & 1 \\
    1 & 0 & 0 & 0 & 0 & 0 & 0 \\
    -1 & 1 & 0 & 0 & 0 & 0 & 0 \\
    -1 & 0 & 1 & 0 & 0 & 0 & 0 \\
    -1 & 0 & 0 & 1 & 0 & 0 & 0 \\
\end{matrix}\right)\]
or explicitly
\begin{alignat*}{3}
    \frac{dR_i}{ds} &= &&\mu_i' &&= \mu_i\mu_4(-\mu_4+\mu_{i+4}-\mu_{j+4}-\mu_{k+4}),\\
    \frac{d\mathcal{L}}{ds} &= \; \; &&\mu_4' \;  &&= \mu_4^3\big({(\mu_4+\mu_5)}^2+{(\mu_4+\mu_6)}^2+{(\mu_4+\mu_7)}^2-1\big) \; \text{ and } \hspace{1.9cm}
\end{alignat*}
\begin{multline}
    \label{eq_Lidynamics}    
    \frac{d(L_i-\mathcal{L})}{ds} = \mu_{i+4}' = -\mu_{i+4}\\
    +\mu_4\Bigg(2\mu_i^2-2{(\mu_j-\mu_k)}^2+\mu_4^2\Big(1-{(\mu_4+\mu_{i+4})}^2-{(\mu_4+\mu_{j+4})}^2-{(\mu_4+\mu_{k+4})}^2\Big)\Bigg)
\end{multline}
for $i,j,k \in \{1,2,3\}$, $i \neq j \neq k \neq i$ again. Splitting into the stable and centre eigenspaces $Y_c = (\mu_1,\mu_2,\mu_3,\mu_4)$ and $Y_s = (\mu_5,\mu_6,\mu_7)$, we have equations
\[Y_c' = F_c(Y_c,Y_s) \; \; \; \text { and } \; \; \; Y_s' = -Y_s+F_s(Y_c,Y_s).\]
By the center manifold theorem, there exists a function $H:\mathbb{R}^4 \to \mathbb{R}^3$ and a constant $\varepsilon>0$ such that a local (analytic) centre manifold is a graph of $H$:
\[W_{loc}^c = \{(Y_c,H(Y_c)):|Y_c|<\varepsilon\}.\]
In particular, since centre manifolds begin at the origin and are tangential to the eigenvectors corresponding to $0$, $H(0) = 0$ and $DH(0) = 0$. Moreover, as centre manifolds are part of the flow $Y_s = H(Y_c)$ and $Y_s' = DH(Y_c)Y_c'$. This yields the partial differential equation for $H$:
\[DH(Y_c)F_c(Y_c,H(Y_c)) = -H(Y_c) + F_s(Y_c,H(Y_c)),\]
which can be written as
\begin{multline*}
    H^i := \frac{H_i}{\mu_4} =2\mu_i^2-2{(\mu_j-\mu_k)}^2+\mu_{4}^2\big(1-{(\mu_4+H_i)}^2-{(\mu_4+H_j)}^2-{(\mu_4+H_k)}^2\big) \\
    -\bigg(\frac{\partial H_i}{\partial \mu_1}(\mu_1(-\mu_4+H_1-H_2-H_3))+\frac{\partial H_i}{\partial \mu_2}(\mu_2(-\mu_4+H_2-H_1-H_3))\\
    + \frac{\partial H_i}{\partial \mu_3}(\mu_3(-\mu_4+H_3-H_1-H_2))+\frac{\partial H_i}{\partial \mu_4}\mu_4^2\big({(\mu_4+H_1)}^2+{(\mu_4+H_2)}^2+{(\mu_4+H_3)}^2-1\big) \bigg),
\end{multline*}
where $H = (H_1,H_2,H_3)$ and $i \in \{1,2,3\}$. In particular, we can extend $H^i$ smoothly to some neighborhood of the origin in $\mathbb{R}^4$ such that $H^i(Y_c) = O(|Y_c|)$ as $Y_c \to 0$. The origin is now easily seen to be asymptotically stable since the dynamics along the centre manifold are
\begin{equation}
    \mu_i' = -\mu_i\mu_4^2(1-H^1-H^2-H^3) = -\mu_i\mu_4^2(1+O(|Y_c|))
\end{equation}
for $i \in \{1,2,3\}$ and
\begin{equation}
    \label{eq_xidynamics}
    \mu_4' = -\mu_4^3\big(1-{(\mu_4+H_1)}^2-{(\mu_4+H_2)}^2-{(\mu_4+H_3)}^2\big) = -\mu_4^3(1+O(|Y_c|)).
\end{equation}
Moreover, (\refeq{eq_xidynamics}) implies (\refeq{eq_asxi}) for all solutions with parameters in $\Omega$ or $\Omega_n$.
\end{proof}

Equations (\refeq{eq_Lidynamics})--(\refeq{eq_xidynamics}) reproduce the asymptotics in (\refeq{eq_asxi})--(\refeq{eq_asRi}). Therefore, all solitons in Lemma \refeq{lem_noneinstein} are asymptotically conical. Theorem \refeq{thm_expanders} is now proven as it is a consequence of Lemmas \refeq{lem_einstein} and \refeq{lem_noneinstein}. 

For the remainder of \S 4, we numerically simulate some slices of $\Omega$ and $\Omega_n$ using MATLAB. The following figures aim to highlight the large scope of new metrics produced via Theorem \refeq{thm_expanders}. We begin by plotting slices of $\Omega$ in Figure \refeq{fig_noneinstein_fixed}, where we also display the one-parameter family of metrics found by Cao in the $\mathrm{U}(2)$-invariant slice. We also include some slices of $\Omega_1$, $\Omega_2$ and $\Omega_4$ in Figure \refeq{fig_noneinstein_n}.

\begin{figure}[h]
    \caption{Numerical simulation of $\Omega$}
    \includegraphics[scale = 0.41]{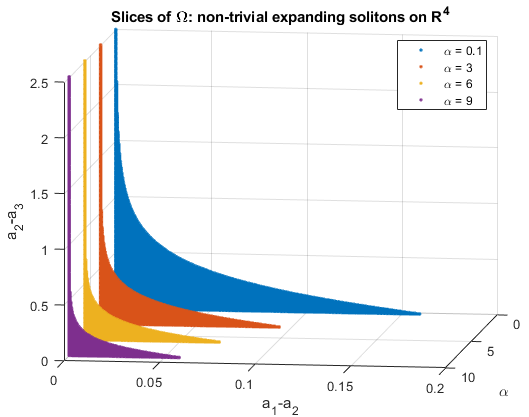}  \includegraphics[scale = 0.42]{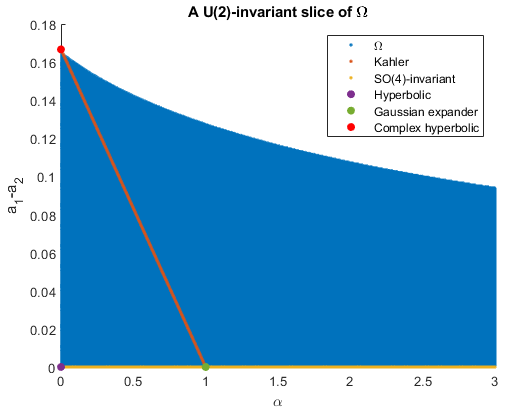}\label{fig_noneinstein_fixed}
\end{figure}
\pagebreak
\begin{figure}[h!]
    \caption{Numerical simulation of $\Omega_n$}
    \includegraphics[scale = 0.41]{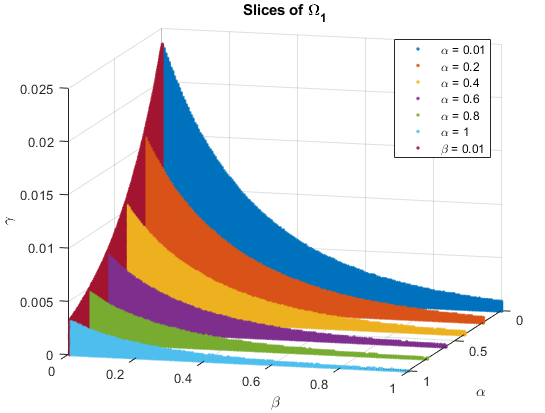}\includegraphics[scale = 0.42]{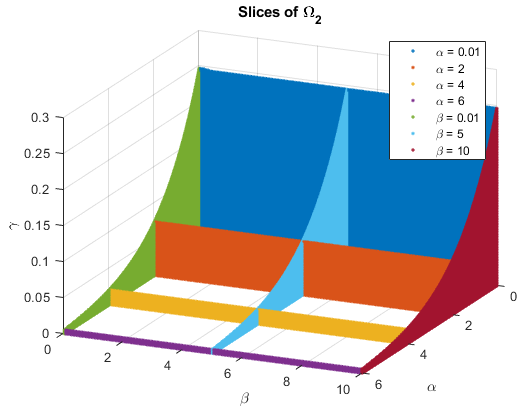}
    \includegraphics[scale = 0.42]{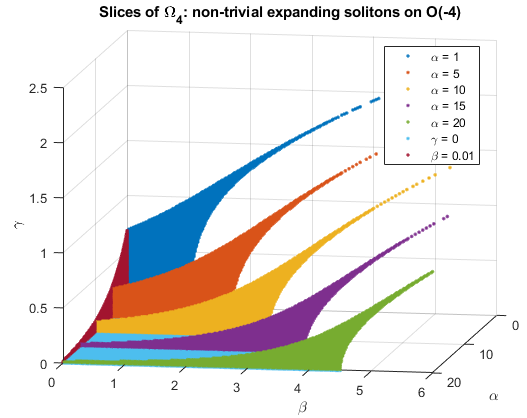}\includegraphics[scale = 0.42]{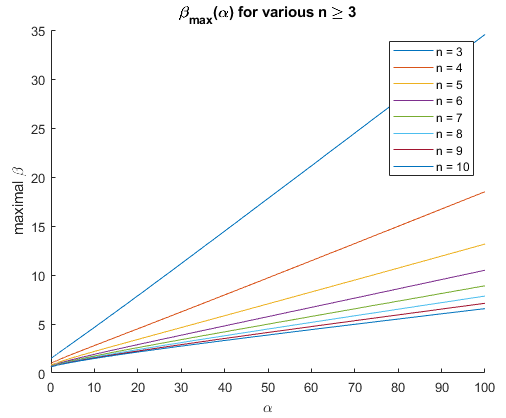}\label{fig_noneinstein_n}
\end{figure}

We note that for all $n \geq 3$, $\Omega_n$ (including the $\mathrm{U}(2)$-invariant slice of $\Omega_4$) appears to admit continuous, asymptotically affine functions $\beta_{\max}^n:\mathbb{R}^+ \to \mathbb{R}^+$, $\beta_{\max}^n(\alpha) = \sup_{\beta \in \mathbb{R}^+}\{\beta:(\alpha,\beta) \in \Omega_n\}$ such that $\lim_{\alpha \to 0^+} \beta_{\max}^n(\alpha) = {n}/{(2n-4)}$ and $(\alpha,\beta) \in \Omega_n$ if and only if $\beta < \beta_{\max}^n(\alpha)$. We also simulate $\beta_{\max}^n$ for $n \leq 3 \leq 10$ in Figure \refeq{fig_noneinstein_n}. If illustrated, the FIK expanders on $\mathcal{O}(n)$ would appear in the $\beta_{\max}^n$ plot as horizontal lines at ${n}/{(2n-4)}$.

Moreover, in all cases the numerics also suggest that the Einstein manifolds found in Lemma \refeq{lem_einstein} are dense in the boundary of $\Omega$ and $\Omega_n$ as $\alpha \to 0$.

\section{$\mathrm{SU}(2)$-invariant compact solitons}
For the remainder of this work we provide evidence for Conjecture \refeq{con_compact_solitons}. First, we recall that compact solitons are shrinking and we can assume, up to rescaling, $\lambda = 1$ (see \S 3). We also recall that we may assume the soliton potential is also $\mathrm{SU}(2)$-invariant if the gradient Ricci soliton is. Hoelscher's classification and the results of \S 3 imply that Conjecture \refeq{con_compact_solitons} is true if is it true over the spaces listed in Table \refeq{tab_compact} for the boundary conditions in Table \refeq{tab_compact_smoothness}.

We summarize in the following lemma.

\begin{lemma}
    \label{lem_compactconditions}
    If $g$ is a cohomogeneity one $\mathrm{SU}(2)$-invariant gradient Ricci soliton on a compact simply-connected $4$-manifold $M$ whose soliton potential is also invariant, then $M$ is diffeomorphic to one of 
    \[\mathbb{S}^4, \; \; \mathbb{S}^2 \times \mathbb{S}^2, \; \; \mathbb{CP}^2 \; \text{ or } \;  \mathbb{CP}^2\# \overline{\mathbb{CP}^2},\]
    and $g$ is of the form (\refeq{eq_metricfunctions}). The metric functions and the soliton potential satisfy equations (\refeq{eq_ODE1})--(\refeq{eq_ODEi}) with $\lambda > 0$ and smoothness conditions $u^{(odd)}(0)=0=u^{(odd)}(T)$ and as listed in Table \refeq{tab_compact_smoothness}. Moreover, for fixed $\lambda >0$, $g$ is uniquely determined by its corresponding parameters satisfying $(\alpha,a_1-a_2,a_2-a_3) \in \mathbb{R} \times {\mathbb{R}_{\geq 0}}^2$, $(\alpha,\beta,\gamma) \in \mathbb{R} \times \mathbb{R}^+ \times \mathbb{R}_{\geq 0}$ or $(\alpha,\beta) \in \mathbb{R} \times \mathbb{R}^+$, both at $0$ and at $T$. 
\end{lemma}

\begin{remark}
    If the smoothness conditions as $t \to 0$ and $T-t \to 0$ are equivalent, then a metric and its reversed orientation ($t \mapsto T-t$) may be realized by two separate parameters at $t=0$. This is why the Koiso--Cao soliton appears twice in Figure \refeq{fig_1_1}. This is the only known scenario where a smooth metric is realized by more than one set of parameters at a chosen end.
\end{remark}

\subsection{Examples}

The Koiso-Cao soliton is constructed in \S5.3 in the proof of Theorem \refeq{thm_Kahlerclassification}. Its construction does not give explicit functions so we do not include it here. The Page metric, while explicit, is unwieldy, so we do not include it here either. 

The round metric on $\mathbb{S}^4$, the Fubini--Study metric on $\mathbb{CP}^2$ and the direct product of round metrics of the same radii on $\mathbb{S}^2\times \mathbb{S}^2$ can be converted into the form (\refeq{eq_metricfunctions}) under the actions listed in \S 2, which we list explicitly in the examples below. All these examples are Einstein so they satisfy $\alpha = 0$ and $u'=0$ uniformly.

\begin{example}\label{ex_round_metric_1}
    The round metric on $\mathbb{S}^4 = \{A \in M_{3\times 3} (\mathbb{R}):A^T=A,\text{tr}(A)=0,\text{tr}(A^2)=1\}$ is invariant under conjugation by $\mathrm{SO}(3)$, which has the group diagram 
    \[\mathrm{S}(\mathrm{O}(1)\times\mathrm{O}(1)\times\mathrm{O}(1)) \subseteq \mathrm{S}(\mathrm{O}(1)\times\mathrm{O}(2)),\mathrm{S}(\mathrm{O}(2)\times\mathrm{O}(1)) \subseteq \mathrm{SO}(3).\] The round metric for this action is realized by functions 
    \[f_1(t) =  4 \sqrt{3} \sin\left(\frac{t}{\sqrt{3}}\right) \; \text{ and } \; f_{2,3}(t)=6\cos\left(\frac{t}{\sqrt{3}}\right)\pm 2\sqrt{3}\sin\left(\frac{t}{\sqrt{3}}\right)\]
    on $I = (0,{\sqrt{3}\pi}/{3})$ with parameters $\beta = {1}/{9}$ and $\gamma = {2}/{3}$ at both ends.
\end{example}

\begin{example}\label{ex_Fubini-Study_1}
    The Fubini--Study metric on $\mathbb{CP}^2$ is invariant under the $\mathrm{SO}(3)$ subaction of the standard $\mathrm{SU}(3)$ action, which has the group diagram \[\mathbb{Z}_2 \subseteq \mathrm{S}(\mathrm{O}(1) \times \mathrm{O}(2)),\mathrm{SO}(2)\times\mathrm{SO}(1) \subseteq \mathrm{SO}(3).\] 
    The Fubini--Study metric for this action is realized by functions
    \[f_1(t) = 2\sqrt{6}\sin\left(\frac{\sqrt{6}}{3}t\right), \; \; f_2(t)=2\sqrt{6}\cos\left(\frac{\pi}{4}-\frac{t}{\sqrt{6}}\right) \; \text{ and } \; f_3(t) = 2\sqrt{6}\sin\left(\frac{\pi}{4}-\frac{t}{\sqrt{6}}\right)\]
    on $I = (0,{\sqrt{6}\pi}/{4})$ with parameters $\beta={1}/{3}$, $\gamma = {\sqrt{6}}/{3}$ at $t=0$ and $\beta={1}/{12}$, $\gamma = {1}/{4}$ at $t=T={\sqrt{6}\pi}/{4}$.
\end{example}

\begin{example}\label{ex_round_metric_2}
    The round metric on $\mathbb{S}^4 \subseteq \mathbb{H} \times \mathbb{R}$ is invariant under the standard $Sp(1)$ action on the first (quaternionic) component, which is equivalent to the group action with diagram $1 \subseteq \mathrm{SU}(2),\mathrm{SU}(2) \subseteq \mathrm{SU}(2)$. The round metric for this action is realized by functions 
    \[f_1(t)=f_2(t)=f_3(t) = \sqrt{3}\sin\left(\frac{t}{\sqrt{3}}\right)\]
    on $I = (0,\sqrt{3}\pi)$ with parameters $ a_1=a_2=a_3=-{1}/{9}$ at both ends.
\end{example}

\begin{example}\label{ex_Fubini-Study_2}
    The Fubini--Study metric on $\mathbb{CP}^2$ is also invariant under the $\mathrm{SU}(2)$ action $A\cdot[z_0,z_1,z_2] = [z_0,A(z_1,z_2)]$, which has the group diagram $1 \subseteq \mathrm{U}(1),\mathrm{SU}(2)\subseteq \mathrm{SU}(2)$. The Fubini--Study metric for this action is realized by functions
    \[f_1(t) =  \frac{\sqrt{6}}{2}\sin\left(\frac{\sqrt{6}}{3}t\right) \; \text{ and } \; f_2(t)=f_3(t)=\sqrt{6}\sin\left(\frac{t}{\sqrt{6}}\right)\]
    on $I = (0,{\sqrt{6}\pi}/{2})$ with parameters $a_1=-{2}/{9}$, $a_2=a_3=-{1}/{18}$ at $t=0$ and $\beta={1}/{6}$, $\gamma=0$ at $t=T={\sqrt{6}\pi}/{2}$.
\end{example}

\begin{example}\label{ex_sum_of_round}
    The direct product of round metrics with the same radii on $\mathbb{S}^2\times \mathbb{S}^2 \subseteq \mathbb{R}^3 \times \mathbb{R}^3$ is invariant under the $\mathrm{SO}(3)$ action given by $A(x,y) = (Ax,Ay)$, which has the group diagram 
    \begin{multline*}
        1 \subseteq \mathrm{SO}(1) \times \mathrm{SO}(2),\mathrm{SO}(2)\times\mathrm{SO}(1) \subseteq \mathrm{SO}(3) \\
        \cong 1 \subseteq \mathrm{SO}(1)\times\mathrm{SO}(2),\mathrm{SO}(1)\times\mathrm{SO}(2) \subseteq \mathrm{SO}(3) \simeq \mathbb{Z}_2 \subseteq \mathrm{U}(1),\mathrm{U}(1) \subseteq \mathrm{SU}(2),
    \end{multline*}
    where $\cong$ represents a $G$-equivariant diffeomorphism since $N(\{1\}) = \mathrm{SO}(3)$ and $\simeq$ represents equivalence up to finite cover. The round metric is realized with different circle groups as in the first presentation of the group diagram above, with metric functions 
    \[f_1(t) = 2\sqrt{2}\sin\left(\frac{t}{\sqrt{2}}\right), \; \; f_2(t) = 2\sqrt{2}  \; \text{ and } \; f_3(t) = 2\sqrt{2}\cos\left(\frac{t}{\sqrt{2}}\right)\]
    on $I = (0,{\sqrt{2}\pi}/{2})$ with parameters $\beta = {1}/{4} = \gamma$ at both ends.
\end{example}

\subsection{Numerical procedure and results}
There are already analytical results supporting Conjecture \refeq{con_compact_solitons} --- most notably Brendle and Schoen's classification of manifolds with weakly $1/4$-pinched sectional curvature, implying simply-connected compact manifolds of that type are $\mathbb{S}^4$ with the round metric or $\mathbb{CP}^2$ with the Fubini--Study metric \cite{BrendleSchoen2008}.

We will not investigate the $\mathrm{SO}(3)\times \mathrm{SO}(2)$ actions in this work. Such actions always give doubly warped metrics over the standard metrics on $\mathbb{S}^1$ and $\mathbb{S}^2$. The short-term existence results in \S 3 can be repeated for these simpler actions by instead using a transformation with $R_i = (f_i)^{-1}$. The short-term existence results also extend to higher dimensional warped product metrics. Some results in this case are also already known; the B\"ohm--Wilking rounding theorem implies the only Einstein $\mathrm{SO}(3)\times \mathrm{SO}(2)$ invariant metric on $\mathbb{S}^4$ is the round one \cite{BohmWilking2008}. Moreover, Buttsworth has completed a comprehensive numerical analysis of these Ricci solitons in \cite{Buttsworth2022}.

There are few results pertaining to the uniqueness of Ricci solitons, specifically on compact cohomogeneity one $4$-dimensional spaces. Cohomogeneity one actions on simply-connected compact $4$-dimensional smooth manifolds form a small list. We give a brief numerical analysis for each of the remaining actions --- those by $\mathrm{SU}(2)$. 

As we have mentioned many times above, in the compact case, the orbit space of an $\mathrm{SU}(2)$-invariant gradient Ricci soliton with its inherited metric is isometric to $[0,T]$. We now point out that the problem of determining $T$ via initial conditions and the evolution equations is easily solved by the monotone quantity $\xi$ as in the following lemma.

\begin{lemma}
    If $(f_1,f_2,f_3,u)$ is a solution to (\refeq{eq_ODE1})--(\refeq{eq_ODEi}) with $\lambda \geq 0$, then $\xi$ is strictly decreasing. If $(f_1,f_2,f_3,u)$ satisfies any of the smoothness conditions at $t=0$ and $t=T$, then $\lim_{t \to 0^+}\xi = \infty$ and $\lim_{t \to T^-}\xi = -\infty$. 
\end{lemma}

\begin{proof}
The first claim follows trivially from (\refeq{eq_xiSU2}). The short-time existence in Lemma \refeq{lem_fixedSTE} and \refeq{lem_SU2STE} imply that $\xi(t) = {1}/{t} + \eta$ or $\xi(t) = {3}/{t} + \eta$ for small $t>0$ and some smooth $\eta$ with $\eta(0)=0$. The result at $T$ holds similarly.
\end{proof}

In particular, any smooth metric on a compact space has $T$ characterized by $\lim_{t \to T^-}\xi = -\infty$ since $\xi$ is well defined on $I = (0,T)$. As we can find $T$ from the evolution of the metric functions, we may now conduct a numerical analysis following the work in \cite{DHW2013}. Initially, we approximate $\xi,L_i$ and $R_i$ using a power series approximation in $t$ required for the short-time existence, i.e., up first order unless $n = 1$ in (\refeq{eq_boundarynonzeroSU2}). The first order approximations of $\xi,L_i$ and $R_i$ from the smoothness condition (\refeq{eq_boundaryzero}) are:
\begin{equation}
    \label{eq_fixed_approx}
    \begin{aligned}
        \xi(\delta) &= \frac{3}{\delta}+(a_1+a_2+a_3+\alpha)\delta + O(\delta^2)\\
        L_i(\delta) &= \frac{1}{\delta}+a_i\delta + O(\delta^2)\\
        R_i(\delta) &= \frac{1}{\delta}+\frac{1}{2}(a_i-a_j-a_k)\delta + O(\delta^2)
    \end{aligned}
\end{equation}
for $i,j,k \in \{1,2,3\}$, $i \neq j \neq k \neq i$. We now list the first order approximations of $\xi,L_i$ and $R_i$ from from the smoothness condition (\refeq{eq_boundarynonzeroSU2}) for $n \neq 1$ (recalling solutions with $n \not \in \{1,2,4\}$ are $\mathrm{U}(2)$-invariant) in Table \refeq{tab_compact_smoothness_approx} below.

\renewcommand{\arraystretch}{1.35}
\begin{table}[h!]
\begin{center}
    \begin{tabular}{||c|| c | c | c ||} 
        \hline
        {} & $\mathrm{U}(2)$-invariant &  $n = 2$ & $n = 4$ \\ [0.5ex] 
        \hline\hline
        $\xi(\delta)$ & $\frac{1}{\delta}+(\frac{2n\alpha+8\beta-3n\lambda}{3n})\delta$ & $\frac{1}{\delta}+(\frac{2\alpha+4\beta-3\lambda}{3})\delta$ & $\frac{1}{\delta}+(\frac{4\alpha+4\beta-\gamma^2-6\lambda}{6})\delta$\\
        \hline
        $L_1(\delta)$ & $\frac{1}{\delta}-(\frac{n\alpha+4\beta}{3n})\delta$ & $\frac{1}{\delta}-(\frac{\alpha+2\beta}{3})\delta$ & $\frac{1}{\delta}-(\frac{2\alpha+2\beta+\gamma^2}{6})\delta$\\
        \hline
        $L_2(\delta)$ & $(\frac{4\beta-n\lambda}{2n})\delta$ & $(\frac{2\beta+2\gamma-\lambda}{2})\delta$ & $\frac{\gamma}{2} + (\frac{\beta-\lambda}{2})\delta$\\
        \hline
        $L_3(\delta)$ & $(\frac{4\beta-n\lambda}{2n})\delta$ & $(\frac{2\beta-2\gamma-\lambda}{2})\delta$ & $-\frac{\gamma}{2} + (\frac{\beta-\lambda}{2})\delta$\\
        \hline
        $R_1(\delta)$ & $\beta\delta$  & $\beta\delta$ & $\beta\delta$\\
        \hline
        $R_2(\delta)$ & $\frac{1}{n\delta}+(\frac{n\alpha+4\beta}{6n^2})\delta$ & $\frac{1}{2\delta}+(\frac{\alpha+2\beta+6\gamma}{12})\delta$ & $\frac{1}{4\delta}+\frac{\gamma}{4}+(\frac{2\alpha+2\beta+7\gamma^2}{48})\delta$\\
        \hline
        $R_3(\delta)$ & $\frac{1}{n\delta}+(\frac{n\alpha+4\beta}{6n^2})\delta$ & $\frac{1}{2\delta}+(\frac{\alpha+2\beta-6\gamma}{12})\delta$ & $\frac{1}{4\delta}-\frac{\gamma}{4}+(\frac{2\alpha+2\beta+7\gamma^2}{48})\delta$\\
        \hline
    \end{tabular}
    \caption{First order approximations of functions corresponding to smooth gradient Ricci solitons for $\delta>0$.}\label{tab_compact_smoothness_approx}
\end{center}
\end{table}
\renewcommand{\arraystretch}{1}

\pagebreak

If $n=1$, $\gamma$ only appears in third order and higher terms for $\xi,L_i$ and $R_i$. Thus, we list the third order approximation of $\xi,L_i$ and $R_i$ given the boundary conditions (\refeq{eq_boundarynonzeroSU2}) with $n=1$ as shown below:

\begin{align*}
    \xi(\delta) &= \frac{1}{\delta}+\left(\frac{2\alpha+8\beta-3\lambda}{3}\right)\delta + \left(\frac{-4\alpha^2-32\alpha\beta+3\alpha \lambda-226\beta^2 + 84 \beta \lambda - 9 \lambda^2}{45}\right)\delta^3 \\
    L_1(\delta) &= \frac{1}{\delta}-\left(\frac{\alpha+4\beta}{3}\right)\delta+\left(\frac{14\alpha^2+112\alpha\beta-18\alpha\lambda+476\beta^2-144\beta\lambda+9\lambda^2}{180}\right)\delta^3 \\
    L_2(\delta) &= \left(\frac{4\beta-\lambda}{2}\right)\delta +\left(\frac{48\gamma-8\alpha\beta + 2 \alpha \lambda-92 \beta^2+32\beta \lambda-3\lambda^2}{24}\right)\delta^3\\
    L_3(\delta) &= \left(\frac{4\beta-\lambda}{2}\right)\delta +\left(\frac{-48\gamma-8\alpha\beta + 2 \alpha \lambda-92 \beta^2+32\beta \lambda-3\lambda^2}{24}\right)\delta^3\\
    R_1(\delta) &= \beta\delta+\left(\frac{-\alpha\beta-16\beta^2+3\beta\lambda}{6}\right)\delta^3\\
    R_2(\delta) &= \frac{1}{\delta}+\left(\frac{\alpha+4\beta}{6}\right)\delta+\left(\frac{720\gamma-4\alpha^2-32\alpha\beta+18\alpha \lambda - 316\beta^2+144\beta\lambda-9\lambda^2}{720}\right)\delta^3\\
    R_3(\delta) &= \frac{1}{\delta}+\left(\frac{\alpha+4\beta}{6}\right)\delta+\left(\frac{-720\gamma-4\alpha^2-32\alpha\beta+18\alpha \lambda - 316\beta^2+144\beta\lambda-9\lambda^2}{720}\right)\delta^3,
\end{align*}
where the $O(\delta^4)$ terms are missing. We then check the smoothness of the metric as $t \to T$ by a function $(\delta,C,x) \mapsto SOL(\delta,C,x) \in \mathbb{R}_{\geq 0}$. Explicitly, one first approximates $\xi,L_i$ and $R_i$ at $\delta$ with the truncated series expansions as above for the chosen parameters $x = a$ for (\refeq{eq_boundaryzero}) and $x = (\alpha,\beta,\gamma)$ for (\refeq{eq_boundarynonzeroSU2}). We then solve (\refeq{eq_xiSU2})--(\refeq{eq_LiSU2}) from $\delta$ to $t_0$ where $\xi(t_0) = C$. Finally, we compute
\[SOL_F(\delta,C,x) = u'(t_0)^2+\sum_{i=1}^3 f_i(t_0)^2+(f_i'(t_0)+1)^2\]
if the end smoothness conditions are (\refeq{eq_boundaryzero}) and
\[SOL_n(\delta,C,x) = u'(t_0)^2+f_1(t_0)^2+(f_1'(t_0)+n)^2+(f_2(t_0)-f_3(t_0))^2+f_2'(t_0)^2+f_3'(t_0)^2\]
if the end smoothness conditions are (\refeq{eq_boundarynonzeroSU2}) with $n \neq 4$ and
\[SOL_4(\delta,C,x) = u'(t_0)^2+f_1(t_0)^2+(f_1'(t_0)+4)^2+(f_2(t_0)-f_3(t_0))^2+(f_2'(t_0)+f_3'(t_0))^2\]
otherwise. We also minimize $SOL$ over all permutations of functions $(f_1,f_2,f_3)$ if $n=1$ or $n=2$ as described in \S 3. If $x$ are one-sided parameters for a compact Ricci soliton, then the smoothness conditions at $t=T$ imply that 
\[limSOL(x):=\displaystyle \lim_{C \to -\infty}\lim_{\delta \to 0} SOL(\delta,C,x) = 0.\] 
Moreover, for fixed $C$, $\displaystyle \lim_{\delta \to 0} SOL(\delta,C,\cdot)$ is continuous in a neighborhood of the parameters for a compact Ricci soliton. In particular, smaller $C$ distinguishes compact solitons, while for larger $C$, the better behaved $SOL(\delta,C,\cdot)$ is about a compact Ricci soliton. We found using $\delta = 0.001$ and $C= -20$ gave results with small numerical error, while still being able discern all known Ricci solitons and without losing clarity. 

We numerically compute $SOL(0.001,-20,\cdot)$ and display results in Figure \refeq{fig_compact_solitons} and in the Appendix. Large $SOL(0.001,-20,x)$ suggests that $x$ are not the parameters for a Ricci soliton on a compact cohomogeneity one manifold by the discussion above. In Figure \refeq{fig_compact_solitons}, we compute $SOL(0.001,-20,\cdot)$ over each slice of $\alpha$ and find that the only smooth solitons suggested are the known Einstein metrics, the Koiso--Cao (and reverse Koiso--Cao) metric and a family of singular metrics with a smooth end described in the followed subsection. The slice of the image of $SOL(0.001,-20,\cdot)$ about the known Ricci solitons is found in the Appendix.

\begin{figure}[h!]
    \caption{Numerical simulation of compact solitons}
    \includegraphics[scale = 0.7]{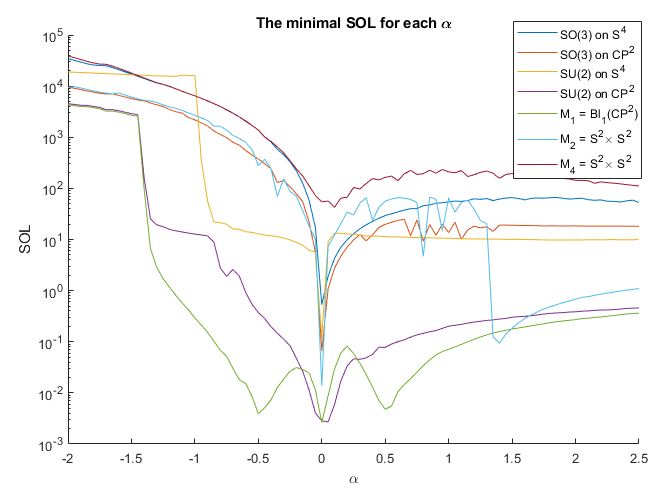}\label{fig_compact_solitons}
\end{figure}

We remark that the \textit{almost} smooth compact solitons detected by the numerics in light blue are on $\mathbb{S}^2\times \mathbb{S}^2$ and is in fact a $\mathrm{U}(2)$-invariant K\"ahler Ricci soliton that has cone angle arbitrarily close to $2\pi$ along an $\mathbb{S}^2$ orbit. In fact we find parameters which give $limSOL(x)$ arbitrarily close to $(n-2)^2$ on $M_n$ for each $n$. They are a few of the many singular gradient Ricci solitons we produce in the following section. These metrics constructed in the following section (applying the construction to non-orbifold metrics) give approximate upper bounds for the $SOL$ function that agree with Figure \refeq{fig_compact_solitons} on:
\begin{itemize}
    \item the $\mathrm{SU}(2)$ action on $\mathbb{S}^4$ (yellow) for $\alpha \in (-1,\infty) \setminus\{0\}$,  
    \item the $\mathrm{SU}(2)$ action on $\mathbb{CP}^2$ (purple) for $\alpha > -\sqrt{2}$,
    \item $M_1$ (green) for $\alpha \in (-\sqrt{2},\infty) \setminus\{0\}$,
    \item $M_2$ (light blue) for $\alpha > \approx 1.344.$
\end{itemize}
Indeed, in Figure \refeq{fig_compact_solitons_orbifolds} we plot $limSOL$ for the parameters which give the orbifold metrics we construct in the following section. We can do similarly for parameters on $M_n$ with $n\geq 3$ and on the $\mathrm{SO}(3)$ action on $\mathbb{S}^4$, but these all have large closing values and only occur for $\alpha > 2.604$ on $M_3$, $\alpha > 3.771$ on $M_4$ and on the $\mathrm{SO}(3)$ action on $\mathbb{S}^4$, $\alpha > 4.872$ on $M_5$ \dots . Figure \refeq{fig_compact_solitons_orbifolds} and Figure \refeq{fig_1_1} describe the only times when parameters give closing value $SOL<10$ in Figure \refeq{fig_compact_solitons}, thereby giving strong evidence for Conjecture \refeq{con_compact_solitons}.

\begin{figure}[h!]
    \caption{Closing conditions for compact K\"ahler solitons with one smooth end}
    \includegraphics[scale = 0.9]{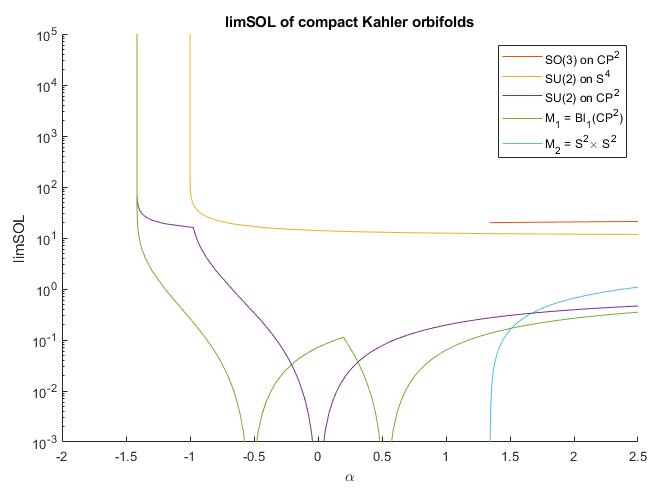}\label{fig_compact_solitons_orbifolds}
\end{figure}

\pagebreak

\subsection{$\mathrm{U}(2)$-invariant shrinking gradient K\"ahler Ricci solitons with orbifold singularities}

Conjecture \refeq{con_compact_solitons} suggests that the Koiso--Cao metric is the only non-Einstein $\mathrm{SU}(2)$-invariant compact gradient Ricci soliton. However, the numerics in Figure \refeq{fig_compact_solitons} detect ``almost smooth'' gradient Ricci solitons on $M_n$ and $\mathbb{CP}^2$ even where smooth Ricci solitons are not known. These are detected about the $\mathrm{U}(2)$-invariant K\"ahler parameters. If any of these metrics are indeed K\"ahler, then they cannot be smooth by the classification of K\"ahler shrinkers in complex dimension 2. Nevertheless, we still construct the metrics detected in Figure \refeq{fig_compact_solitons} --- they contain a family of orbifolds. This construction leads to a classification of $\mathrm{U}(2)$-invariant shrinking gradient K\"ahler Ricci solitons with orbifold singularities whose base space is a smooth manifold --- Theorem \refeq{thm_Kahlerclassification}.

The proof of Theorem \refeq{thm_Kahlerclassification} is only possible because of the simplifications of the ODEs (\refeq{eq_ODE1})--(\refeq{eq_ODEi}) to just one ODE due to K\"ahlerity (see Lemma \refeq{lem_U2Kahler_simpler}). The boundary conditions (\refeq{eq_u_smoothness})--(\refeq{eq_boundarynonzeroSU2}) are adjusted to include orbifolds and applied to the ODE to obtain the classification.

Without loss of generality, let $f_2 = f_3$ by $\mathrm{U}(2)$-invariance and set $f:= f_2$. We also set $\lambda = 1$ as we classify the metrics up to rescaling. 

By Lemmas \refeq{lem_fixedSTE} and \refeq{lem_SU2STE}, $\mathrm{U}(2)$-invariant K\"ahler metrics are described by $f$ since $f = f_2 = f_3$, $f_1 = \pm ff'$ and $u'=Cff'$ for some $C \in \mathbb{R}$. This is also true for $\mathrm{U}(2)$-invariant K\"ahler orbifold metrics. The singular orbits are thus characterized by the vanishing of $f$ or $f'$. Clearly, $f$ is either strictly increasing or strictly decreasing on $I$. 

The ODEs (\refeq{eq_ODE1})--(\refeq{eq_ODEi}) are reduced to:
\begin{align} 
    -\frac{f'''}{f'}-5\frac{f''}{f}+C(ff''+(f')^2) &= 1 \; \text { and } \label{eq_Kahler_order_3}\\
    -2\frac{f''}{f}+C(f')^2-4\frac{(f')^2}{f^2}+\frac{4}{f^2} &= 1. \nonumber
\end{align}
The latter equation can be presented more conveniently as
\begin{equation}
    \label{eq_Kahler_order_2}
    \frac{d}{df}\left(f^4(f')^2e^{-C f^2/2}\right) = -f^5e^{-C f^2/2}\left(-2\frac{f''}{f}+Cf'^2-4\frac{(f')^2}{f^2}\right) = f^3e^{-C f^2/2}(4-f^2). 
\end{equation}
Differentiating equation (\refeq{eq_Kahler_order_2}), one easily sees that (\refeq{eq_Kahler_order_2}) implies (\refeq{eq_Kahler_order_3}). 

We now find the boundary conditions to ensure the metrics close as orbifolds. On a $\mathbb{CP}^1$ singular orbit, i.e., one given by stabilizer groups 
\[H = \mathbb{Z}_n \times \mathrm{U}(1) \subset K = \mathrm{U}(1)\times \mathrm{U}(1) \subset G = \mathrm{U}(2),\] 
the orbifold closing conditions are well known (see \cite{PetersenZhu1995} for example). The orbifold closing conditions with cone angle ${2 \pi}/{q}$ are simply given by replacing $n$ with $n/q$ in (\refeq{eq_boundarynonzeroSU2}):
\[u^{(odd)}(0)=0, \; \; f_1'(0) = n/q, \; \; f_1^{(even)}(0) = 0 \; \; \text{ and } \; \; f_{2}^{(odd)}(0) = f_{3}^{(odd)}(0) = 0.\]

We note that in the $\mathrm{U}(2)$-invariant case, Lemma \refeq{lem_SU2STE} was proven for any real $n>0$. In particular, Lemma \refeq{lem_SU2STE} implies that if the $\mathrm{U}(2)$-invariant orbifold metric is K\"ahler, then $f:=f_2=f_3$, $f_1 = \pm ff'$ and $u' = \mp \alpha qff'/n$. Moreover, if $f$ does not vanish on the singular orbit at $t=0$, then Lemma \refeq{lem_SU2STE} and an appropriately adjusted definition for $\beta$ show that K\"ahler orbifolds also satisfy $f(0)^2 = 4\pm 2n/q$. Since $f(0)>0$, if $f(0)^2<4$, i.e., $f$ is increasing, then $q>n/2$. By a simple induction argument, (\refeq{eq_Kahler_order_3}) implies that if $f(t) \neq 0$ and $f'(t) = 0$ then $f^{(\text{odd})}(t)=0$. Moreover, $f(0)^2 = 4 \pm 2n/q$ and $f'(0) = 0$ gives $f_1'(0) = n/q$ by (\refeq{eq_Kahler_order_2}). 

The only other possibility for an orbifold singular orbit is when $f$ vanishes  (with $f'$ bounded). This occurs when $f \to 0$ and $f_1 \to 0$. By (\refeq{eq_Kahler_order_2}), $\displaystyle \lim_{f \to 0} (f_1)^2=0$ implies $\displaystyle \lim_{f \to 0}(f_1')^2=1=\lim_{f \to 0}(f')^2$. A simple induction argument applied to (\refeq{eq_Kahler_order_2}) shows that if $f(t) = 0$ and $f'(t) = \pm 1$, then $f^{(even)}(t)=0$. In particular, if $f \to 0$ and $ff' \to 0$, then (\refeq{eq_boundaryzero}) is satisfied and a point singular orbit is formed. Moreover, the metric is smooth on a neighborhood of the singular orbit point if and only if the principal orbit is a $3$-sphere. In fact, since a fixed point has isotropy $K=G$, Lemma \refeq{lem_isotropy_sphere} implies the underlying space is not a smooth manifold if the principal orbit $G/H=K/H$ not a sphere. If $f$ and $ff'$ both vanish at $0 = \inf(I)$, Lemma \refeq{lem_fixedSTE} implies $f_1 = ff'$ and $u'=-\alpha ff'$.

\begin{remark}
    Orbifold points can be achieved by an obvious quotient of a space whose only singular orbits are fixed points. Orbifold points can also be achieved when the other singular orbit is $\mathbb{CP}^1$ and $n \neq 1$. We will not explicitly analyze these options, however we will note the non-trivial constructions with orbifold points from the $\mathrm{U}(2)$ actions we must investigate.
\end{remark}

We integrate (\refeq{eq_Kahler_order_2}) and summarize in the following Lemma. Recall that we only consider $\mathrm{U}(2)$ actions that reduce to a cohomogeneity one $\mathrm{SU}(2)$-action.

\begin{lemma}
    \label{lem_U2Kahler_simpler}
    If $g$ is a $\mathrm{U}(2)-$invariant shrinking gradient K\"ahler Ricci soliton with orbifold singularities along singular orbits, then there exists a constant $C \in \mathbb{R}$ and a function $f:I \to \mathbb{R}_{\geq 0}$ such that $g$ is of the form (\refeq{eq_metricfunctions}) where $f_1 = \pm ff'$, $f = f_2 = f_3$ and $u'=Cff'$ for some $C \in \mathbb{R}$. Moreover, the vanishing set of $ff'$ is $\partial I$. If $C = 0$, the metric is Einstein and
    \begin{equation}
        \label{eq_Kahler_Einstein_cons}
        (f')^2 = 1-\frac{1}{6}f^2 + D_0f^{-4},
    \end{equation}
    where $D_0$ is an integration constant. If $C \neq 0$, instead $f$ solves
    \begin{equation}
        \label{eq_Kahler_not_Einstein}
        (f')^2 = \frac{1}{C^3 f^4} (C^2f^2( f^2-4)+4C( f^2-2)+8) + Df^{-4}e^{Cf^2/2},
    \end{equation}
    where $D$ is an integration constant. If $\inf I = 0$, then there exists a $q \in \mathbb{Z}_{>0}$ such that 
    \begin{align}
        f(0) = 0, \; \; f'(0) = 1 \; \; \; &\text{ if } f(0)=0  \label{eq_Kahler_orbifold_closingvan}; \\
        f(0) = \sqrt{4-2n/q}, \; \; f'(0) = 0 \; \; \; &\text{ if } f(0)>0 \text{ and } f \text{ is increasing}; \; \; \text{or} \label{eq_Kahler_orbifold_closinginc} \\
        f(0) = \sqrt{4+2n/q}, \; \; f'(0) = 0 \; \; \; &\text{ if } f(0)>0 \text{ and } f \text{ is decreasing}. \label{eq_Kahler_orbifold_closingdec}
    \end{align}
    Conversely, any solution $f:I \to \mathbb{R}$ to (\refeq{eq_Kahler_Einstein_cons}) or (\refeq{eq_Kahler_not_Einstein}) with $ff' \neq 0$ on $\text{Int}(I)$ induces a $\mathrm{U}(2)-$invariant shrinking gradient K\"ahler Ricci soliton with orbifold singularities along singular orbits if $f$ satisfies one of (\refeq{eq_Kahler_orbifold_closingvan})--(\refeq{eq_Kahler_orbifold_closingdec}) (and another of (\refeq{eq_Kahler_orbifold_closingvan})--(\refeq{eq_Kahler_orbifold_closingdec}) after the transformation $t \mapsto T-t$ if $I = [0,T]$).
\end{lemma}

\begin{remark}
    Lemma \refeq{lem_U2Kahler_simpler} can be rewritten for steady and expanding gradient K\"ahler Ricci solitons, though (\refeq{eq_Kahler_orbifold_closingdec}) only appears in the shrinking case.
\end{remark}

We can now construct our singular metrics, noting that the constructions in the following three lemmas also produce non-orbifold metrics for positive non-integer $q$.

\begin{lemma}
    \label{lem_kahler_orbifolds}
    For any $q_1, q_2 \in \mathbb{Z}_{>0}$, define $m$ as follows: 
    \begin{itemize}
        \item $m=0$ if $q_1,q_2\leq {n}/{2}$;
        \item $m=2$ if $q_1,q_2 > {n}/{2}$ and $q_1 \neq q_2$; and
        \item $m=1$ otherwise.
    \end{itemize}
    Up to rescaling, there exists exactly $m$ $\mathrm{U}(2)$-invariant gradient K\"ahler Ricci solitons on $M_n$ with orbifold singularities at the singular orbits whose cone angles are ${2 \pi}/{q_1}$ and ${2 \pi }/{q_2}$.
\end{lemma}

\begin{remark}
    \label{rem_KE}
    If $\max(q_1,q_2) > {n}/{2}$, then one of the metrics in Lemma \refeq{lem_kahler_orbifolds} is Einstein if and only if $6q_i^2(2q_j+n)-6q_i(q_i+q_j)(q_j+n)+2n(q_i+q_j)^2=0$ for $i,j \in \{1,2\}$, $i \neq j$. For each pair $q_1,q_2 \in \mathbb{Z}_{>0}$, there is at most one Einstein metric. There are infinitely Einstein orbifold metrics included in Lemma \refeq{lem_kahler_orbifolds}.
\end{remark}

\begin{remark}
    The Koiso--Cao soliton is the only smooth metric included in Lemma \refeq{lem_kahler_orbifolds}. 
\end{remark}

\begin{proof}[Proof of Lemma \refeq{lem_kahler_orbifolds}]
    Fix $q_1,q_2 \in \mathbb{Z}_{>0}$ and define $k_i = {n}/{q_i}$. Since $M_n$ is compact, we can choose the parameterization such that $f_1 = ff'$. We also choose the closing conditions to have cone angle ${2 \pi}/{q_1}$ at $t=0$. By Lemma \refeq{lem_SU2STE}, $C=-\alpha/k_1$. It suffices to prove that there exists a unique solution for $f$ in Lemma \refeq{lem_U2Kahler_simpler} with $f' \neq 0$ on $\text{Int}(I) = (0,T)$ that satisfies the boundary conditions 
    \[L:=f(0)^2 = 4-2k_1, \; \; f(T)^2 = 4+2k_2 \; \text{ and } \; f'(0)=0=f'(T)\]
    for each $k_1,k_2>0$, such that $k_1<2$. If $k_1,k_2 < 2$ then the second solution is given by exchanging $k_1$ and $k_2$ which produces distinct metric functions unless $k_1=k_2$.

    \textbf{Step 1: Existence of the Einstein metrics.}\\
    The metric is Einstein if and only if $\alpha = 0$. All these metrics have been constructed in \cite{PetersenZhu1995} already. We closely follow their procedure below. The boundary condition at $0$ implies $D_0 = -(4-2k_1)^2(1+k_1)/3$. We then factorize to produce:
    \begin{equation}
        \label{eq_Kahler_Einstein}
        (f')^2 = \frac{1}{6f^4}(L-f^2)(f^4-2(1+k_1)f^2-4(2-k_1)(1+k_1)).
    \end{equation}
    We note that $f(t)^{-4}/6 \neq 0 \neq L-f(t)^2$ while $f'(t) > 0$. A solution to (\refeq{eq_Kahler_Einstein}) has $f'(t) > 0$ for $t > 0$ until $f(t)^2 = (1+k_1)+\sqrt{3(1+k_1)}$. In particular, the solution always exists while $f(t)^2 \leq 4$. The boundary condition at $T$ is therefore satisfied exactly when choosing 
    \begin{equation}
        \label{eq_Kahler_Einstein_rationals}
        k_2 = \frac{1}{2}(-(3-k_1)+\sqrt{3(1+k_1)(3-k_1)}).
    \end{equation}
    If $0<k_1<2$, then (\refeq{eq_Kahler_Einstein_rationals}) implies $0<k_2<k_1$. One easily finds an infinite family of solutions to (\refeq{eq_Kahler_Einstein_rationals}) with $0<k_1<2$. This gives us Remark \refeq{rem_KE} and existence of K\"ahler Einstein metrics.

    \textbf{Step 2: Existence of non-Einstein Ricci solitons.}\\
    When $\alpha \neq 0$, $C \neq 0$, and the boundary condition at $0$ is
    \[C^2L(L-4)+4C(L-2)+8 = -C^3 De^{\frac{CL}{2}}.\]
    Applying the initial condition, we can rewrite (\refeq{eq_Kahler_not_Einstein}) as
    \[(f')^2 = \frac{e^{C(f^2-L)/2}}{C^3 f^4} \bar{H_1}(f^2)\]
    where
    \[\bar{H_1}(F) = e^{-C(F-L)/2}(C^2F(F-4)+4C(F-2)+8) - (C^2L(L-4)+4C(L-2)+8).\]
    Clearly, for all $f^2>0$, $f'=0$ if and only if $\bar{H_1}(f^2)=0$. Since 
    \[\bar{H_1}'(F) = -C^3e^{-C(F-L)/2}F(F-4)/2 \text{ and } \bar{H_1}(L)=0,\]
    $\bar{H_1}'$ is non-zero and does not change sign on $F \in (0,4)$. Thus, $\bar{H_1}(F)$ is non-zero on $(L,4]$. The expression for $\bar{H_1}'(F)$ also implies there is at most one root of $\bar{H_1}$ for $F>4$. Therefore, the solution for $f$ satisfying the boundary conditions in (\refeq{eq_Kahler_not_Einstein}) exists with $f' \neq 0$ on $(0,T)$ if and only if we can apply the boundary conditions.
    
    We can also note that, by the intermediate value theorem, there exists a root bigger than $4$ for $\bar{H_1}$ if and only if the sign of $\bar{H_1}'(L)$ is opposite that of $\displaystyle \lim_{F \to \infty} \bar{H_1}(F)$. This occurs exactly when 
    \[C < -\frac{2(L-2) - 2 \sqrt{-L^2+4L+4}}{L(4-L)}\]  
    where $\displaystyle \lim_{F \to \infty} \bar{H_1}(F) = 0$ at equality. We now impose the the orbifold closing conditions at $T$:
    \[4C^2k_2(2+k_2)+8C(1+k_2)+8 = -C^3De^{C(2+k_1)}.\]
    Equivalently, $\alpha/k_1 = -C$ is a root of the function
    \[h_1(x) = e^{(k_1+k_2)x}\left(k_2(2+k_2)x^2-2(1+k_2)x+2\right) + k_1(2-k_1)x^2+2(1-k_1)x-2.\]
    Therefore, there exists a non-Einstein Ricci soliton on $M_n$, with orbifold closing conditions corresponding to $k_1$ and $k_2$, for each non-zero root of $h_1$. We now follow the construction in \cite{Cao1996} for the remainder of this and the next step. It is easily seen that $h_1(0)=0=h_1'(0)=h_1''(0)$ and 
    \[h_1'''(0) = (k_1+k_2)\big(6k_2(2+k_2)-6(k_1+k_2)(1+k_2)+2(k_1+k_2)^2\big).\]
    Moreover, $h_1'''(0)=0$ exactly when the metric is Einstein. The limits $\displaystyle \lim_{x \to \pm\infty}h_1(x) = \infty$, imply there exists a negative root for $h_1$ if $h_1'''(0)>0$ and a positive root for $h_1$ if $h_1'''(0)<0$ by the intermediate value theorem.

    \textbf{Step 3: Uniqueness of Ricci solitons with $\alpha<0$.}\\
    If $h_1'''(0)\geq 0$, then, for all $x>0$, $h_1'''(x)>0$ and hence $h_1(x)>0$. Indeed,
    \[h_1'''(x)e^{-(k_1+k_2)x} = k_2(2+k_2)(k_1+k_2)^3x^2+(k_1+k_2)\big(h_1'''(0)+2(k_1+k_2)^2(k_2+2-k_1)\big)x + h_1'''(0).\]
    Now consider the Taylor series for $\tilde{h_1}$:
    \[\tilde{h_1}(x) := e^{(k_1+k_2)x}h_1(-x) = -\sum_{l=3}^\infty \frac{(k_1+k_2)^{l-2}}{l!}\kappa_1(l)x^l,\]
    where
    \[\kappa_1(l) = -k_1(2-k_1)l^2+(2(1-k_1)(k_1+k_2)+k_1(2-k_1))l+2(k_1+k_2)^2.\]
    $\kappa_1$ changes sign at most once for $l>0$, so as in Step 4. pp. 13--14 of \cite{Cao1996}, $\tilde{h_1}$ has at most one root in $(0,\infty)$ and $h_1$ has at most one root in $(-\infty,0)$

    \textbf{Step 4: Uniqueness of Ricci solitons with $\alpha\geq 0$.}\\
    We first show that if $h_1'''(0) \leq 0$, then $h_1(x)>0$ for all $x<0$. Note that 
    \[h_1''(x) = 2k_1(2-k_1)(1-e^{(k_1+k_2)x})+k_2(k_1+k_2)^2(2+k_2)xe^{(k_1+k_2)x}(x-A_0)\]
    where
    \[A_0 = \frac{2(k_1+k_2)(1+k_2)-4k_2(2+k_2)}{k_2(k_1+k_2)(2+k_2)}.\]
    In particular, $h_1''(x)>0$ for all $x<\min(0,A_0)$. Thus, it suffices to show $h_1'''(x)<0$ for all $x \in (A_0,0)$ whenever $A_0<0$. In fact, it suffices to show that $h_1'''(A_0)< 0$ as the coefficients of $h_1'''(x)e^{-(k_1+k_2)x}$ imply that $h_1'''(x)$ changes sign at most once on $(-\infty,0)$. We compute that, if $A_0 <0$, then
    \[h_1'''(A_0) = (k_1+k_2)e^{(k_1+k_2)A_0}(A_0-2k_1(2-k_1))<0.\]
    Therefore, if $h_1'''(0)=0$, then $0$ is the only root and the Einstein solution is unique. As in the previous step, all that remains to consider is uniqueness of a positive root of $h_1$. We have the Taylor series 
    \[h_1(x) = \sum_{l=3}^\infty \frac{(k_1+k_2)^{l-2}}{l!}\kappa_2(l)x^l\]
    for $h_1$, where
    \[\kappa_2(l) = k_2(2+k_2)l^2-(2(1+k_1)(k_1+k_2)+k_1(2+k_1))l+2(k_1+k_2)^2.\]
    Since $\kappa_2$ changes sign at most once for $l>0$, if $h_1'''(0)<0$, then there is a unique positive root of $h_1$.
\end{proof}

We now construct some complete non-compact solitons. If there is a singular orbit, we may assume $I = [0,\infty)$. Since $f$ is monotone, if $f$ is bounded, then $f \to K$ as $t \to \infty$ for some $K\geq 0$. Moreover, it must be that $f',f'' \to 0$ as $f \to K$. However, the ODEs (\refeq{eq_Kahler_Einstein_cons}) and (\refeq{eq_Kahler_not_Einstein}) imply that if $f \to 0$, then $f' \not \to 0$. Moreover, (\refeq{eq_Kahler_Einstein_cons}) and (\refeq{eq_Kahler_not_Einstein}) also imply that if $f \to K>0$ and $f' \to 0$, then $f'' \to (K^2-4)/K$. Finally, if $f' = 0$ at $f^2 = 4$, then the integration constants of (\refeq{eq_Kahler_Einstein_cons}) and (\refeq{eq_Kahler_not_Einstein}) must be either $D_0 = -16/3$ or $D = -8e^{-2C}(C+1)/C^3$ and then the ODEs become $(f')^2 \leq 0$ with equality only at $f^2=4$ --- there are also no metrics in this case. Therefore, if the metric is on a non-compact space, $f$ is not bounded, $f$ is increasing and $f' > 0$ for all $f > f(0)$.

\begin{lemma}
    \label{lem_kahler_orbifolds_noncompact}
    For every integer $q>{n}/{2}$, there exists a unique, up to rescaling, complete $\mathrm{U}(2)$-invariant gradient K\"ahler Ricci soliton on $\mathcal{O}(-n)$ with cone angle ${2 \pi}/{q}$ along orbifold singularities on the singular orbit.
\end{lemma}

\begin{remark}
    There are no Einstein metrics included in Lemma \refeq{lem_kahler_orbifolds_noncompact}. 
\end{remark}

\begin{remark}
    The FIK shrinker on $\mathcal{O}(-1)$ is the only smooth metric included in Lemma \refeq{lem_kahler_orbifolds_noncompact}. 
\end{remark}

\begin{proof}[Proof of Lemma \refeq{lem_kahler_orbifolds_noncompact}]
    Let $k = {n}/{q}$. The initial conditions for a $\mathrm{U}(2)$-invariant gradient K\"ahler Ricci soliton orbifold on $\mathcal{O}(-n)$ are (\refeq{eq_Kahler_orbifold_closinginc}). Moreover, as in the proof of Lemma \refeq{lem_kahler_orbifolds}, this metric is over $\mathcal{O}(-n)$ if and only if 
    \[C = -\frac{\alpha}{k} \geq -\frac{(1-k) - \sqrt{1+2k-k^2}}{k(2-k)}>0.\]
    If this inequality is strict, then the limit $\displaystyle \lim_{f \to \infty} f^4e^{-Cf^2/2}(f')^2$ is finite and positive, so the metric is incomplete since $C>0$. Instead, at equality,
    \[(f')^2 = \frac{C^2f^2(f^2-4)+4C(f^2-2)+8}{C^3f^4} \; \text{ and } \; \lim_{f \to \infty} (f')^2 = \frac{1}{C},\]
    so the metric is asymptotically conical and is complete.
\end{proof}

We now construct the compact K\"ahler gradient Ricci soliton orbifolds with smooth fixed points. Clearly, the two-fixed point action does not give any $\mathrm{U}(2)$-invariant K\"ahler gradient Ricci solitons since $f$ cannot vanish twice as it is strictly increasing.

\begin{lemma}
    \label{lem_kahler_orbifolds_CP1}
    For every $q \in \mathbb{Z}_{>0}$, there exists a unique, up to rescaling, complete $\mathrm{U}(2)$-invariant gradient K\"ahler Ricci soliton on $\mathbb{CP}^2$ with cone angle ${2 \pi}/{q}$ along the orbifold singularities on the $\mathbb{CP}^1$ orbit.
\end{lemma}

\begin{remark}
    The Fubini--Study metric is both the only Einstein metric and the only smooth metric included in Lemma \refeq{lem_kahler_orbifolds_CP1}. 
\end{remark}

\begin{proof}[Proof of Lemma \refeq{lem_kahler_orbifolds_CP1}]
The $\mathrm{SU}(2)$-action on $\mathbb{CP}^2$ that does not have a fixed point does not produce any $\mathrm{U}(2)$-invariant metrics. Indeed, if it were to, then the boundary conditions necessitate $f_1=f_2$ and $f_2=f_3$. In particular $f_1=f_2=f_3$ vanish together so the only possible singular orbits are points.

Therefore the $\mathrm{U}(2)$-invariant K\"ahler metrics on $\mathbb{CP}^2$ have a fixed point, which we reparameterize to be at $t=0$ as in (\refeq{eq_Kahler_orbifold_closingvan}). The Einstein equation now becomes (\refeq{eq_Kahler_Einstein_cons}) with $D_0=0$. This has the unique solution given in Example \refeq{ex_Fubini-Study_2} and is the Fubini study metric on $\mathbb{CP}^2$. In the non-Einstein case, $C = - \alpha$ and $D = (8 \alpha +8)/\alpha^3$ in (\refeq{eq_Kahler_not_Einstein}). The Ricci soliton equation becomes: 
\[(f')^2 = e^{-\alpha f^2/2}\frac{\bar{H_2}(f^2)}{\alpha^3 f^4}, \; \text{ where } \; \bar{H_2}(F) = (8\alpha+8) - e^{\alpha F/2}(\alpha^2F(F-4)-4\alpha(F-2)+8).\] 

As in the proof of Lemma \refeq{lem_kahler_orbifolds}, for all $f^2>0$, $f'=0$ if and only if $\bar{H_2}(f^2)=0$. Again, $\bar{H_2}'(F) = -\alpha^3e^{\alpha F/2}F(F-4)/2$ and $\bar{H_2}(F)$ is non-zero on $(0,4]$ as $\bar{H_2}(0)=0$. Moreover, the expression for $\bar{H_2}'(F)$ also implies there is at most one root of $\bar{H_2}$ for $F>4$. Therefore, the solution for $f$ satisfying the boundary conditions in (\refeq{eq_Kahler_not_Einstein}) exists with $f' \neq 0$ on $(0,T)$ if and only if we can apply the boundary conditions.

Once again, by the intermediate value theorem, if $\alpha \neq 0$, then there exists a root bigger than $4$ for $\bar{H_2}$ if and only if the sign of $\bar{H_2}''(0)$ is opposite that of $\displaystyle \lim_{F \to \infty} \bar{H_2}(F)$. This occurs exactly when $\alpha > -1$.

As $f$ is strictly increasing and the principal orbit is a $3$-sphere, adding a $\mathbb{CP}^1$ singular orbit such that the manifold closes up as an orbifold is equivalent to
\[L:=f(T)^2 = 4+\frac{2}{q} \; \text{ and } \; f'(T) = 0.\]
Therefore, the boundary conditions are satisfied for an non-Einstein soliton if and only if $\alpha$ is a non-zero root of the function:
\[h_2(\alpha) = L(L-4)\alpha^2-4(L-2)\alpha+8 - (8\alpha+8)e^{-\frac{L\alpha}{2}}.\]
We note that $h_2(0) = 0 = h_2'(0)=h_2''(0)$ and $h_2'''(\alpha) = L^2e^{-\frac{L\alpha}{2}}\left(L-6+L\alpha\right)$. Clearly $h_2'''(0) < 0$ for $L \in (0,6)$ with equality at the Einstein case $L=6$. Moreover, $h_2'''(\alpha) < 0$ on $(-\infty,0)$ if $L \in (0,6]$ and $h_2'''(\alpha) > 0$ on $(0,\infty)$ if $L \in [6,\infty)$. Therefore the Fubini--Study metric is unique for its boundary conditions. By the intermediate value theorem, there exists a root to $h_2$ on $(0,\infty)$ if $L \in (4,6)$ since $h_2'''(0)<0$ for all $L<6$ and $\displaystyle \lim_{\alpha \to \infty} h_2(\alpha) = \infty \; \text{ for all } \; L>4$. To get uniqueness, we note that the coefficients of the Taylor series
\[h_2(\alpha)e^{\frac{L\alpha}{2}} = \sum_{m=3}^\infty\frac{L^{m-1}}{m!2^{m-2}}(m-2)\big((L-4)m-L\big)\alpha^m\]
change sign at most once. 
\end{proof}

\begin{remark}
    One also gets non-zero roots of $h_2$ for all $L>6$ since $\displaystyle \lim_{\alpha \to -\infty} h_2(\alpha) = \infty$. By a similar Taylor series expansion argument for $\alpha<0$, it is seen that each non-zero root is unique for the choice of $L$. This allows one to take $L = 4 + {n}/{q}$ for any $n\in \mathbb{Z}_{>0}$, generating orbifold metrics on $G_n$. The space $G_n$ is constructed by adding a point to $\mathcal{O}(-n)$ at infinity. These metrics have orbifold singularities along the $\mathbb{CP}^1$ orbit and also an isolated orbifold point, generalizing the metrics found in \cite{FIK2003}. The only Einstein metrics constructed with this procedure are orbifold quotients of the Fubini--Study metric.  
\end{remark}

\begin{theorem}
    \label{thm_kahler_orbifolds_classification}
    Let $M$ be a $4$-dimensional simply-connected smooth manifold admitting a cohomogeneity one action by both $\mathrm{SU}(2) \subset \mathrm{U}(2)$ with a non-principal orbit. Every complete shrinking gradient K\"ahler Ricci soliton on $M$ invariant under these actions, potentially with orbifold singularities along non-principal orbits, is a rescaling of one of the following:
    \begin{enumerate}
        \item A metric from Lemma \refeq{lem_kahler_orbifolds} on $M_n$;
        \item A metric from Lemma \refeq{lem_kahler_orbifolds_noncompact} on $\mathcal{O}(-n)$;
        \item A metric from Lemma \refeq{lem_kahler_orbifolds_CP1} on $\mathbb{CP}^2$; or
        \item The shrinking Gaussian soliton on $\mathbb{C}^2$.
    \end{enumerate}
\end{theorem}

\begin{proof}
    We note that the $\mathrm{SU}(2)$ actions with different $\mathrm{U}(1)$ singular isotropies cannot produce $\mathrm{U}(2)$-invariant metrics as described in the proof of Lemma \refeq{lem_kahler_orbifolds_CP1}. Also as described above, a two fixed point action does not produce K\"ahler gradient Ricci solitons. Following Table \refeq{tab_compact}, all the compact examples are treated in Lemma \refeq{lem_kahler_orbifolds} and Lemma \refeq{lem_kahler_orbifolds_CP1}. 

    In the non-compact case, following Table \refeq{tab_non_compact}, by Lemma \refeq{lem_kahler_orbifolds_noncompact} it only remains to consider the $\mathrm{U}(2)$-action on $\mathbb{C}^2$. This again gives the initial conditions in (\refeq{eq_Kahler_orbifold_closingvan}) and by the proof of Lemma \refeq{lem_kahler_orbifolds_CP1}, the metric is over a non-compact space exactly when $\alpha \leq -1$. If $\alpha <-1$, then
    \[\lim_{f \to \infty} f^4e^{\frac{\alpha}{2}f^2}(f')^2 = -8\alpha - 8>0,\]
    and the metric is incomplete. Finally if $\alpha = -1$, then $(f')^2 = 1$ and the metric is the Gaussian shrinker is given by $f(t)=t$.
\end{proof}

Theorem \refeq{thm_Kahlerclassification} is therefore proven as $\mathbb{R}\times \mathbb{S}^3$ does not admit a complete K\"ahler metric (since they are connected at infinity). 

\begin{remark} 
    Since $\mathrm{U}(2)/\{\pm I\}$ is isomorphic to $\mathrm{SO}(3) \times \mathrm{SO}(2)$, a classification of all the $\mathrm{U}(2)$-invariant gradient K\"ahler Ricci soliton orbifolds also includes the $\mathrm{SO}(3) \times \mathrm{SO}(2)$ actions. In particular, it would remain to consider the three $\mathrm{SO}(3) \times \mathrm{SO}(2)$ actions on $\mathbb{CP}^1 \times \mathbb{CP}^1$, $\mathbb{S}^4$ and $\mathbb{C}\times \mathbb{CP}^1$. The remaining smooth examples are the direct Riemannian sum of Fubini--Study metrics on $\mathbb{CP}^1 \times \mathbb{CP}^1$ and the shrinking cylinder $\mathbb{C}\times\mathbb{CP}^1$. However, $\mathrm{SO}(3) \times \mathrm{SO}(2)$ actions are is not within the scope of this work --- we focus on $\mathrm{SU}(2)$ actions and construct the non-smooth metrics in Theorem \refeq{thm_Kahlerclassification} as an aside when generating Figure \refeq{fig_compact_solitons_orbifolds}.
\end{remark}

\printbibliography

\pagebreak

\section{Appendix: additional figures for compact shrinkers}

We give the $2$-dimensional slices of the $3$-dimensional heatmap produced by $SOL(0.001,-20,\cdot)$ that best highlight the known compact gradient solitons. We list these in the same order as in Examples \refeq{ex_round_metric_1}-\refeq{ex_sum_of_round}.

\begin{figure}[h!]
    \caption{Ricci solitons on $\mathbb{S}^4$, $\mathbb{CP}^2$ and $M_n$}
    \includegraphics[scale = 0.38]{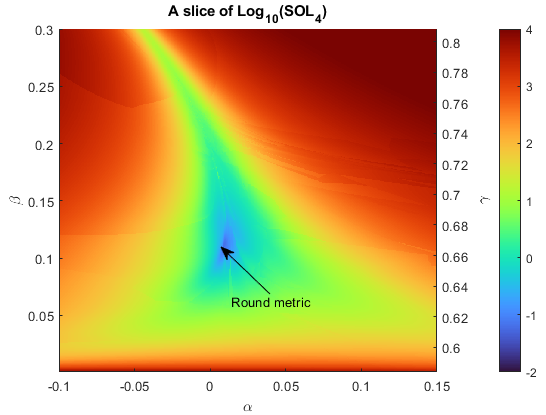}
    \includegraphics[scale = 0.51]{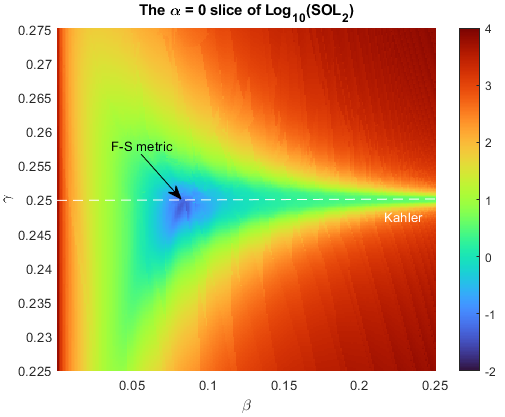}\newline
    \noindent \includegraphics[scale = 0.54]{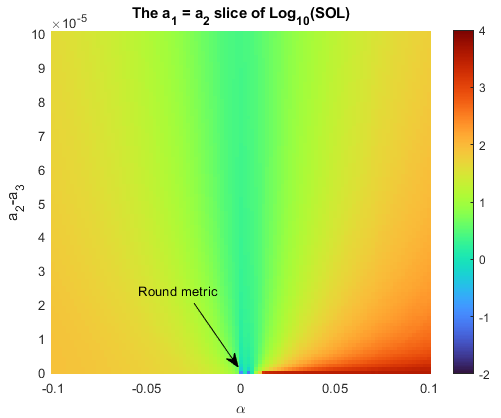}
    \includegraphics[scale = 0.54]{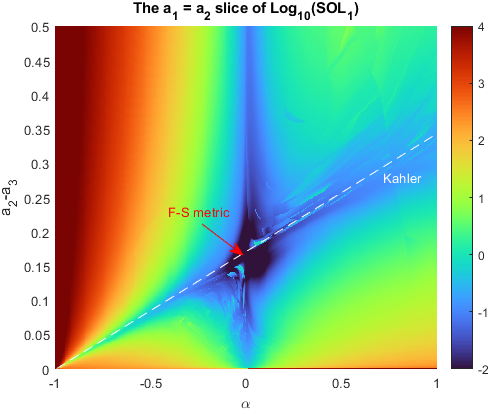} \newline
    \noindent \includegraphics[scale = 0.52]{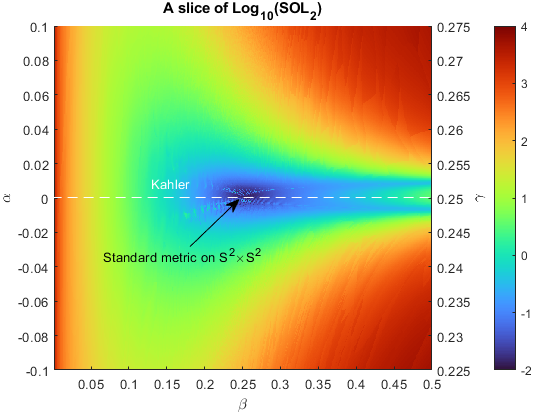}
    \includegraphics[scale = 0.52]{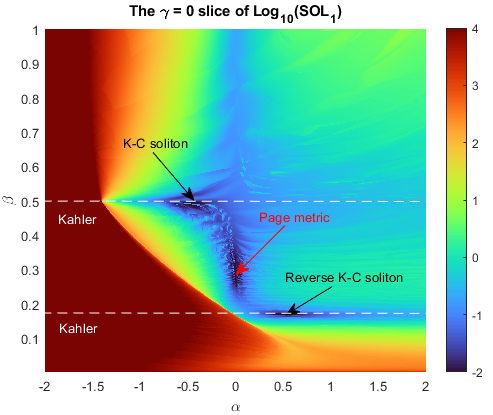}
    \label{fig_1_1}
\end{figure}

\end{document}